\documentclass[3p, 12pt, sort&compress]{elsarticle}
\usepackage{soul}
\usepackage{float} % To re-style the 'figure' float
\usepackage{graphicx}
\usepackage{amsmath}
\usepackage{mathtools}
\usepackage{amssymb}
\usepackage{hyperref}
\usepackage{graphics}
\usepackage{amsthm}
\usepackage{latexsym,amsfonts}
\usepackage{booktabs}
\usepackage{multirow}
\usepackage{xcolor}
\usepackage{subcaption}
\usepackage{cleveref}

\newtheorem{proposition}{Proposition}
\newtheorem{rmk}{Remark}
\newproof{pot}{Proof of Theorem \ref{thm2}}

\newcommand\myeq{\mathrel{\stackrel{\makebox[0pt]{\mbox{\normalfont\scriptsize def}}}{=}}}

\newcounter{lnote}

\newcommand{\red}[1]{#1}

\newcommand{\crossed}[1]{}
\newcommand{\footcomm}[1]{}

\newcommand{\I}[1]{\mathcal{I}_{_{\hbox{\tiny{#1}}}}}
\newcommand{\G}[1]{\mathcal{G}_{_{\hbox{\tiny{#1}}}}}

\newcommand{\C}[1]{C_{_{\hbox{\tiny{#1}}}}}

\newcommand{\SG}[1]{$\text{SG}_{_{\hbox{\tiny{#1}}}}$}
\newcommand{\SGD}[1]{$\text{SGD}_{_{\hbox{\tiny{#1}}}}$}
\newcommand{\ASGD}[1]{$\text{ASGD}_{_{\hbox{\tiny{#1}}}}$}
\newcommand{\rASGD}[1]{$\text{rASGD}_{_{\hbox{\tiny{#1}}}}$}

\newcommand{\bs}[1]{\boldsymbol{#1}}

\begin{document}

\begin{frontmatter}

%\title{Nesterov Accelerated Stochastic Gradient for Expected Information Gain in Optimal Bayesian Design using Laplace approximation and Importance Sampling}

\title{Nesterov-Aided Stochastic Gradient Methods Using Laplace Approximation for Bayesian Design Optimization}

% Group authors per affiliation:
\author[ufsc]{André Gustavo  Carlon\corref{mycorrespondingauthor}}
\cortext[mycorrespondingauthor]{Corresponding author.\\
\quad E-mail addresses: agcarlon@gmail.com (AG Carlon),
mansourben2002@yahoo.fr (BM Dia),
espath@gmail.com (LFR Espath),
rafaelholdorf@gmail.com (RH Lopez),
raul.tempone@kaust.edu.sa (R Tempone)}
\address[ufsc]{Department of Civil Engineering, Federal University of Santa Catarina (UFSC),  Rua Jo\~ao Pio Duarte da Silva, Florian\'opolis, SC, 88040-970, Brazil}

\author[kfupm]{Ben Mansour Dia}
\address[kfupm]{College of Petroleum Engineering and Geosciences, King Fahd University of Petroleum and Minerals (KFUPM),  Dhahran 31261, Saudi Arabia}

\author[kaust]{Luis Espath}

\author[ufsc]{Rafael Holdorf Lopez}

\author[kaust]{Ra\'ul Tempone}

\address[kaust]{Computer, Electrical and Mathematical Science and Engineering Division (CEMSE), King Abdullah University of Science and Technology (KAUST), Thuwal, 23955-6900, Saudi Arabia}

\begin{abstract}
Finding the best setup for experiments is the primary concern for Optimal Experimental Design (OED).
Here, we focus on the Bayesian experimental design problem of finding the setup that maximizes the Shannon expected information gain.
We use the stochastic gradient descent and its accelerated counterpart, which employs Nesterov's method, to solve the optimization problem in OED.
We adapt a restart technique, originally proposed for the acceleration in deterministic optimization, to improve stochastic optimization methods.
We combine these optimization methods with three estimators of the objective function: the double-loop Monte Carlo estimator (DLMC), the Monte Carlo estimator using the Laplace approximation for the posterior distribution (MCLA) and the double-loop Monte Carlo estimator with Laplace-based importance sampling (DLMCIS).
Using stochastic gradient methods and Laplace-based estimators together allows us to use expensive and complex models, such as those that require solving partial differential equations (PDEs).
From a theoretical viewpoint, we derive an explicit formula to compute the gradient estimator of the Monte Carlo methods, including MCLA and DLMCIS.
From a computational standpoint, we study four examples: three based on analytical functions and one using the finite element method.
The last example is an electrical impedance tomography experiment based on the complete electrode model.
In these examples, the accelerated stochastic gradient descent method using MCLA converges to local maxima with up to five orders of magnitude fewer model evaluations than gradient descent with DLMC.
\end{abstract}


\end{frontmatter}

% \linenumbers

\section{Introduction}

% Motivation
Performing experiments can be expensive and time consuming.
Moreover, the efficiency of an experiment depends on its setup.
It is thus advantageous to find, a priori, the experimental setup that maximizes the information to be collected.
Such an approach is named optimal experimental design (OED) \cite{chaloner1995bayesian}.
Due to the inherently probabilistic nature of the collected quantities, OED is an uncertainty quantification task, particularly a stochastic optimization problem.

 In the Bayesian setting, where the inference of the  parameter of interest consists in updating prior knowledge with information carried by the data, the design optimization aims to search the experimental setup that gives the best efficiency.
 To measure the efficiency of an experiment, we use the Shannon expected information gain, which is based on the Kullback--Leibler divergence of the posterior probability density function (pdf) with respect to the prior pdf of the quantities of interest \cite{chaloner1995bayesian}.

To estimate the Shannon expected information gain, we must compute a double integral over both the space of observed data and the space of the parameter of interest.
The optimization process might require several estimations of the Shannon expected information gain, which can be computationally demanding even for inexpensive experiment models.
Hence, challenges in design optimization include, among others, the approximation of the Shannon expected information gain and the estimation of the gradient.
Our main goal is to evaluate the ability of different numerical methods to efficiently perform both the optimization and the uncertainty quantification so that experiments with expensive models can be optimized with a reasonable amount of time and effort.

% Literature review
To address the OED problem, Ryan \cite{ryan2003estimating} develops an expected information gain estimator based on Monte Carlo sampling (MC) that requires the evaluation of two nested MC samplings; thus, we refer to this estimator as double-loop Monte Carlo (DLMC).
Huan \cite{huan2010accelerated} uses the DLMC estimator in the design of a combustion reaction experiment with a non-linear forward model.
To alleviate the computational burden, Huan \cite{huan2010accelerated} estimates the expected information gain over a surrogate model constructed with the Wiener chaos polynomial expansion, where the expected information gain is evaluated in a grid of design candidates, choosing the best candidate as the optimum approximation.
To improve Huan's \cite{huan2010accelerated} optimization procedure, Huan and Marzouk \cite{huan2013simulation} use a variation of the Kiefer--Wolfowitz algorithm proposed by Spall \cite{spall1988stochastic}, which reduces the number of objective function evaluations needed for the finite differences estimates of the gradient to two.
Long et al. \cite{long2013fast} use a Laplace approximation, thus avoiding the evaluation of one of the two nested MC samplings of DLMC.
The resulting expected information gain estimator is referred to as the Monte Carlo with Laplace approximation (MCLA) estimator.
Huan and Marzouk \cite{huan2014gradient} estimate the gradient of the expected information gain for OED problems using mini-batch samples of various sizes, all small compared to the main batch, and use this estimation to perform a steepest descent search.
To assess the efficiency of their method, they compare the convergence cost with a quasi-Newton approach using sample average approximation.
Beck et al. \cite{Beck2017fast} propose an importance sampling approach for DLMC that uses Laplace approximations to draw more informative samples, reducing the cost of DLMC without adding the bias of MCLA.
We refer to DLMC with importance sampling as double-loop Monte Carlo  with Laplace-based importance sampling (DLMCIS).

Since we opt to use gradient-based optimization methods, the estimation of the gradient of the expected information gain plays a crucial role in our framework.
To alleviate the computational burden of computing an accurate full gradient on every iteration, we use stochastic gradient methods, which are a class of optimization methods that use noisy estimates of the true gradient \cite{nemirovski2005efficient}.
Hence, precise gradient estimates can be substituted by inexpensive alternatives, reducing the cost per optimization iteration.
We refer to the gradient estimators used in stochastic optimization methods as stochastic gradients.
To evaluate the gradient of the expected information gain in the stochastic gradient sense, we use three estimation strategies: DLMC, MCLA \cite{long2013fast} and DLMCIS \cite{Beck2017fast}.
The stochastic gradient of the DLMC estimator only has one MC loop, and thus is referred to as the stochastic gradient Monte Carlo (\SG{MC}) estimator.
The MCLA estimator uses an approximation of the posterior distribution as a Gaussian pdf to calculate the Kullback--Leibler divergence of the posterior pdf with respect to the prior pdf, avoiding the evaluation of one of the two nested integrals that appear in DLMC.
Consequently, the number of model evaluations is significantly reduced.
The stochastic gradient of MCLA is the stochastic gradient with Laplace approximation (\SG{LA}) estimator, a gradient estimator that does not use MC sampling.
Alternatively, the DLMCIS estimator dramatically reduces the number of inner samples compared to the DLMC estimator, without introducing the bias of the Laplace approximation.
We use the same importance sampling scheme in its respective stochastic gradient estimator, resulting in the stochastic gradient Monte Carlo estimator with Laplace-based importance sampling (\SG{MCIS}).
The expected information gain estimators are discussed in Section \ref{sec:EIGestimators} and their gradients in Section \ref{sec:SGestimators}.

% The search for the optimal experimental setup is a stochastic optimization problem, which can be computationally expensive.
% Here, we will set up the search for the optimal experimental design as a stochastic optimization problem.
To solve the OED problem, we employ three optimization methods: stochastic gradient descent (SGD), SGD with Nesterov's acceleration (ASGD), and ASGD with a restart technique (rASGD).
The SGD method is an application of the stochastic approximation proposed by Robbins and Monro \cite{robbins1951stochastic} that is used in the optimization of expected values of functions.
Therefore, SGD is well suited for optimization in the presence of uncertainties.
Although SGD converges to the optimum using an inexpensive estimate of the gradient, its convergence is slow.
To improve the convergence while maintaining a low-cost gradient estimate, we use Nesterov's acceleration \cite{nesterov1983method} coupled with a restart technique proposed by O'Donoghue and Cand\`es \cite{o2015adaptive}.
Nitanda \cite{nitanda2016accelerated} employs this restart technique with a variance reduction technique and mini-batches to multiclass logistic regression problems.
The use of variance reduction, combined with mini-batches, makes the estimation of the gradient nearly deterministic, which is different to our approach.
Moreover, Nitanda \cite{nitanda2016accelerated} uses the rASGD for regression problems where the objective function is a finite sum of functions.
Here, we combine the restart technique for the acceleration, originally proposed by O'Donoghue and Cand\`es \cite{o2015adaptive} for deterministic optimization, with ASGD.
The SGD method, Nesterov's acceleration, and the restart technique are presented in Section \ref{sec:Opt}.

We assess the performance of the presented methods by solving four stochastic optimization problems, three of which are OED problems.
The first example, presented in Section \ref{sec:ex1}, is not an OED problem, but a stochastic optimization problem used to compare the optimization methods.
In the second example, shown in Section \ref{sec:ex2}, we use a quadratic forward model to test the efficiency of MCLA and DLMCIS, as well as their coupling with the optimization methods.
In the third example, shown in Section \ref{sec:ex3}, we search for the optimal positioning of a strain gauge on a beam in order to maximize the expected information gain with respect to some mechanical properties of the material.
Finally, in the fourth example (Section \ref{sec:ex4}), we optimize the currents applied to electrodes during an electrical impedance tomography (EIT) experiment in order to maximize the expected information gain regarding the orientation angles of plies in a composite laminate material.
The model for this problem is based on partial differential equations (PDEs) and is solved using the finite element method (FEM).

The main contribution of this work, from a theoretical standpoint, lies in the derivation of the estimators of the gradients of the expected information gain and in their adaptation to be used in stochastic gradient methods.
Moreover, from a numerical standpoint, we successfully tailor recent ideas of Nesterov-based optimizers with the restart technique proposed for deterministic optimization by O'Donoghue and Cand\`es \cite{o2015adaptive} to the stochastic gradient framework.
Finally, we provide numerical engineering examples to highlight the performances of our methods.

The following notation is used throughout the paper: $\text{det}(\cdot)$ is the determinant;
the tensor notation is adopted, where ($\cdot$) is the single contraction and ($\colon$) is the double contraction;
$||\cdot||$ is the $L^2$-norm; $|| \bs{a} ||_{\bs{\Sigma}} = \bs{a} \cdot \bs{\Sigma} \cdot \bs{a}$ is the $\bs{\Sigma}$-norm of $\bs{a}$; $\mathbb{E}[\cdot]$ is the expectation operator; $\mathbb{V}[\cdot]$ is the variance operator; and $\text{dim}(\cdot)$ is the dimension.

\section{Bayesian experimental design}

\subsection{Bayesian inference}

The experimental data are represented by $\bs{y}_i \in \mathbb{R}^{r}$, a vector of $r$ observations that are given by the experiment model response with an additive error, as
\begin{equation}
\bs{y}_i(\bs{\xi}) = \bs{g}(\bs{\xi},\bs{\theta}_t) + \bs{\epsilon}_i, \qquad i=1,\dots,N_e,
\label{eq:datamodel}
\end{equation}
where $\bs{g}(\bs{\xi},\bs{\theta}_t) \in \mathbb{R}^{r}$ are the deterministic model responses, $\bs{\theta}_t \in \mathbb{R}^{d}$ is the parameter vector to be recovered, $\bs{\xi} \in \Xi$ is the design parameter vector, and $N_e$ is the number of repetitive experiments.
Here, $\Xi$ is the experimental design space.
We assume that the measurement noise vectors $\bs{\epsilon}_i$
 % $\bs{\epsilon}_i\sim\mathcal{N}(\bs{0},\bs{\Sigma_{\epsilon}})$,
% $\bs{\epsilon}_i: \Omega^d \mapsto \mathcal{E} \subset \mathbb{R}^{d}$
are independent and identically distributed (i.i.d.) Gaussian-distributed with zero-mean and covariance matrix $\bs{\Sigma_{\epsilon}}$.
Moreover, the noise vectors $\bs{\epsilon}_i$ are also independent of both $\bs{\theta}$ and $\bs{\xi}$.
We characterize the unknown parameter $\bs{\theta}_t$ as a random variable vector $\bs{\theta}: \Omega^d \mapsto \Theta \subset \mathbb{R}^{d}$ with a prior distribution $\pi(\bs{\theta})$, where $\Omega$ is the set of random events.
The set of observed data is $\bs{Y}=\{\bs{y}_i\}^{N_e}_{i=1}$, the functional $\bs{g}$ is assumed to be twice differentiable with respect to $\bs{\theta}$ and differentiable with respect to $\bs{\xi}$ and the true value of $\bs{\theta}_t$ is assumed to be unknown.
% In lieu of the actual vector $\bs{\theta}_t$, we consider a vector of random variables $\bs{\theta}: \Theta$ $\mapsto$ $\mathbb{R}^{d}$ with prior distribution $\pi(\bs{\theta})$, where $\Theta$ is the prior space.
%The functional $\bs{g}$ is assumed to be twice differentiable with respect to $\bs{\theta}$ and differentiable with respect to $\bs{\xi}$.

Once the data is collected, the prior pdf is updated through a likelihood of events, thus producing the posterior pdf.
The fundamental idea of the Bayesian framework for OED consists of finding the experimental setup that produces data that, on average, maximize the knowledge about the quantities of interest, i.e., that maximize the Kullback-Leibler divergence of the posterior pdf with respect to the prior pdf.
This machinery is built on Bayes' formula, i.e.,
\begin{equation}
\pi(\bs{\theta}|\bs{Y},\bs{\xi}) = \frac{p(\bs{Y}|\bs{\theta},\bs{\xi}) \pi(\bs{\theta})}{p(\bs{Y}\vert \bs{\xi})} ,
\label{eq:bayes}
\end{equation}
where $\pi(\bs{\theta})$ is the prior pdf (the initial belief about the parameter to be inferred), $\pi(\bs{\theta}|\bs{Y}, \bs{\xi})$ is the posterior distribution (the updated pdf of the random variable $\bs{\theta}$, given the observation $\bs{Y}$), $p(\bs{Y}|\bs{\theta}, \bs{\xi})$ is the likelihood (the information provided by the observation $\bs{Y}$), and  $p(\bs{Y}|\bs{\xi})$ is the evidence (the pdf of the marginal distribution of the observation $\bs{Y}$, describing the data distribution).
Considering the data model \eqref{eq:datamodel} and the Gaussian assumption for the noise, the likelihood has the form
\begin{equation}
p(\bs{Y} \vert \bs{\theta}, \bs{\xi}) = \text{det}\left(2\pi\bs{\Sigma_\epsilon} \right)^{-\frac{N_e}{2}} \exp \left( -\frac{1}{2}  \sum_{i=1}^{N_e} \left\| \bs{y}_{i}(\bs{\xi}) - \bs{g}(\bs{\xi}, \bs{\theta})\right\|^2_{\bs{\Sigma_\epsilon}^{-1}} \right).
\end{equation}
% where the norm is $ \|\bs{x}\|^{2}_{\bs{\Sigma}_{\bs{\epsilon}}^{-1}} = \bs{x} \cdot \bs{\Sigma}_{\bs{\epsilon}}^{-1} \cdot \bs{x}$ for a vector $\bs{x}$ and covariance matrix $\bs{\Sigma_\epsilon}$.

\subsection{Expected information gain}

To evaluate the quality of each experiment, we measure the Kullback--Leibler divergence ($D_{kl})$ of the posterior pdf with respect to the prior pdf:
\begin{equation}\label{eq:dkl}
D_{kl}\left(\bs{\xi},\bs{Y}\right)= \int_{\Theta} \log{\left(\frac {\pi(\bs{\theta}|\bs{Y}, \bs{\xi})}{\pi(\bs{\theta})}\right)} \pi(\bs{\theta}|\bs{Y}, \bs{\xi})\text{d}\bs{\theta}.
\end{equation}
The expected information gain, proposed by Shannon \cite{shannon1948mathematical}, is the expectation of the $D_{kl}$ \eqref{eq:dkl} with respect to the distribution of the data $p(\bs{Y}|\bs{\xi})$.
By accounting for \eqref{eq:bayes}, we obtain the expected information gain as
\begin{align}\label{eq:infgain}
I(\bs{\xi}) =& \int_{\mathcal{Y}}{\int_{\Theta}{\log \left( \frac{\pi(\bs{\theta} \vert \bs{Y}, \bs{\xi})}{\pi(\bs{\theta})} \right) \pi(\bs{\theta} \vert \bs{Y}, \bs{\xi}) \text{d}\bs{\theta}} p(\bs{Y} \vert \bs{\xi}) \text{d}\bs{Y}} \nonumber\\
=& \int_{\Theta} \int_{\mathcal{Y}} \log \left( \frac{p(\bs{Y} \vert \bs{\theta}, \bs{\xi})}{p(\bs{Y} \vert \bs{\xi})} \right)  p(\bs{Y} \vert \bs{\theta}, \bs{\xi}) \text{d}\bs{Y} \pi(\bs{\theta}) \text{d}\bs{\theta}.
\end{align}
Since the evidence $p(\bs{Y}|\bs{\xi})$ is not known, we substitute it by marginalization of the likelihood with respect to the prior $\pi(\bs{\theta}^*)$, i.e.,
\begin{equation}
p(\bs{Y}|\bs{\xi})= \int_{\Theta} p(\bs{Y}|\bs{\theta}^*,\bs{\xi}) \pi(\bs{\theta}^*) \text{d} \bs{\theta^*}.
\end{equation}
Bear in mind that $\bs{\theta}^*$ and $\bs{\theta}$ are independent and that $\bs{Y}$ depends on $\bs{\theta}$, $\bs{\xi}$, and $\bs{\epsilon}$, i.e., the parameter $\bs{\theta}$ used to generate $\bs{Y}$ is different from $\bs{\theta}^*$ in the integral within the logarithm.
Thus, we rewrite the expected information gain as
\begin{equation}
  I(\bs{\xi}) = \int_{\Theta} \int_{\mathcal{Y}} \log \left( \frac{p(\bs{Y} \vert \bs{\theta}, \bs{\xi})}{\int_{\Theta} p(\bs{Y}|\bs{\theta}^*,\bs{\xi}) \pi(\bs{\theta}^*) \text{d} \bs{\theta^*}} \right)  p(\bs{Y} \vert \bs{\theta}, \bs{\xi}) \text{d}\bs{Y} \pi(\bs{\theta}) \text{d}\bs{\theta},
\end{equation}
where the likelihood pdf is
\begin{equation}\label{eq:likelihood}
p(\bs{Y}(\bs{\xi},\bs{\theta},\bs{\epsilon}) \vert \bs{\theta}^*, \bs{\xi}) = \text{det}( 2 \pi \bs{\Sigma}_{\bs{\epsilon}}) ^{-\frac{N_e}{2}} {\exp \left(-\frac{1}{2} \sum_{i=1}^{N_e} \| \bs{r}_i(\bs{\xi},\bs{\theta},\bs{\theta}^*,\bs{\epsilon}) \|^2 _{\bs{\Sigma_\epsilon}^{-1}} \right)},
% \prod_{i=1}^{N_e}{\exp \left(-\frac{1}{2} \bs{r}_i^T(\bs{\xi},\bs{\theta},\bs{\theta}^*,\bs{\epsilon}) \bs{\Sigma_\epsilon}^{-1} \bs{r}_i(\bs{\xi},\bs{\theta},\bs{\theta}^*,\bs{\epsilon}) \right)},
\end{equation}
and $\bs{r}_i(\bs{\xi},\bs{\theta},\bs{\theta}^*, \bs{\epsilon}) = \bs{g}(\bs{\xi},\bs{\theta}) + \bs{\epsilon} - \bs{g}(\bs{\xi},\bs{\theta}^*)$ is the residual of the $i$-th experimental data.

\begin{rmk}[Expected information gain with Laplace approximation] \label{rmk1:eigla}
The Laplace estimator for $D_{kl}$ is proposed by Long et al. \cite{long2013fast} and relies on approximating the logarithm of the posterior pdf by a second-order Taylor expansion at the maximum posterior estimate.
As a consequence, the approximated posterior is Gaussian-distributed.
The Gaussian approximation of the posterior pdf can be written as
\begin{equation}\label{eq:gaussian.posterior}
\pi(\bs{\theta} \vert \bs{Y}, \bs{\xi}) \approx \pi_{_{\hbox{\tiny{LA}}}}(\bs{\theta} \vert \bs{Y}, \bs{\xi})
\myeq \text{\emph{det}}(2 \pi \bs{\Sigma}(\bs{\xi}, \bs{\hat{\theta}})) ^{-\frac{1}{2}} \exp\left(-\frac{1}{2} \|\bs{\theta} - \bs{\hat{\theta}}(\bs{\xi})\|^{2}_{\bs{{\Sigma}}^{-1}(\bs{\xi}, \bs{\hat{\theta}})}\right),
\end{equation}
where $\bs{\hat{\theta}}$ is the maximum a posteriori (MAP) estimate, i.e.,
\begin{equation}\label{eq:thetahat}
\bs{\hat{\theta}}(\bs{\xi}) \myeq \underset{\bs{\theta} \in \Theta}{\arg\min} \left[ \frac{1}{2} \sum_{i=1}^{N_e} \left\| \bs{y}_i - \bs{g}(\bs{\xi}, \bs{\theta}) \right\|^2_{\bs{\Sigma_{\epsilon}}^{-1}} - \log (\pi(\bs{\theta})) \right], \qquad \text{and}
\end{equation}
\begin{equation}\label{eq:sigmahat}
\bs{\Sigma}^{-1}(\bs{\xi}, \bs{\hat{\theta}}) = N_e \nabla_{\bs{\theta}} \bs{g} (\bs{\xi}, \bs{\hat{\theta}}) \cdot \bs{\Sigma_\epsilon}^{-1} \cdot \nabla_{\bs{\theta}} \bs{g} (\bs{\xi}, \bs{\hat{\theta}}) -  \nabla_{\bs{\theta}}  \nabla_{\bs{\theta}}  \log (\pi(\bs{\hat{\theta}})) +  \mathcal{O}_\mathbb{P}\left(\sqrt{N_e}\right)
\end{equation}
is the Hessian matrix of the negative logarithm of the posterior pdf evaluated at $\bs{\hat{\theta}}$.
Moreover, Long et al. \cite{long2013fast} show that
\begin{equation}
  \bs{\hat{\theta}} = \bs{\theta}_t + \mathcal{O}_{\mathbb{P}} \left( \frac{1}{\sqrt{N_e}} \right).
\end{equation}

Finally, the Gaussian approximation \eqref{eq:gaussian.posterior} with $\hat{\bs{\theta}}$ and $\bs{\Sigma}$ given by \eqref{eq:thetahat} and \eqref{eq:sigmahat}, respectively, leads to an analytical expression of the $D_{kl}$.
Using the approximation $\bs{\hat{\theta}} \approx \bs{\theta}_t$ subsequently yields the approximate expected information gain as
\begin{equation}\label{eq:I.Lap}
I(\bs{\xi}) =
\int_{\Theta}{\left[-\frac{1}{2} \log(\text{\emph{det}}(2 \pi \bs{\Sigma} (\bs{\xi}, \bs{\theta}_t))) - \frac{d}{2} - \log (\pi(\bs{\theta}_t)) \right] \pi(\bs{\theta}_t) \text{d}\bs{\theta}_t} + \mathcal{O} \left(\frac{1}{N_e}\right).
\end{equation}
\qed
\end{rmk}

\subsection{Maximization of the expected information gain}

We want to find the optimal setup $\bs{\xi}^*$ in a Bayesian framework that, on average, provides the most informative data.
We formulate the problem of finding $\bs{\xi}^*$ as the optimization problem
\begin{equation}
 \bs{\xi}^* = \underset{\bs{\xi} ~ \in ~  \Xi}{\text{arg max}}(I(\bs{\xi})).
\label{eq:infopt}
\end{equation}
With the assumption that the local search methods converge to $\bs{\xi}^*$, gradient-based methods are suited to solve the optimization problem given by \eqref{eq:infopt}.

We write the gradient of $I$ in \eqref{eq:infgain} with respect to the design variable $\bs{\xi}$ as
\begin{equation}
 \nabla_{\bs{\xi}} I(\bs{\xi}) = \nabla_{\bs{\xi}} \int_{\Theta} \int_{\mathcal{Y}} \log \left( \frac{p(\bs{Y} \vert \bs{\theta}, \bs{\xi})}{p(\bs{Y}\vert \bs{\xi})} \right)  p(\bs{Y} \vert \bs{\theta}, \bs{\xi}) \text{d}\bs{Y} \pi(\bs{\theta}) \text{d}\bs{\theta}.
 \label{eq:fullgradI}
\end{equation}
We also denote the quantity defined in \eqref{eq:fullgradI} as the \textit{full gradient} of the expected information gain.

\begin{proposition}\label{propo:1}
Assuming that $\bs{Y}=\{\bs{y}_i(\bs{\xi}, \bs{\epsilon}_i)\}^{N_e}_{i=1}$ has the particular form \eqref{eq:datamodel}, \eqref{eq:fullgradI} becomes
\begin{equation}\label{eq:grad.exp.inf.gain}
\nabla_{\bs{\xi}} I(\bs{\xi}) =\int_{\Theta} \int_{\mathcal{Y}}  \nabla_{\bs{\xi}}  \log \left( \frac{p(\bs{Y} \vert \bs{\theta}, \bs{\xi})}{p(\bs{Y}\vert \bs{\xi})} \right)  p(\bs{Y} \vert \bs{\theta}, \bs{\xi}) \text{d}\bs{Y} \pi(\bs{\theta})  \text{d} \bs{\theta}.
\end{equation}
\end{proposition}
\begin{proof}
We assume that $\bs{Y}$ depends on $\bs{\xi}$; thus, we need to apply a change of variables before applying Leibniz's rule,
\begin{equation}
  \text{d}\bs{Y} = \text{det} (\nabla_{\bs{\epsilon}} \bs{Y}(\bs{\xi},\bs{\theta},\bs{\epsilon})) \text{d}\bs{\epsilon}.
\end{equation}
With $\mathcal{E}$ being the sample space of $\bs{\epsilon}$, from \eqref{eq:fullgradI},
\begin{equation} \label{eq:A2}
  \begin{split}
  \nabla_{\bs{\xi}} I(\bs{\xi}) &= \nabla_{\bs{\xi}} \int_{\Theta} \int_{\mathcal{E}}  \log \left( \frac{p(\bs{Y} \vert \bs{\theta}, \bs{\xi})}{p(\bs{Y}\vert \bs{\xi})} \right)  p(\bs{Y} \vert \bs{\theta}, \bs{\xi}) \det (\nabla_{\bs{\epsilon}}\bs{Y}) \text{d}\bs{\epsilon} \pi(\bs{\theta})  \text{d} \bs{\theta}\\
  &= \int_{\Theta} \int_{\mathcal{E}} \nabla_{\bs{\xi}} \log \left( \frac{p(\bs{Y} \vert \bs{\theta}, \bs{\xi})}{p(\bs{Y}\vert \bs{\xi})} \right)  p(\bs{Y} \vert \bs{\theta}, \bs{\xi}) \det (\nabla_{\bs{\epsilon}}\bs{Y}) \text{d}\bs{\epsilon} \pi(\bs{\theta})  \text{d} \bs{\theta}\\
  & \quad + \int_{\Theta} \int_{\mathcal{E}} \log \left( \frac{p(\bs{Y} \vert \bs{\theta}, \bs{\xi})}{p(\bs{Y}\vert \bs{\xi})} \right)  \nabla_{\bs{\xi}}  p(\bs{Y} \vert \bs{\theta}, \bs{\xi}) \det (\nabla_{\bs{\epsilon}}\bs{Y}) \text{d}\bs{\epsilon} \pi(\bs{\theta})  \text{d} \bs{\theta}\\
  & \quad + \int_{\Theta} \int_{\mathcal{E}} \log \left( \frac{p(\bs{Y} \vert \bs{\theta}, \bs{\xi})}{p(\bs{Y}\vert \bs{\xi})} \right)  p(\bs{Y} \vert \bs{\theta}, \bs{\xi}) \nabla_{\bs{\xi}} \det (\nabla_{\bs{\epsilon}}\bs{Y}) \text{d}\bs{\epsilon} \pi(\bs{\theta})  \text{d} \bs{\theta}.
\end{split}
\end{equation}

Next, we prove that, for experiments with data modeled as \eqref{eq:datamodel}, the two last integrals on the r.h.s. of \eqref{eq:A2} vanish.
For the particular data $\bs{Y}(\bs{\xi}, \bs{\theta}, \bs{\epsilon})$, the likelihood $p(\bs{Y}(\bs{\xi}, \bs{\theta}, \bs{\epsilon}) \vert \bs{\theta}, \bs{\xi})$ can be obtained from \eqref{eq:likelihood} as
\begin{equation} \label{eq:A3}
p(\bs{Y}(\bs{\xi},\bs{\theta},\bs{\epsilon}) \vert \bs{\theta}, \bs{\xi}) = \text{det}(2 \pi \bs{\Sigma_\epsilon}^{-1}) ^{-\frac{N_e}{2}}\exp \left(-\frac{1}{2} \sum_{i=1}^{N_e} \| \bs{\epsilon}_i \|^2 _{\bs{\Sigma_\epsilon}^{-1}} \right);
\end{equation}
thus, it neither depends on the model nor on the design parameters $\bs{\xi}$.
Consequently,
\begin{equation}
  \begin{split}
    \nabla_{\bs{\xi}} p(\bs{Y}(\bs{\xi}, \bs{\theta}, \bs{\epsilon}) \vert \bs{\theta}, \bs{\xi}) = \bs{0};
  \end{split}
\end{equation}
hence,
\begin{equation} \label{eq:A6}
\int_{\Theta} \int_{\mathcal{E}} \log \left( \frac{p(\bs{Y} \vert \bs{\theta}, \bs{\xi})}{p(\bs{Y}\vert \bs{\xi})} \right)  \nabla_{\bs{\xi}}  p(\bs{Y} \vert \bs{\theta}, \bs{\xi}) \det (\nabla_{\bs{\epsilon}}\bs{Y}) \text{d}\bs{\epsilon} \pi(\bs{\theta})  \text{d} \bs{\theta} = \bs{0}.
\end{equation}

Regarding the last integral on the r.h.s. of \eqref{eq:A2}, from the experiment model we adopt on \eqref{eq:datamodel}, we get $\nabla_{\bs{\epsilon}} \bs{Y} = \bs{I}$; thus
\begin{equation} \label{eq:A7}
  \begin{split}
    \nabla_{\bs{\xi}} \text{det} (\nabla_{\bs{\epsilon}} \bs{Y}(\bs{\xi},\bs{\theta},\bs{\epsilon})) \text{d}\bs{\epsilon} &= \nabla_{\bs{\xi}} \text{det} (\bs{I}) \text{d}\bs{\epsilon} \\
    &= \bs{0}.
  \end{split}
\end{equation}
Consequently,
\begin{equation} \label{eq:A9}
    \int_{\Theta} \int_{\mathcal{E}} \log \left( \frac{p(\bs{Y} \vert \bs{\theta}, \bs{\xi})}{p(\bs{Y}\vert \bs{\xi})} \right)  p(\bs{Y} \vert \bs{\theta}, \bs{\xi}) \nabla_{\bs{\xi}} \text{det} (\nabla_{\bs{\epsilon}} \bs{Y}) \text{d}\bs{\epsilon} \pi(\bs  {\theta})  \text{\text{d}} \bs{\theta} = \bs{0}.
\end{equation}

Combining \eqref{eq:A2}, \eqref{eq:A6}, and \eqref{eq:A9} results in
\begin{equation} \label{eq:A10}
  \begin{split}
  \nabla_{\bs{\xi}} I(\bs{\xi}) &= \int_{\Theta} \int_{\mathcal{E}} \nabla_{\bs{\xi}} \log \left( \frac{p(\bs{Y} \vert \bs{\theta}, \bs{\xi})}{p(\bs{Y}\vert \bs{\xi})} \right)  p(\bs{Y} \vert \bs{\theta}, \bs{\xi}) \det (\nabla_{\bs{\epsilon}}\bs{Y}) \text{d}\bs{\epsilon} \pi(\bs{\theta})  \text{d} \bs{\theta}\\
  &= \int_{\Theta} \int_{\mathcal{Y}} \nabla_{\bs{\xi}} \log \left( \frac{p(\bs{Y} \vert \bs{\theta}, \bs{\xi})}{p(\bs{Y}\vert \bs{\xi})} \right)  p(\bs{Y} \vert \bs{\theta}, \bs{\xi}) \text{d}\bs{Y} \pi(\bs{\theta})  \text{d} \bs{\theta},
\end{split}
\end{equation}
for the experiment model in \eqref{eq:datamodel}.
\end{proof}

\section{Expected information gain estimators} \label{sec:EIGestimators}

In this section, we present the three estimators used throughout the paper: DLMC, MCLA and DLMCIS, denoted by $\I{DLMC}$, $\I{MCLA}$, and $\I{DLMCIS}$, respectively.
If the computation of $\bs{g}$ requires a numerical approximation of differential equations, we denote by $h^{-\varrho}$ the proportional factor of the average work to evaluate the model outcome $\bs{g}_h$, using a mesh size $h$, with $\varrho >0$.
Moreover, we assume that the numerical error of the PDE solver is proportional to $h^{\eta}$, with $\eta >0$.

\subsection{Double-loop Monte Carlo estimator}

To estimate \eqref{eq:infgain}, we approximate the double integral over both $\Theta$ and $\mathcal{Y}$ using Monte Carlo integration (the outer loop) and the marginalization of the evidence by another Monte Carlo integration (the inner loop).
Thus, the DLMC estimator is defined as
\begin{equation}
\I{DLMC}(\bs{\xi}) \myeq \frac{1}{N}\sum_{n=1}^{N}\left( \log \left( \frac{p(\bs{Y}_n | \bs{\theta}_n, \bs{\xi})}{\frac{1}{M} \sum_{m=1}^{M} p(\bs{Y}_n | \bs{\theta}^*_{n,m}, \bs{\xi})}\right) \right),
\label{eq:DLMC}
\end{equation}
where $N$ and $M$ are the number of samples for the outer and inner loops, respectively.
% An explicit analysis of the average computational work and the optimal sample sizes, according to the error control of the DLMC estimator for $\mathcal{I}_{DL}$, was carried out by Beck et al. \cite{Beck2017fast}.
Note that $(\bs{Y}_n, \bs{\theta}_n)$ are sampled jointly from the likelihood, whereas $\bs{\theta}^*_{n,m}$ is sampled independently from $\bs{\theta}_n$ \red{in each iteration of the} inner loop.
An explicit analysis of the average computational work and the optimal sample sizes ($N$ and $M$) required to achieve a particular error for $\I{DLMC}$ is carried out by Beck et al. \cite{Beck2017fast}.
They show that the total work required to compute the expected information gain using the DLMC estimator is of the order $MNh^{-\varrho}$.
Finally, the DLMC estimator is consistent but has a bias and variance respectively given by
\begin{align} \label{eq:DLMCbias}
  |I - \mathbb{E}[\I{DLMC}]| &\le \C{DL,1}h^{\eta} + \frac{\C{DL,2}}{M} + o(h^{\eta}) + \mathcal{O}\left( \frac{1}{M^2}\right),\\ \label{eq:DLMCvar}
  \mathbb{V}[\I{DLMC}] &= \frac{\C{DL,3}}{N} + \frac{\C{DL,4}}{NM} + \mathcal{O}\left( \frac{1}{NM^2}\right),
\end{align}
for the constants $\C{DL,1}$, $\C{DL,2}$, $\C{DL,3}$, and $\C{DL,4}$ (cf. \cite{Beck2017fast}).

\subsection{Monte Carlo with Laplace approximation estimator}
% As proposed by Long et al. \cite{long2013fast}, we use the Laplace approximation to approximate the Kullback--Leibler divergence.
The Laplace estimator for $D_{kl}$ reduces the approximation of the expected information gain to a single integral over the parameter space $\Theta$.
Thus, the MC estimator of \eqref{eq:I.Lap}, i.e., the MCLA estimator, is defined as
\begin{equation}\label{eq:MCLA}
 \I{MCLA}(\bs{\xi}) \myeq \frac{1}{N} \sum_{n=1}^N\left[-\frac{1}{2} \log(\text{det}(2 \pi \bs{\Sigma}(\bs{\xi}, \bs{\theta}_n))) -\frac{d}{2} - \log (\pi(\bs{\theta}_n)) \right],
\end{equation}
where $N$ is the number of MC samples and $d$ is the dimensionality of $\bs{\theta}$.
Using forward finite differences to estimate the Jacobian of $\bs{g}$ with respect to $\bs{\theta}$, the cost of evaluating the MCLA estimator is $N(d+1) h^{-\varrho}$.

According to Beck et al. \cite{Beck2017fast}, the bias and variance of the MCLA estimator are, respectively,
\begin{align}
  |I - \mathbb{E}[\I{MCLA}]| &\le \C{LA,1}h^{\eta} + \frac{\C{LA,2}}{N_e} + o(h^{\eta}),\\
  \mathbb{V}[\I{MCLA}] &= \frac{\C{LA,3}}{N},
\end{align}
where $\C{LA,1}$, $\C{LA,2}$, and $\C{LA,3}$ are constants to be estimated.
For a fixed number of experiments $N_e$, the bias of the MCLA estimator does not vanish as the number of samples goes to infinity; thus, the MCLA estimator is inconsistent.
However, the more concentrated the mass of probability of the true posterior is around the maximum a posteriori value, the better the Laplace approximation is.
Therefore, as the optimization is performed and the posterior becomes more concentrated at the true values of the parameters, the Laplace approximation bias decreases, i.e., we expect constant $\C{LA,2}$ to decrease as the optimization goes on.

\subsection{Double-loop Monte Carlo with Laplace-based importance sampling estimator}
The evaluation of $\I{DLMC}$ in \eqref{eq:DLMC} may be unsuccessful due to \emph{numerical underflow} if the prior is not concentrated enough around the posterior or if the standard deviation of the measurement errors and the number of repetitive experiments $N_e$ are large.
The MCLA estimator does not have this issue, but, as mentioned before, it includes a possible bias due to the Laplace approximation.
An alternative estimator that possesses the robustness of DLMC and the speed of MCLA is proposed in \cite{Beck2017fast}, where the Laplace approximation of the posterior distribution, $\pi_{_{\hbox{\tiny{LA}}}}(\bs{\theta} | \bs{Y},\bs{\xi})$, is used as an importance sampling distribution to estimate the evidence.
We write the DLMCIS estimator as
\begin{equation}
\I{DLMCIS}(\bs{\xi}) \myeq \frac{1}{N} \sum_{n=1}^{N}\left( \log \left( \frac{p(\bs{Y}_n | \bs{\theta}_n, \bs{\xi})}{\frac{1}{M} \sum_{m=1}^{M} \ell (\bs{Y}_n | \bs{\theta}^*_{n,m}, \bs{\xi})}\right) \right),  \quad \hbox{with}\quad \ell(\bs{Y}; \cdot, \bs{\xi}) = \frac{p(\bs{Y}| \cdot, \bs{\xi}) \pi(\cdot) }{\pi_{_{\hbox{\tiny{LA}}}}(\cdot|\bs{Y}, \bs{\xi})},
\label{eq:DLMCIS}
\end{equation}
where $\pi_{_{\hbox{\tiny{LA}}}}$ is given in Remark \ref{rmk1:eigla}.
As in DLMC, the inner-loop samples $\bs{\theta}^*_{n,m}$ are independent from the outer-loop samples, $\bs{\theta}_n$.
The change of measure in the importance sampling requires approximating the MAP value and the covariance matrix at the MAP value.
As can be observed in \eqref{eq:thetahat}, estimating the MAP value is an optimization problem in itself.
Here, we employ the Nelder-Mead algorithm \cite{nelder1965simplex} to find a MAP estimate due to its robustness and global convergence quality.
As for $\bs{\Sigma}$, it can be calculated from \eqref{eq:sigmahat} using the Jacobian of the forward model with respect to $\bs{\theta}$.

Beck et al. \cite{Beck2017fast} show that the error decomposition for the DLMCIS estimator is the same as for DLMC, \eqref{eq:DLMCbias} and \eqref{eq:DLMCvar}, but with much smaller constants on the error decomposition.
This results in fewer forward model evaluations in the inner loop being required to achieve a given tolerance.
If forward differences are used to approximate the Jacobian of the model with respect to $\bs{\theta}$, needed to approximate $\bs{\Sigma}$, each evaluation of the DLMCIS estimator has cost $N(d + 1 + M + \C{MAP}) h^{-\varrho}$, where $\C{MAP}$ is the number of model evaluations required to find $\bs{\hat\theta}$.
In comparison to DLMC, the DLMCIS estimator has an extra cost per outer loop iteration of $\C{MAP} + d + 1$, but, since $M$ is reduced, there is often a significant overall reduction in computational effort.

\section{Gradient estimators for stochastic optimizers}\label{sec:SGestimators}
Let $f(\bs{\xi}, \bs{\theta}, \bs{Y})$ be the entropic discrepancy function between the data evidence and the likelihood.
From (\ref{eq:infgain}), $f$ is given by
\begin{equation}
 f(\bs{\xi},\bs{\theta}, \bs{Y}) = \log \left( \frac{p(\bs{Y} \vert \bs{\theta}, \bs{\xi})}{p(\bs{Y}\vert \bs{\xi})} \right).
\end{equation}
Consequently, we have $ \nabla_{\bs{\xi}} I(\bs{\xi}) = \nabla_{\bs{\xi}} \mathbb{E}_{\bs{\theta}, \bs{Y}}[f(\bs{\xi}, \bs{\theta}, \bs{Y})]$.
Moreover, from Proposition \ref{propo:1}, we conclude that, for the experimental design problem with data model \eqref{eq:datamodel},
\begin{equation}
  \nabla_{\bs{\xi}} \mathbb{E}_{\bs{\theta}, \bs{Y}}[f(\bs{\xi}, \bs{\theta}, \bs{Y})] = \mathbb{E}_{ \bs{\theta}, \bs{Y}}[\nabla_{\bs{\xi}} f(\bs{\xi}, \bs{\theta}, \bs{Y})].
\end{equation}
We name the unbiased stochastic gradient estimators of the expected information gain $\mathcal{G} = \nabla_{\bs{\xi}} f$, i.e., $\mathbb{E}_{\bs{\theta}, \bs{Y}}[\mathcal{G}] = \mathbb{E}_{\bs{\theta}, \bs{Y}}[\nabla_{\bs{\xi}} f]$.
% Regarding biased estimators of $f$, the stochastic gradient method can be used as long as an unbiased gradient of the estimator is available.

Supposing that $\widehat{f}$ is a possibly biased estimator of $f$, the stochastic gradient $\widehat{\mathcal{G}} \myeq \nabla_{\bs{\xi}} \widehat{f}$ is an unbiased estimator of the gradient of $\red{\nabla_{\bs{\xi}} \mathbb{E}_{\bs{\theta}, \bs{Y}}}[ \widehat{f}]$ by construction.
\red{
However, the stochastic gradient estimators are not necessarily unbiased with respect to the true gradient $\nabla_{\bs{\xi}} \mathbb{E}_{\bs{\theta}, \bs{Y}}[f]$.}

Next, we derive three stochastic gradient estimators associated with the expected information gain estimators presented in Section \ref{sec:EIGestimators}.

\subsection{Stochastic gradient of the double loop Monte Carlo estimator}
We denote by $\G{MC}$ the gradient of the entropic function $f$ using a Monte Carlo sample of size $M$ to approximate the evidence $p(\bs{Y}|\bs{\xi})$.
This approach is similar to taking the gradient of $\I{DLMC}$ in \eqref{eq:DLMC}, except that, due to the nature of stochastic gradient methods, the variance of the estimator is allowed to be large, i.e., $N$ is set to one.
% The $\G{MC}$ estimator is
% The stochastic gradient of the DLMC estimator \eqref{eq:DLMC}, the \SG{MC} estimator, is given by
Therefore, the \SG{MC} estimator is given by
\begin{equation} \label{eq:SGDLMC}
\G{MC}(\bs{\xi}, \bs{\theta}, \bs{Y}) \myeq  \nabla_{\bs{\xi}} \left(\log \left( \frac{p(\bs{Y} | \bs{\theta}, \bs{\xi})}{\frac{1}{M} \sum_{m=1}^{M} p(\bs{Y} | \bs{\theta}^*_m, \bs{\xi})} \right) \right).
\end{equation}

Note that $\G{MC}$ is an asymptotically unbiased estimator of $\mathbb{E}_{\bs{\theta},\bs{Y}}[\nabla_{\bs{\xi}}f(\bs{\xi},\bs{\theta},\bs{Y})]$, i.e.,
\begin{equation}
  \G{MC} = \widehat{\mathbb{E}_{\bs{\theta},\bs{Y}}}[\nabla_{\bs{\xi}}f(\bs{\xi},\bs{\theta},\bs{Y})] \quad \text{and} \quad \underset{M \xrightarrow{} \infty}{\lim} \red{\mathbb{E}_{\bs{\theta},\bs{Y}}} [\G{MC}] =
  \mathbb{E}_{\bs{\theta},\bs{Y}}[\nabla_{\bs{\xi}}f(\bs{\xi},\bs{\theta},\bs{Y})],
\end{equation}
as the Monte Carlo sampling for the marginal likelihood generates a bias of order $M^{-1}$.

The estimation of (\ref{eq:SGDLMC}) by forward finite differences requires $\text{dim}(\bs{\xi}) + 1$ model evaluations per inner sample.
Thus, the total number of model evaluations is $(\text{dim}(\bs{\xi}) + 1)M$ per iteration in the optimization.
In contrast, the gradient of the DLMC estimator presented in \eqref{eq:DLMC}, using forward finite differences, costs $(\text{dim}(\bs{\xi}) + 1)NM$, i.e., $N$ times more per evaluation than the \SG{MC} estimator.
Finally, the estimator \eqref{eq:SGDLMC} is biased, with bias of order $M^{-1}$, but consistent.

\subsection{Stochastic gradient of the Monte Carlo with Laplace approximation estimator}
The stochastic gradient estimator with respect to $\bs{\xi}$ based on the Laplace approximation \eqref{eq:MCLA}, the \SG{LA} estimator, is denoted by $\G{LA}(\bs{\xi}, \bs{\theta})$.

\begin{proposition}
$\G{LA}(\bs{\xi}, \bs{\theta})$ is given by
\begin{align}\label{eq:sgd_mcla}
\G{LA} (\bs{\xi}, \bs{\theta}) = - \frac{1}{2} \bs{\Sigma}^{-1} (\bs{\xi}, \bs{\theta}) \colon \nabla_{\bs{\xi}} \bs{\Sigma}  (\bs{\xi}, \bs{\theta}) = - \sum\limits_{k=1}^{d} \sigma_k^{-1} \nabla_{\bs{\xi}} \sigma_k,
\end{align}
where $\left\{ \sigma^2_i \right\}^{d}_{i=1} $ are the eigenvalues of $\bs{\Sigma}(\bs{\xi}, \bs{\theta})$.
\end{proposition}

\begin{proof}
Considering the gradient of the integrand of \eqref{eq:I.Lap},
\begin{equation}
\G{LA}(\bs{\xi}, \bs{\theta}) =  \nabla_{\bs{\xi}}\left( -\frac{1}{2} \log \left( \text{det}\left(2 \pi\bs{\Sigma}(\bs{\xi}, \bs{\theta})\right) \right) - \frac{d}{2} - \log(\pi(\bs{\theta})) \right),
\end{equation}
and since the prior does not depend on $\bs{\xi}$, we write the \SG{LA} estimator using Jacobi's formula as
\begin{align}
\G{LA}(\bs{\xi}, \bs{\theta})
&= \nabla_{\bs{\xi}}\left( -\frac{1}{2} \log \left( \text{det}\left(2 \pi\bs{\Sigma}(\bs{\xi}, \bs{\theta})\right) \right) \right)\nonumber \\
&=  \frac{-1}{2\,\text{det}\, \bs{\Sigma}(\bs{\xi}, \bs{\theta})} \nabla_{\bs{\xi}}\left( \text{det}\,\bs{\Sigma} (\bs{\xi}, \bs{\theta}) \right) \nonumber \\
&= - \frac{1}{2} \text{tr}\left(\bs{\Sigma}^{-1}(\bs{\xi}, \bs{\theta}) \cdot \nabla_{\bs{\xi}} \bs{\Sigma}  (\bs{\xi}, \bs{\theta}) \right) \nonumber \\
&= - \frac{1}{2} \bs{\Sigma}^{-1} (\bs{\xi}, \bs{\theta}) \colon \nabla_{\bs{\xi}} \bs{\Sigma}  (\bs{\xi}, \bs{\theta}). \label{eq:grad.F.1}
\end{align}
Considering \eqref{eq:sigmahat}, we write the gradient of $\bs\Sigma^{-1}$ as
\begin{equation}\label{eq:grad.Sigma_inv}
\nabla_{\bs{\xi}} \bs\Sigma^{-1}(\bs{\xi}, \bs{\theta}) = 2 N_e\,\text{Sym}\left(\nabla_{\bs{\xi}} \nabla_{\bs{\theta}}\bs{g}(\bs{\xi}, \bs{\theta}) \cdot \bs{\Sigma}_{\epsilon}^{-1} \cdot \nabla_{\bs{\theta}}\bs{g}(\bs{\xi}, \bs{\theta})\right),
\end{equation}
where $\text{Sym}(\cdot)$ is the symmetric algebraic operator $\text{Sym}_{ij}(\bs{A}) = \frac{1}{2}(A_{ij} + A_{ji})$.
Moreover, the gradient of a nonsingular square matrix $\bs{A}$ can be written as $\nabla_{\bs{x}} \bs{A} = -\bs{A} \cdot \nabla_{\bs{x}} \bs{A}^{-1} \cdot \bs{A}$  or, in index notation, as $ \frac{\partial A_{ij}}{\partial x_s} = -A_{ik} \frac{\partial A^{-1}_{kl}}{\partial x_s} A_{lj}$.
Then, we express $\nabla_{\bs{\xi}} \bs\Sigma$ using \eqref{eq:grad.Sigma_inv} as
\begin{equation}\label{eq:grad.Sigma}
\nabla_{\bs{\xi}} \bs\Sigma(\bs{\xi}, \bs{\theta}) = - 2 N_e\,\bs\Sigma(\bs{\xi}, \bs{\theta})\cdot\text{Sym}\left(\nabla_{\bs{\xi}} \nabla_{\bs{\theta}}\bs{g}(\bs{\xi}, \bs{\theta}) \cdot \bs{\Sigma}_{\epsilon}^{-1} \cdot \nabla_{\bs{\theta}}\bs{g}(\bs{\xi}, \bs{\theta})\right)\cdot\bs{\Sigma}(\bs{\xi}, \bs{\theta}),
\end{equation}
or, in index notation, as
\begin{equation}
	\frac{\partial \Sigma_{uv}}{\partial \xi_s} = - 2 N_e \Sigma_{ul} \, \text{Sym}_{lm} \left(\frac{\partial^2 g_i}{\partial{\xi}_s\partial{\theta}_l} (\Sigma_{\epsilon}^{-1})_{ij} \frac{\partial g_j}{\partial{\theta}_m} \right) \Sigma_{mv}.
\end{equation}
Therefore, we can write \eqref{eq:grad.F.1} as
\begin{align}\label{eq:grad.F.2}
\G{LA}(\bs{\xi}, \bs{\theta}) & = N_e \, \bs{\Sigma}(\bs{\xi}, \bs{\theta})^{-1} \colon \left[\bs{\Sigma}(\bs{\xi}, \bs{\theta}) \cdot \text{Sym}\left(\nabla_{\bs{\xi}} \nabla_{\bs{\theta}}\bs{g}(\bs{\xi}, \bs{\theta}) \cdot \bs{\Sigma}_{\epsilon}^{-1} \cdot \nabla_{\bs{\theta}}\bs{g}(\bs{\xi}, \bs{\theta})\right) \cdot \bs{\Sigma}(\bs{\xi}, \bs{\theta}) \right] \nonumber \\
& = N_e \, \bs{\Sigma}(\bs{\xi}, \bs{\theta}) \colon \text{Sym}\left(\nabla_{\bs{\xi}} \nabla_{\bs{\theta}}\bs{g}(\bs{\xi}, \bs{\theta}) \cdot \bs{\Sigma}_{\epsilon}^{-1} \cdot \nabla_{\bs{\theta}}\bs{g}(\bs{\xi}, \bs{\theta})\right).
\end{align}
Thus, in index notation, the $s^{th}$ component of $\G{LA}$ is given by
\begin{align} \label{eq:G_LA_index}
\left(\G{LA} (\bs{\xi}, \bs{\theta})\right)_s & = N_e \, \Sigma_{ul} \, \text{Sym}_{lm} \left(\frac{\partial^2 g_i}{\partial{\xi}_s\partial{\theta}_l} (\Sigma_{\epsilon}^{-1})_{ij} \frac{\partial g_j}{\partial{\theta}_m} \right) \Sigma_{mv} \Sigma_{uv} ^{-1} \nonumber \\
& = N_e \, \Sigma_{ml} \, \text{Sym}_{lm} \left(\frac{\partial^2 g_i}{\partial{\xi}_s\partial{\theta}_l} (\Sigma_{\epsilon}^{-1})_{ij} \frac{\partial g_j}{\partial{\theta}_m} \right).
\end{align}
Moreover, considering that $\left\{ \sigma^2_i \right\}^{d}_{i=1} $ are the eigenvalues of $\bs{\Sigma}$, we can write the determinant of $\bs{\Sigma}$ as $\prod\limits_{i=1}^{d}\sigma_i^{2}$.
Then, to explicitly show the relation between the \SG{LA} estimator and the eigenvalues of the covariance of the posterior pdf, we rewrite the gradient in \eqref{eq:grad.F.1} as
\begin{align}
\G{LA}(\bs{\xi}, \bs{\theta})
&= \frac{-1}{2\,\text{det}\, \bs{\Sigma}(\bs{\xi}, \bs{\theta})} \nabla_{\bs{\xi}}\left( \text{det}\,\bs{\Sigma} (\bs{\xi}, \bs{\theta}) \right) \nonumber \\
&= -\frac{1}{2} \prod\limits_{i=1}^{d}\sigma_i^{-2} \nabla_{\bs{\xi}}\left(\prod\limits_{j=1}^{d}\sigma_j^2\right) \nonumber \\
&= -\frac{1}{2} \prod\limits_{i=1}^{d}\sigma_i^{-2} \sum\limits_{k=1}^{d} \left( \nabla_{\bs{\xi}} \sigma_k^2 \prod\limits_{\substack{j=1\\j\neq k}}^{d} \sigma_j^2 \right) \nonumber \\
&= -\frac{1}{2} \sum\limits_{k=1}^{d} \sigma_k^{-2} \nabla_{\bs{\xi}} \sigma_k^2 \nonumber \\ &= - \sum\limits_{k=1}^{d} \sigma_k^{-1} \nabla_{\bs{\xi}} \sigma_k.
\end{align}
Finally, from \eqref{eq:sgd_mcla}, we state that maximizing the expected information gain is equivalent to minimizing the sum of the logarithm of the posterior standard deviations.
\end{proof}

% The \SG{LA} estimator only requires a single evaluation of the Jacobian of the model with respect to the parameters and a gradient of the Jacobian with respect to the optimization parameters.
The \SG{LA} estimator requires the cross-partial derivatives of the model $\bs{g}$ with respect to $\bs{\xi}$ and $\bs{\theta}$, as can be seen on \eqref{eq:G_LA_index}.
Therefore, when the forward finite differences are applied with respect to both $\bs{\xi}$ and $\bs{\theta}$, the cost of the estimator per evaluation is $(\textrm{dim}(\bs{\xi})+1)(d+1) h^{-\varrho}$.

\subsection{Stochastic gradient of the double loop Monte Carlo with Laplace-based importance sampling estimator}\label{sec:sgd_dlmcis}

In the same spirit as the DLMCIS estimator, we reduce the occurrence probability of numerical underflow by changing the measure in the evidence estimation using Laplace approximation.
% To avoid the bias of the Laplace approximation and to improve the efficiency of the MC sampling used to estimate the evidence in $\nabla_{\bs{\xi}}f$, we introduce the \SG{MCIS} estimator,
Let \SG{MCIS} be the stochastic gradient of DLMCIS
\begin{equation}
 \G{MCIS}(\bs{\xi}, \bs{\theta}, \bs{Y}) = \nabla_{\bs{\xi}} \left( \log \left( \frac{p(\bs{Y} | \bs{\theta}, \bs{\xi})}{\frac{1}{M} \sum_{m=1}^{M} \ell(\bs{Y} | \bs{\theta}^*_{m}, \bs{\xi})}\right) \right).
\label{eq:gradis}
\end{equation}
Note that $\bs{\theta}$ is sampled from the prior pdf $\pi(\bs{\theta})$, whereas $\bs{\theta}^*_m$ is sampled from the Laplace importance sampling pdf $\pi_{_{\hbox{\tiny{LA}}}}(\bs{\theta}^*|\bs{Y}, \bs{\xi})$.
From \eqref{eq:gradis}, it can be seen that $\G{MCIS}$ is an asymptotically unbiased estimator of $\mathbb{E}_{\bs{\theta},\bs{Y}}[\nabla_{\bs{\xi}}f(\bs{\xi},\bs{\theta},\bs{Y})]$.

To evaluate $\G{MCIS}$, we estimate the MAP value solving \eqref{eq:thetahat}, and then we evaluate the covariance matrix at the MAP estimate using \eqref{eq:sigmahat}.
The gradient is evaluated using the $\bs{\theta}^*$ sampled using importance sampling.
If forward finite differences are used to approximate the derivatives, the cost of evaluating \eqref{eq:gradis} is $[d + 1 + \C{MAP} + (\textrm{dim}(\bs{\xi})+1)M ]h^{-\varrho}$.
However, $M$ is much smaller for \SG{MCIS} than for \SG{MC} due to the more efficient sampling \cite{Beck2017fast}.

\section{Optimization methods}\label{sec:Opt}
We present three stochastic optimization methods to solve the OED problem: SGD, ASGD, and rASGD.
We combine these with the stochastic gradient estimators presented in Section \ref{sec:SGestimators}, e.g., rASGD using \SG{LA} is denoted as \rASGD{LA}.
We recall that $f$ is assumed to be smooth enough with respect to both $\bs{\xi}$ and $\bs{\theta}$.
We consider that the steepest descent algorithm of the maxima search, using the \textit{full gradient} (FGD) and starting at $\bs{\xi}_0$, is given by
\begin{equation}
\bs{\xi}_{k+1} = \bs{\xi}_k +  \alpha_k \nabla_{\bs{\xi}} \mathbb{E}_{\bs{\theta}, \bs{Y}}[ f(\bs{\xi}_k, \bs{\theta}, \bs{Y})],\quad k \geq 0,
\label{eq:gradbasedopt}
\end{equation}
where $\alpha_k$ is a step-size sequence of positive values, also known as learning rates.
Based on this algorithm, we present the three stochastic optimization methods that we apply to OED.

\subsection{Stochastic gradient descent} \label{sec:SGD}
SGD estimates the gradient, based on the stochastic approximation introduced by Robbins and Monro \cite{robbins1951stochastic, kiefer1959optimum, lai1979adaptive}, cumulatively, and throughout several iterations.
It requires only one sample per iteration.
We write SGD for OED as
\begin{equation}\label{eq:sgd}
\bs{\xi}_{k+1} = \bs{\xi}_{k} + \alpha_k \mathcal{G}(\bs{\xi}_k, \bs{\theta}_k, \bs{Y}_{k}), \quad k \geq 0,
\end{equation}
where $\bs{\theta}_k$ is sampled independently from $\pi(\bs{\theta})$ for each iteration, and $\bs{Y}_k$ is sampled from $p(\bs{Y} |\bs{\theta}_k , \bs{\xi})$.
Additionally, $\mathcal{G}$ is any of the stochastic gradient estimators $\G{MC}$, $\G{LA}$, or $\G{MCIS}$ presented in Section \ref{sec:SGestimators} evaluated with the singleton sample set $\left\{\bs{\theta}_k,\bs{Y}_k\right\}$.
In this framework, SGD evaluates jointly the expectations over both $\bs{\theta}$ and $\bs{Y}$; the statistical error averages out as more iterations are completed.
This can be motivated by using \eqref{eq:sgd} to write
\begin{align}
  \bs{\xi}_{k+1} = \bs{\xi}_0 +  \sum_{i=0}^{k} \alpha_i \mathcal{G}(\bs{\xi}_i, \bs{\theta}_i,\bs{Y}_i),
\end{align}
remembering that $\mathcal{G}$ is an unbiased estimator of the gradient of the objective function to be minimized.
% Since $\mathcal{G}$ is an unbiased estimator of $\nabla_{\bs{\xi}} f$,
% \begin{align}
%   \mathbb{E}_{\bs{\theta},\bs{Y}}[\bs{\xi}_{k+1}] &= \bs{\xi}_0 +  \sum_{i=0}^{k} \alpha_i \mathbb{E}_{\bs{\theta},\bs{Y}} [  \nabla_{\bs{\xi}} f(\bs{\xi}_i, \bs{\theta}_i,\bs{Y}_i)]\\
% 	&= \bs{\xi}_0 +  \sum_{i=0}^{k} \alpha_i  \nabla_{\bs{\xi}} F(\bs{\xi}_i)
% \end{align}

For SGD to converge to the optimum, the step-size must decrease as the number of iterations increases.
Robbins and Monro \cite{robbins1951stochastic} prove convergence when the step-size is a divergent series with squared convergence, i.e., $\alpha_k = {\alpha_0}/{k}$.
%\begin{equation}\label{eq:stepsize}
%\sum_{k=1}^{\infty} \alpha_k = \infty \quad \textrm{and} \quad \sum_{k=1}^{\infty} \alpha^2_k < \infty,
%\end{equation}
%the convergence of $(\bs{\xi}_k)_{k\geq 0}$ to the optimum is guaranteed, i.e. $||\bs{\xi}_k - \bs{\xi}^*|| \xrightarrow{} 0$ as $k$ goes to infinity.
% Robbins and Monro \cite{robbins1951stochastic} proposed the natural step-size sequence of $\alpha_k = \frac{\alpha_0}{k}$ with $\alpha_0 > 0$.
Polyak and Juditsky \cite{polyak1992acceleration} prove that the average of $\{\bs{\xi}_i\}_{i=0}^k$ converges to the optimum when the step-size sequence satisfies $\alpha_k = \alpha_0/k^\beta$ for $1/2 < \beta < 1$.
%Neither Robbins and Monro \cite{robbins1951stochastic} nor Polyak and Juditsky \cite{polyak1992acceleration} discuss the initial step-size $\alpha_0$.
For an objective function whose gradient is $L$-Lipschitz continuous, Nemirovski \cite{nemirovski2005efficient} uses a step-size of $\alpha_k = \alpha_0/\sqrt{k}$, with $\alpha_0 = D/L$ and $D$ being the diameter of the search space.
Nemirovski \cite{nemirovski2005efficient} proves that, in this case, the weighted sliding average $\bar{\bs{\xi}}$ converges to the optimum at a rate of $\mathcal{O}(1/\sqrt{k})$, with
\begin{equation}\label{eq:sliding.average}
	\bar{\bs{\xi}}_k = \left(\sum_{\frac{k}{2}\le i \le k} \alpha_i\right)^{-1}
	\sum_{\frac{k}{2}\le i \le k} \alpha_i \bs{\xi}_i.
\end{equation}
For the strongly convex case, Nemirovski \cite{nemirovski2005efficient} also proves that stochastic gradient descent with a sliding average achieves a convergence of $\mathcal{O}(1/k)$ when the step $\alpha_0$ satisfies $\alpha_0 \mu < 1$, where $\mu$ is the strong-convexity constant.
Here, we follow the approach of Nemirovski \cite{nemirovski2005efficient} and adopt the step-size sequence $\alpha_k = {\alpha_0}/{\sqrt{k}}$, given that we assume $\mu$ to be unknown.
According to Nemirovksi \cite{nemirovski2005efficient}, the convergence of the objective function in SGD is bounded as
\begin{equation} \label{eq:boundSGD}
	\red{
  \mathbb{E}_{\bs{\theta}}[f(\bs{\xi}^*,\bs{\theta})
	- f(\bs{\bar\xi}_k,\bs{\theta})]
	}
	\le
  \left(2 \sum_{\frac{k}{2}\le i \le k} \alpha_i\right)^{-1}
  \left[
  D^2 + \left( \sigma^2 +  \mathbb{E}[\|\nabla f(\bs{\xi}_k,\bs{\theta}) \|]^2\right)
  \sum_{\frac{k}{2}\le i \le k} \alpha_i^2
  \right],
\end{equation}
where
\begin{equation}
  \sigma^2 \geq ~ \underset{\bs{\xi} \in \Xi}{\text{sup}} ~ \mathbb{V}[\| \nabla f (\bs{\xi},\bs{\theta})\|].
\end{equation}
Using step-size $\alpha_k = \alpha_0/\sqrt{k}$, as $k \rightarrow \infty$, $\left(\sum_{\frac{k}{2}\le i \le k} \alpha_i\right)^{-1}$ decreases faster than $\left(\sum_{\frac{k}{2}\le i \le k} \alpha_i\right)^{-1}  $ $\sum_{\frac{k}{2}\le i \le k} \alpha_i^2$, despite both converging sublinearly.
Consequently, the $\sigma^2 + \mathbb{E}[\|\nabla f \|]^2$ term dominates convergence if optimization runs for long enough.
Moreover, as SGD approximates local optima and the norm of the true gradient decreases, the gradient norm variance dominates the convergence.

For the example described in Section \ref{sec:ex1}, Figure \ref{fig:QuadraticMBByCost} depicts how the distance from the optimal design evolves as a function of the number of gradient evaluations for different sample sizes $N$.
Figure \ref{fig:MB1} shows the distance to the optimum for $\bs{\xi}$ and Figure \ref{fig:MB1} shows the distance for its sliding average $\bar{\bs{\xi}}$.
\begin{figure}[ht]
	\centering
  \begin{subfigure}[t]{0.48\textwidth}
    \centering
	\includegraphics[width=\textwidth]{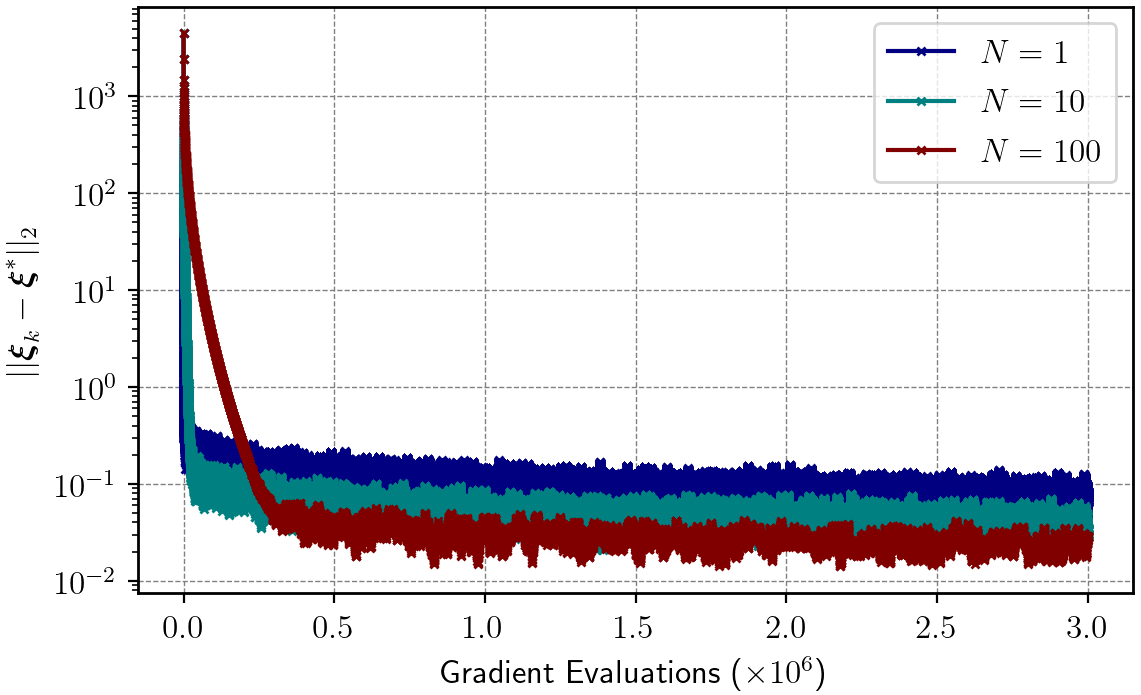}
  \caption{}
  \label{fig:MB1}
\end{subfigure}
\begin{subfigure}[t]{0.48\textwidth}
  \centering
  \includegraphics[width=\textwidth]{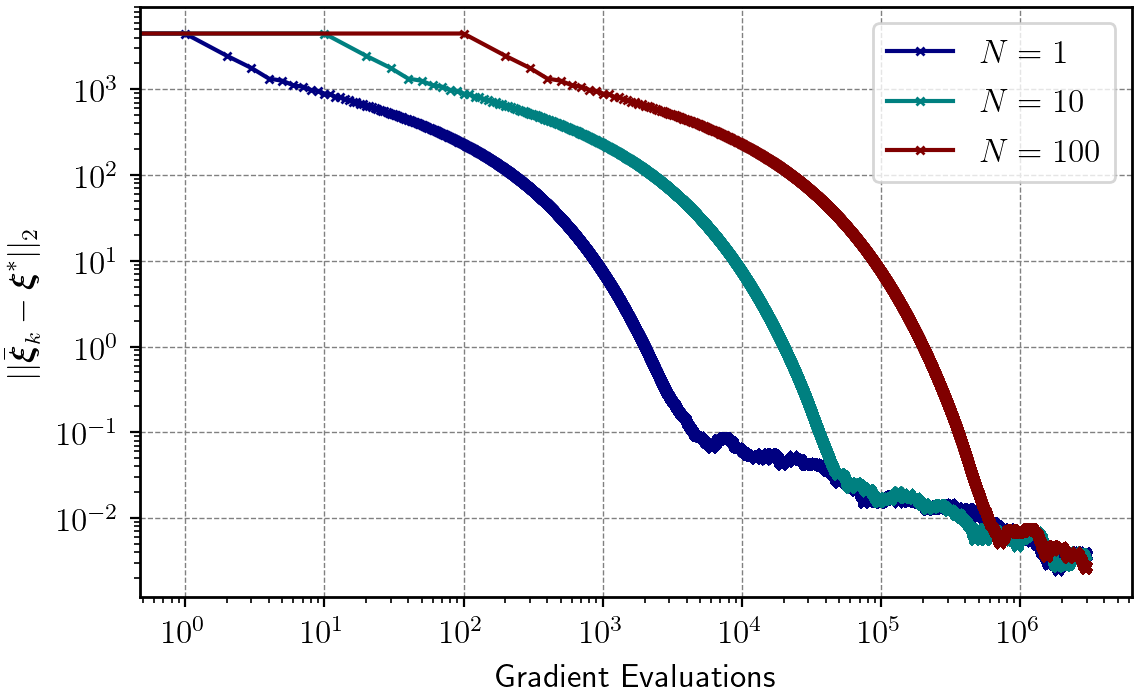}
  \caption{}
  \label{fig:MB2}
\end{subfigure}
\caption{Convergences of $\bs{\xi}$ (left) and $\bar{\bs{\xi}}$ (right, cf. \eqref{eq:sliding.average}) for the quadratic function (Example 1) using SGD with different mini-batch sample sizes $N$.}
\label{fig:QuadraticMBByCost}
\end{figure}
% In Figure \ref{fig:QuadraticMBByCost}, SGD converges fast up to a point, where the convergence changes to a sublinear regime.
In Figure \ref{fig:MB1}, it can be seen that, in initial iterations, SGD converges similarly to the steepest descent method.
As optimization progresses and the noise in the gradient estimates dominates convergence, sublinearity emerges.
To illustrate the sublinear regime, we plot Figure \ref{fig:MB2} in logarithmic scale on both axes; thus the Q-sublinear convergence shows as linear in the plot.
Independently of the mini-batch size used, SGD's asymptotic convergence is limited by the same lower bound: the term in \eqref{eq:boundSGD} containing $\sigma^2$.
Mini-batch sampling of size $N$ reduces the variance to $\sigma^2/N$, but the cost per iteration is also increased by the same amount; therefore, cost-wise, the lower bound remains the same, as shown in Figure \ref{fig:MB2}.
According to Cotter et al \cite{cotter2011better}, the advantage of using mini-batch sampling is that parallelization can be used to speed up the optimization process.
Cotter et al \cite{cotter2011better} use a distributed mini-batch technique to parallelize SGD without losing efficiency.
However, in terms of the total number of gradient evaluations required to achieve a certain tolerance, SGD without mini-batch sampling is more efficient than SGD using simple mini-batch sampling.

\subsection{Nesterov's accelerated gradient descent} \label{sec:Nesterov}
The Nesterov gradient scheme is a first-order accelerated method for \emph{deterministic} optimization \cite{nesterov1983method,nesterov2013introductory,nemirovski2005efficient}.
The basic idea is to use a momentum (an analogy to linear momentum in physics \cite{rumelhart1986learning,o2015adaptive}) that determines the step to be performed, based on information from previous iterations.
The Nesterov gradient scheme is considered \textit{accelerated} because, for convex and smooth objective functions with $L$-Lipschitz gradient, it improves the convergence rate of the objective function from $\mathcal{O}(1/k)$ to $\mathcal{O}(1/k^2)$, which is provably the optimal convergence rate for first-order optimization methods in this class of problems \cite{nesterov2013introductory}.

Nesterov's accelerated gradient descent (AGD) algorithm for the Bayesian design optimization problem in \eqref{eq:infopt} is defined as
\begin{equation}
\left\{
\begin{array}{ll}
 \bs{z}_{k+1} = \bs{\xi}_k + \alpha \nabla_{\bs{\xi}}\mathbb{E}_{\bs{\theta}, \bs{Y}} \left[ f(\bs{\xi}_k, \bs{\theta}, \bs{Y})\right] \\\\
\bs{\xi}_{k+1} = \bs{z}_{k+1} + \gamma_{k+1} (\bs{z}_{k+1} - \bs{z}_k).
\end{array}
\right. \label{eq:nesterov}
\end{equation}

Here, the sequence $\left(\gamma_k \right)_{k\geq 0}$ is given by
\begin{equation}
\gamma_{k+1} = \frac{\lambda_k(1-\lambda_k)}{\lambda_k^2 + \lambda_{k+1}},
\end{equation}
where the sequence $\left(\lambda_k\right)_{k\geq 0}$ solves
\begin{equation}
\lambda_{k+1}^2 = (1 - \lambda_{k+1})\lambda_{k}^2 + q \lambda_{k+1}, \quad \lambda_0 = 1,
\end{equation}
and $q$ is a positive real number that is less than one ($q \in (0,1)$).
The constant $q$ defines how much momentum is used in the acceleration, e.g., setting $q=1$ results in the classical steepest descent algorithm.
Usually, a value of $0$ is specified for $q$, resulting in the original algorithm proposed by Nesterov \cite{nesterov1983method}.
Since AGD is a deterministic method, a fixed step-size $\alpha$ is used.

Using the stochastic gradient estimators presented in Section \ref{sec:SGestimators}, we obtain the ASGD method as
\begin{equation}
\left\{
\begin{array}{ll}
\bs{z}_{k+1} = \bs{\xi}_k + \alpha_k \mathcal{G}(\bs{\xi}_k, \bs{\theta}_k, \bs{Y}_k) \\\\
\bs{\xi}_{k+1} = \bs{z}_{k+1} + \gamma_{k+1} (\bs{z}_{k+1} - \bs{z}_k) ,
\end{array}
\right. \label{eq:ASGD}
\end{equation}
using decreasing step-sizes as discussed in Section \ref{sec:SGD}.

The use of Nesterov's acceleration in stochastic optimization is not novel, and many publications have addressed the subject in the training process in machine learning \cite{johnson2013accelerating,nitanda2016accelerated,allen2017katyusha}.
However, all of those studies combine stochastic gradient methods with variance-reduction techniques due to the sensitivity of ASGD to noise in the gradient estimation.
Cotter et al. \cite{cotter2011better} show that the lower bound of the asymptotic convergence rate for ASGD is the same as for SGD, i.e., acceleration does not improve the convergence rate of SGD in the asymptotic phase.
However, on the non-asymptotic phase of SGD, Nesterov's acceleration can improve the convergence rate.
In the next section, we present a restart method used to improve the convergence rate of ASGD.

\subsection{Restart Nesterov method} \label{sec:restart}
When using Nesterov's acceleration with $q=0$, oscillations of the algorithms around the optimum because of an excess of momentum are common.
% To avoid this problem, Nesterov \cite{nesterov2013introductory} presents the optimal value for the parameter $q$ for strongly-convex problems in the deterministic setting.
For strongly convex first-order $L$-Lipschitz problems, where $\mu$ is the strong-convexity constant (i.e., $\mu \preceq \nabla \nabla f \preceq L$ with $f$ being the objective function), Nesterov \cite{nesterov2013introductory} proves that $q^*=\mu / L$ achieves the optimal convergence rate for first-order optimization methods.
% Thus, the optimal $q^*$ is the inverse of the conditioning number of the Hessian of $\mathbb{E}[f]$.
For $q < q^*$, the momentum is excessive and leads to the aforementioned oscillations around the optima; for $q > q^*$, the convergence is suboptimal.
The quantities $\mu$ and $L$ are expensive to estimate for OED problems based on PDE models.
O'Donoghue and Cand\`es \cite{o2015adaptive} propose an alternative method for achieving the same convergence rate as with $q^*$ without evaluating $\mu$ and $L$ for the deterministic case.
Their method consists of restarting the acceleration whenever the optimizer moves in an unwanted direction, e.g., for the maximization of $I$, when
\begin{equation}
\nabla_{\bs{\xi}} \mathbb{E}_{\bs{\theta}, \bs{Y}} \left[ f(\bs{\xi}_{k}, \bs{\theta}, \bs{Y})\right] \cdot (\bs{\xi}_{k} - \bs{\xi}_{k-1}) < 0.
\end{equation}
This simple restart technique improves the convergence rate of Nesterov's acceleration without needing to tune $q$, i.e., $q$ can be set to 0.
O'Donoghue and Candès \cite{o2015adaptive} also propose a third, equally efficient method based on verifying whether or not the objective function is decreasing.
However, this method requires the objective function to be evaluated for each step.
Since we are already evaluating the gradient during each iteration, we choose to restart the momentum using the gradient verification.
Su, Boyd and Candès \cite{su2016differential} propose another criterion for the restart based on the increase of speed, i.e., restart if $||\bs{\xi}_{k}-\bs{\xi}_{k-1}|| < ||\bs{\xi}_{k-1} - \bs{\xi}_{k-2}||$; however, the gradient-based restart performs significantly better in their numerical evaluations.
Since we cannot observe the true gradient, we use the stochastic approximation of the gradient as the criterion to perform the restart, i.e.,
\begin{equation}
\mathcal{G}(\bs{\xi}_{k}, \bs{\theta}_k, \bs{Y}_k) \cdot (\bs{\xi}_{k} - \bs{\xi}_{k-1}) < 0,
\end{equation}
where $\mathcal{G}$ may be any of the estimators in Section \ref{sec:SGestimators}.

In Table \ref{tab:convergences}, we present the orders of the lower bounds for the optimality gap ($||\mathbb{E}[f(\bs{\xi}_k) - f(\bs{\xi}^*)]||$) for the full-gradient descent (FGD), AGD, SGD, and ASGD.
FGD uses the gradient of the expectation; therefore, in this respect, it is a deterministic optimizer.
\begin{table}[H]
 \centering
 \caption{Orders of lower bounds for the asymptotic convergence rate of the optimality gap.}
 \begin{tabular}{ccc}
  \toprule
  Method & Convex                                       & Strongly-convex                                                                     \\
  FGD    & $1/k$ \cite{nesterov2013introductory}        & $\left(\frac{L-\mu}{L+\mu}\right)^k$ \cite{nesterov2013introductory}                \\
  AGD    & $1/k^2$ \cite{nesterov2013introductory}      & $\exp \left(-\frac{k \sqrt{\mu}}{\sqrt{L}} \right)$ \cite{nesterov2013introductory} \\
  SGD    & $1/ \sqrt{k}$ \cite{nemirovski2005efficient} & $1/k$ \cite{nemirovski2005efficient}                                                \\
  ASGD   & $1/ \sqrt{k}$ \cite{cotter2011better}        & --                                                                                  \\ \bottomrule
 \end{tabular}
 \label{tab:convergences}
\end{table}

In the present work, we propose the rASGD optimizer and apply it to the OED problem in combination with the estimators presented in Section \ref{sec:SGestimators}, as we shall see in Section \ref{sec:ex}.
% We assume that the noise of the gradient observations is not large compared with the gradient norm; therefore, the convergence acceleration is not lost.

\section{Numerical examples} \label{sec:ex}

In this section, we evaluate the performance of the optimization methods described above by looking at four examples.
% The notation of optimization method with the estimation approach in sub-script means that we combine that optimization method using the gradient associated to the estimation approach, e.g., \SG{MC} stands for the stochastic gradient descent using the gradient estimator $\G{MC}$.

Our first example is the stochastic optimization of a stochastic quadratic function, unrelated to OED problems.
In the second example, we draw comparisons among the performances of \SG{MC}, \SG{LA}, and \SG{MCIS} using different optimization methods (SGD, ASGD, and rASGD).
In the third example, we address the optimization of strain gauge positioning on a beam, modeled following Timoshenko beam theory, in order to measure the beam's mechanical properties.
In the fourth and last example, we identify the optimal currents imposed on electrodes during an EIT experiment in order to maximize the expected information gain about ply orientations in a composite material.
% In the fourth and last example, we identify the optimal currents that maximize the expected information gain about ply orientations in a composite material when imposed on electrodes during an EIT experiment.

In all the examples, we denote the gradient estimator used as a subscript of the optimization method, e.g., \ASGD{LA} means that we are using the ASGD algorithm with \SG{LA} as a gradient estimator.

\subsection{Example 1: Stochastic quadratic function}\label{sec:ex1}

In this first example, we evaluate the performance of stochastic optimization algorithms on finding the maximum of a quadratic function, bearing in mind that this example does not involve the Bayesian framework.
Therefore, the stochastic gradient estimators of the expected information gain, \SG{MC}, \SG{LA}, and \SG{MCIS}, are not required.
Since the problem has a closed-form solution with known optimum and derivatives, we can use the same optimal $q$ tuning as Nemirovksi \cite{nemirovski2005efficient} and compare its effect with the restart technique.

We analyze the problem of finding $\bs{\xi}$ that maximizes the expected value of a function $f(\bs{\xi}, \bs{\theta})$ with respect to $\bs{\theta}$ given as
\begin{equation}
f(\bs{\xi}, \bs{\theta}) = -\left(\frac{1}{2} \bs{\xi} \cdot \bs{A} \cdot \bs{\xi} + \bs{\xi} \cdot \bs{A} \cdot \bs{\theta} \right),
\label{eq:Quadratic}
\end{equation}
where $\bs{A}$ is a diagonal $n\times n$ matrix with elements $A_{jj} = j$ for $j =1,\cdots,n$.
The random parameters vector $\bs{\theta}: \Omega^n  \mapsto \Theta \subset \mathbb{R}^{n}$ is Gaussian-distributed with zero mean and covariance matrix $\bs{\Sigma}_{\bs{\theta}} = \text{diag}(\{\sigma_{\theta}^2\}_{i=1}^n)$.
The vector $\bs{\xi}$ is a design variable, belonging to $\Xi$, a subset of $\mathbb{R}^n$.
The objective function to be maximized is
\begin{align}
\mathbb{E}[f(\bs{\xi}, \bs{\theta})] &= -\mathbb{E}_{\theta}\left[\frac{1}{2} \bs{\xi} \cdot \bs{A} \cdot \bs{\xi} + \bs{\xi} \cdot \bs{A} \cdot \bs{\theta}\right] \\
&= -\frac{1}{2} \bs{\xi} \cdot \bs{A} \cdot \bs{\xi},
\end{align}
which has optimum $\bs{\xi}^*$ = $\bs{0}$.
The stochastic gradient $\nabla_{\bs{\xi}}  \mathbb{E}_{\theta}\left[ f(\bs{\xi}, \bs{\theta})\right]$ is $\mathcal{G}(\bs{\xi}, \bs{\theta})= -\bs{A} \cdot (\bs{\xi} + \bs{\theta})$.
Hence, for this problem, SGD \eqref{eq:sgd} becomes
\begin{align}
\bs{\xi}_{k+1} &= \bs{\xi}_k + \alpha_k \mathcal{G}(\bs{\xi}_k, \bs{\theta}_k)\\
							 &= \bs{\xi}_k - \alpha_k  \bs{A} \cdot (\bs{\xi}_k + \bs{\theta}_k).
\end{align}
The Nesterov formulation is obtained by substituting $\mathcal{G}$ in \eqref{eq:ASGD}.
Note that $\nabla_{\bs{\xi}}  \mathbb{E}_{\bs{\theta}}\left[ f(\bs{\xi}, \bs{\theta})\right] = \mathbb{E}_{\bs{\theta}}\left[\mathcal{G}(\bs{\xi}, \bs{\theta})\right]$; thus, $\mathcal{G}$ is an unbiased estimator for the gradient of the objective function.
Since $\bs{A}$ is diagonal with elements $a_{ii} = i$, the variance of the $i^{th}$ element of the estimator $\mathcal{G}$ is calculated as
\begin{align}
	\mathbb{V}[\mathcal{G}_i(\bs{\xi}, \bs{\theta})] &= i^2 \mathbb{V} [\theta_i] \label{eq:ex1var}\\
	&= i^2 \sigma_\theta^2.
\end{align}
The variance of the gradient estimation does not depend on $\bs{\xi}$ and does not vanish in the optimum.
Thus, as $\nabla_{\bs{\xi}} \mathbb{E}_{\bs{\theta}}\left[f(\bs{\xi}, \bs{\theta})\right]$ converges to zero, the relative error in gradient estimation goes to infinity.

To solve this numerical example, we opt to use $n=20$.
The estimation of the conditioning number $L/\mu$ is straightforward in this case, since the Hessian of the objective function is constant and equal to $-\bs{A}$.
The largest eigenvalue of $\bs{A}$ is $L=20$, while the smallest is $\mu=1$.
Therefore, the optimal value for the parameter $q$ is $q^*=1/20$.
Similarly, the step-size is set to $\alpha_0 = 2/(L+\mu) = 2/21$.

Figure \ref{fig:QuadraticStochastic01Conv} presents the convergence of the distance to the optimum for each method using different standard deviations for the prior pdf $\pi(\theta)$; on the left, $\sigma_{\theta}=0.1$, and on the right, $\sigma_{\theta}=0.01$.
% The SGD method is slower than the other methods, but it exhibits a monotonic decay up to a certain iteration.
% Conversely, as discussed in Section \ref{sec:restart}, Nesterov's acceleration imposes an excessive momentum that generates oscillations over the optimum, which can be corrected by tuning $q$ to be the inverse of the conditioning number of the Hessian (i.e., $q^*$).
% The restart technique achieves the same convergence as ASGD using $q^*$, but without the need for any prior knowledge about the Hessian of the objective function.
The ASGD with restart technique converges faster than ASGD with optimum $q^*$.
\begin{figure}[ht]
	\centering
  \begin{subfigure}[t]{0.48\textwidth}
    \centering
	\includegraphics[width=\textwidth]{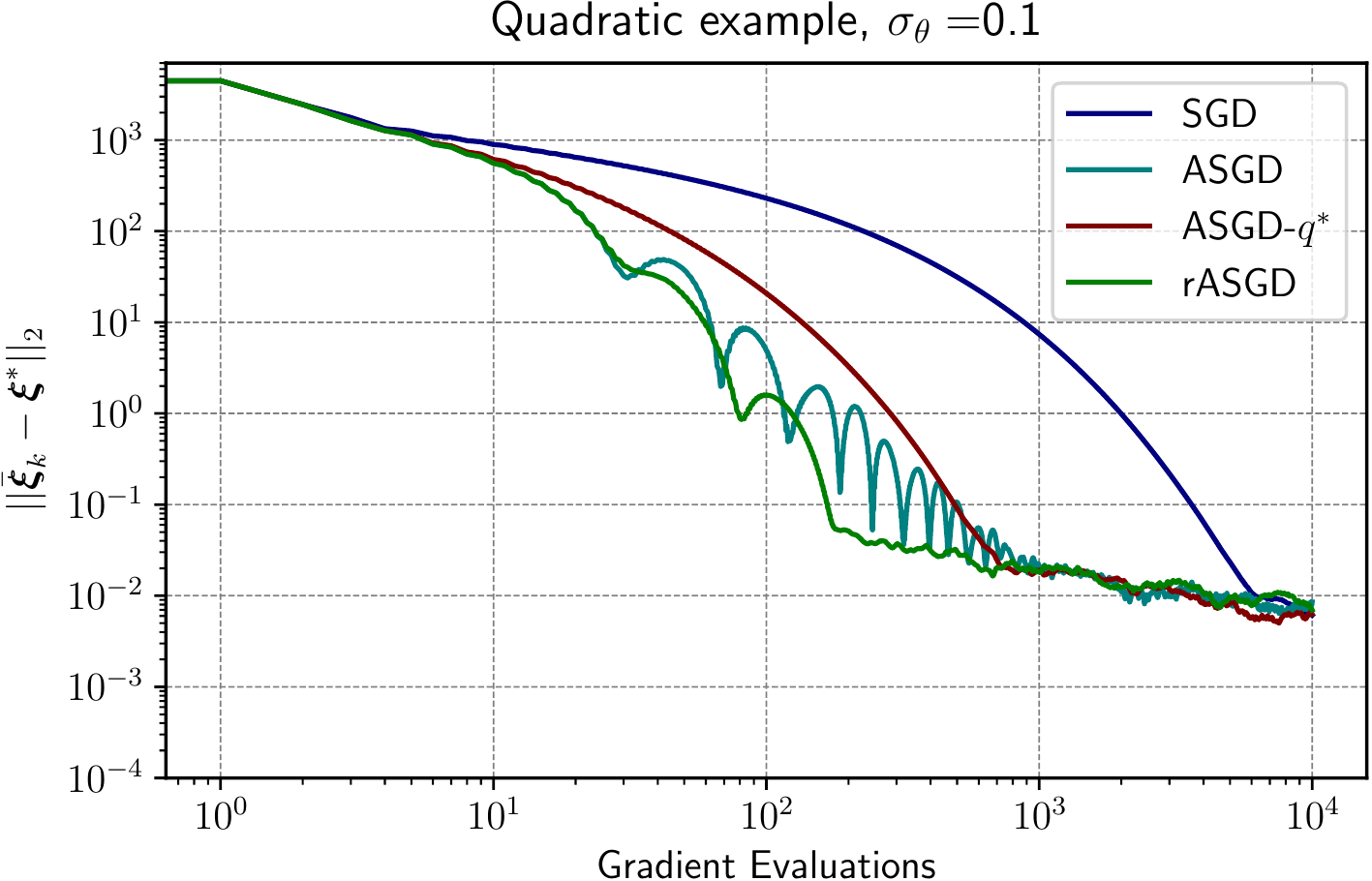}
  \caption{}
\end{subfigure}
\begin{subfigure}[t]{0.48\textwidth}
  \centering
  \includegraphics[width=\textwidth]{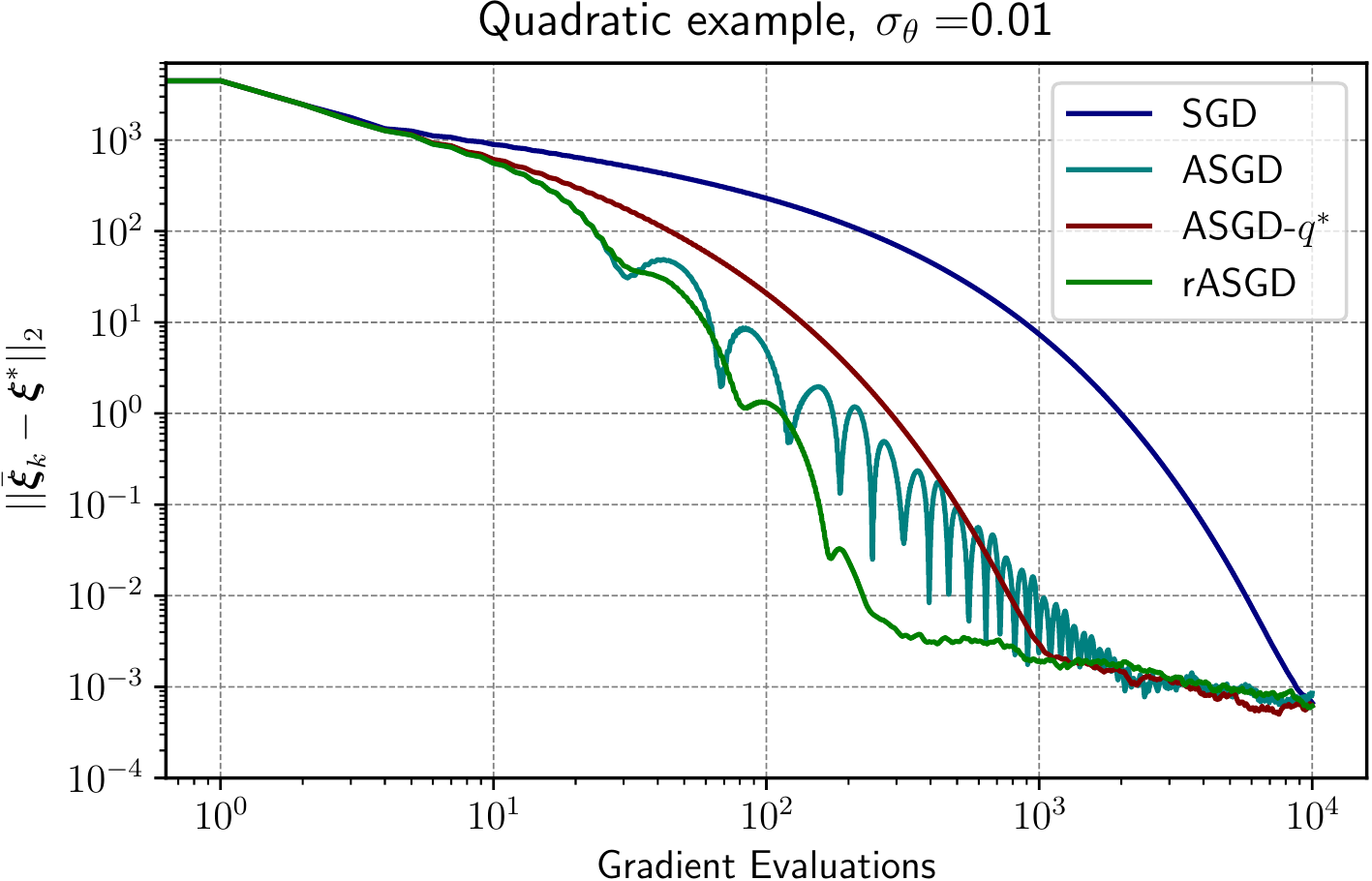}
  \caption{}
\end{subfigure}
	\caption{(Example 1): Convergence of the methods with standard deviations $\sigma_{\theta}=0.1$ (a) and $\sigma_{\theta}=0.01$(b).}
	\label{fig:QuadraticStochastic01Conv}
\end{figure}

As discussed in Section \ref{sec:SGD}, Figure \ref{fig:QuadraticStochastic01Conv} shows that the algorithms behave similarly to their deterministic counterparts up to a certain point.
As the noise in the gradient estimation becomes large in comparison to its magnitude, the convergence becomes dominated by the sublinear convergence of stochastic gradient methods.
Moreover, when the variance $\sigma_{\theta}$ is increased to $0.1$, Figure \ref{fig:QuadraticStochastic01Conv} shows that the asymptotic phase starts sooner.
In ASGD, the Nesterov acceleration imposes an excessive momentum that generates oscillations over the optimum, as discussed in Section \ref{sec:restart}.
For this example, the optimal tuning of $q$ does not improve on ASGD; however, the restart technique speeds up the convergence without the need for any prior knowledge about the Hessian of the objective function.
Figure \ref{fig:QuadraticStochastic01Conv} also shows that rASGD achieves the asymptotic phase at around $300$ gradient evaluations, whereas SGD takes almost $10000$ gradients to get to the asymptotic phase.

\subsection{Example 2: OED with quadratic model}\label{sec:ex2}

Here, we consider an OED problem based on a quadratic forward model that we devised to perform a comparative analysis of the stochastic gradients of expected information gain estimators.
We also test different combinations of these estimators with the optimization methods presented in this study.
Since $q^*$ is difficult to estimate and ASGD-$q^*$ did not perform well, we focus on FGD, SGD, ASGD, and rASGD.

The forward model is
\begin{equation}
	g(\bs{\xi}, \theta) =  \bs{\xi} \cdot \bs{A} \cdot \bs{\xi}~\theta -\bs{\xi} \cdot \bs{A} \cdot \bs{1} \theta^2 - 8 \bs{1} \theta - 1,\quad \hbox{where} \quad
	\bs{A} =
	\begin{bmatrix}
	1    & -0.2 \\
	-0.2 & 0.5
	\end{bmatrix},
\label{eq:Convex}
\end{equation}
where the scalar random variable $\theta$ is sampled from the prior pdf $\pi(\theta) = \mathcal{N}(0, 10^{-4})$, and  $\bs{\xi} \in \Xi= [-2,2]^2 \subset \mathbb{R}^2$.
The observation $y$ is
\begin{align}
y(\bs{\xi}, \theta) = g(\bs{\xi}, \theta) + \epsilon,
\end{align}
where the additive error is assumed to be Gaussian $\epsilon \sim \mathcal{N}(0, 10^{-4})$ and the number of experiments is $N_e = 1$. The initial step-size is $\alpha_0 = 1.00$.

\subsubsection{Comparison between the methods}

In this numerical test, we evaluate the performance of gradients of the expected information gain estimators (DLMC, MCLA, and DLMCIS) and their combination with optimization methods (FGD, SGD, ASGD, and rASGD).
For SGD, ASGD, and rASGD, the stochastic gradient estimators are used (\SG{MC}, \SG{LA}, and \SG{MCIS}), whereas for FGD, we use full gradients of DLMC, MCLA, and DLMCIS.
The efficiency criterion we use to compare different methods is defined as the average number of calls of the forward model (NCFM) required to approximate $\bs{\xi}^*$ for a given tolerance.
We compute the NCFM as the mean value of ten independent runs (due to the randomness of stochastic gradient methods), where we aim for an error tolerance of $0.01$, i.e., $\left\| \bs{\xi}_k - \bs{\xi}^*\right\|_2 \leq 0.01$.

To define the sample sizes for DLMC, MCLA, DLMCIS, we use the optimal sampling from Beck et al. \cite{Beck2017fast}, which we evaluate at the starting point of the optimization and keep constant during the process.
To achieve the tolerance of $0.01$ in the FGD, the optimal numbers of MC samples are  $(N^*, M^*)=(2447, 80)$ for DLMC, $(N^*,M^*)=(2402, 7)$ for DLMCIS, and $N^*=966$ for MCLA.
We use the same values of $M^*$ for the respective stochastic gradient estimators.
By adopting the forward Euler method, we compute the gradients of the model with respect to $\bs{\xi}$ using $3$ ($\dim{(\bs{\xi})}+1=3$) NCFM.
We use the Nelder-Mead algorithm \cite{nelder1965simplex} to estimate $\hat \theta$ in \eqref{eq:thetahat} for DLMCIS.

Table \ref{tab:ConvexCosts} presents the mean NCFM for different combinations of the optimization methods and gradient estimators.
The optimization methods are indicated at the top of each column, and the gradient estimators in Section \ref{sec:EIGestimators} are listed by row.
\begin{table}[H]
 \centering
 \caption{Mean NCFM over the ten runs required for the estimation of $\bs{\xi}^*$ for $\left\| \bs{\xi}_k - \bs{\xi}^*\right\|_2 \leq 0.01$.}
 \begin{tabular}{lclccc}
  \toprule
  \multicolumn{2}{c}{Full gradient} &
  \multicolumn{4}{c}{Stochastic gradient}\\
  \cmidrule(lr){1-2}
  \cmidrule(lr){3-6}
  Estim.                            & FGD                & Estim.    & SGD                & ASGD               & rASGD                       \\ \midrule
  $\nabla \I{DLMC}$                 & $2.99 \times 10^7$ & \SG{MC}   & $1.68 \times 10^5$ & $9.94 \times 10^3$ & $1.18 \times 10^4$          \\
  $\nabla \I{DLMCIS}$               & $6.57 \times 10^6$ & \SG{MCIS} & $3.18 \times 10^4$ & $3.17 \times 10^3$ & $2.56 \times 10^3$          \\
  $\nabla \I{MCLA}$                 & $2.80 \times 10^5$ & \SG{LA}   & $4.06 \times 10^3$ & $2.87 \times 10^2$ & $\mathbf{2.75 \times 10^2}$ \\ \bottomrule
 \end{tabular}
 \label{tab:ConvexCosts}
\end{table}

By analyzing the first line of Table \ref{tab:ConvexCosts}, we see that the two methods using Nesterov's acceleration (ASGD and rASGD) reduce the computational burden by three to four orders of magnitude compared to FGD.
Moreover, \rASGD{LA} estimates $\bs{\xi}^*$ in fewer than $1000$ calls of the forward model.

\subsubsection{Comparison between \SG{LA} and \SG{MCIS}}

Here, we compare the performance of \rASGD{LA} and \rASGD{MCIS} by testing the \SG{MCIS} estimator setting with variable sample sizes for the inner loop.
Figure \ref{fig:QuadOED_IS_ASGD} shows the contour of $I(\bs{\xi})$, approximated by MCLA, and the optimization path for \rASGD{LA} and \rASGD{MCIS} using the fixed cost of 1000 model evaluations.
Due to the lower cost of \SG{LA}, the optimization using this estimator is able to get closer to the optimum than using \SG{MCIS}.
However, \SG{MCIS} is able to converge even for $M=1$.
\begin{figure}[H]
\centering
\includegraphics[width=0.6\linewidth]{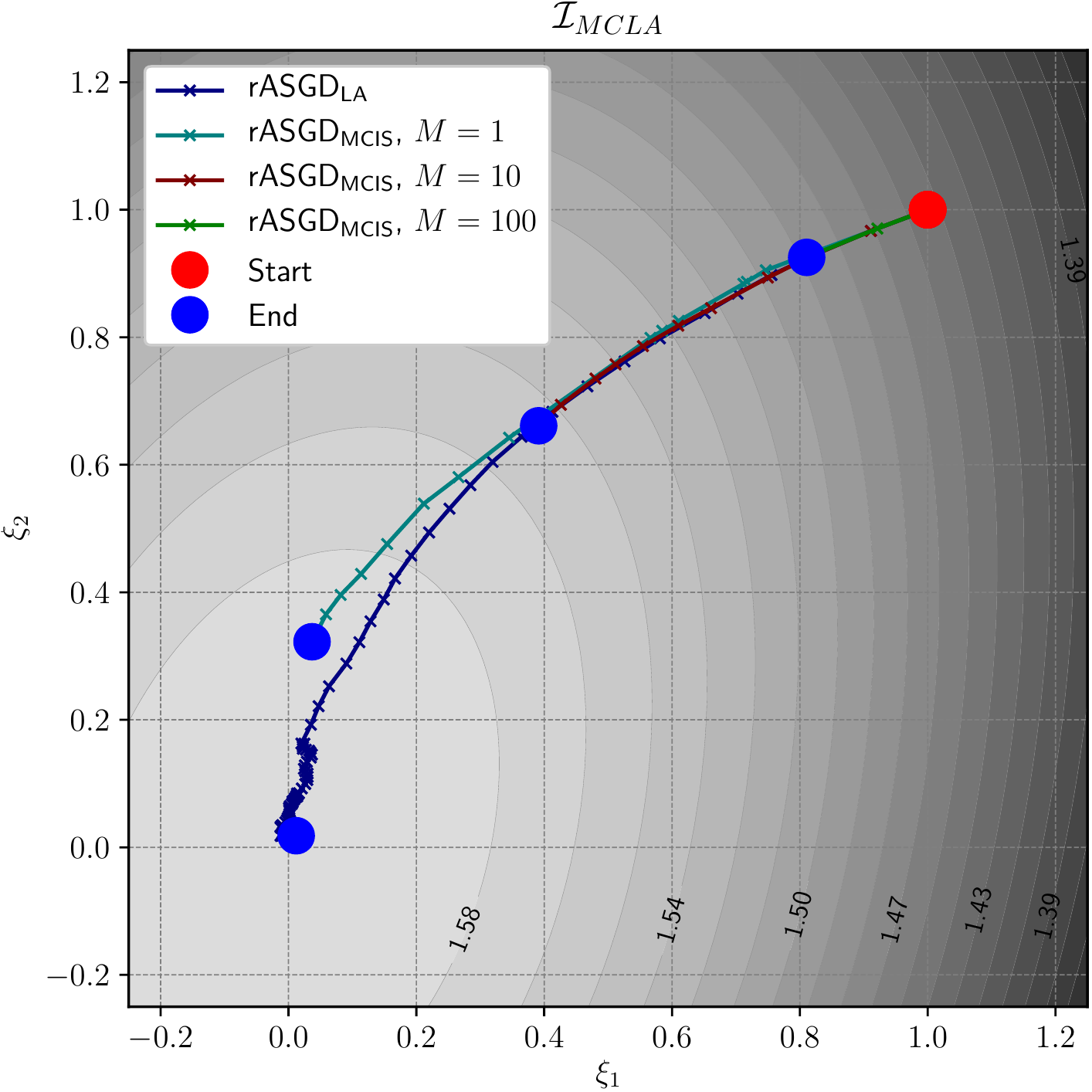}
\caption{(Example 2) Contour of the expected information gain and optimization ascent paths for \rASGD{LA} and \rASGD{MCIS}.}
\label{fig:QuadOED_IS_ASGD}
\end{figure}

Figures \ref{fig:QuadOED_IS_ASGDConv} and \ref{fig:QuadOED_IS_ASGDByCost} present the convergence history of the error in terms of $\bs{\xi}$ \emph{versus} the number of iterations and NCFM, respectively.
In Figure \ref{fig:QuadOED_IS_ASGDConv}, we see that \rASGD{LA} performed almost 175 iterations, whereas \rASGD{MCIS} with $M=1$ did not achieve 25 iterations.
Increasing the size of $M$ did not improve the convergence by much. The acceleration in convergence was not sufficient to compensate for the increase in cost.
\begin{figure}[H]
  \centering
  \includegraphics[width=.6\linewidth]{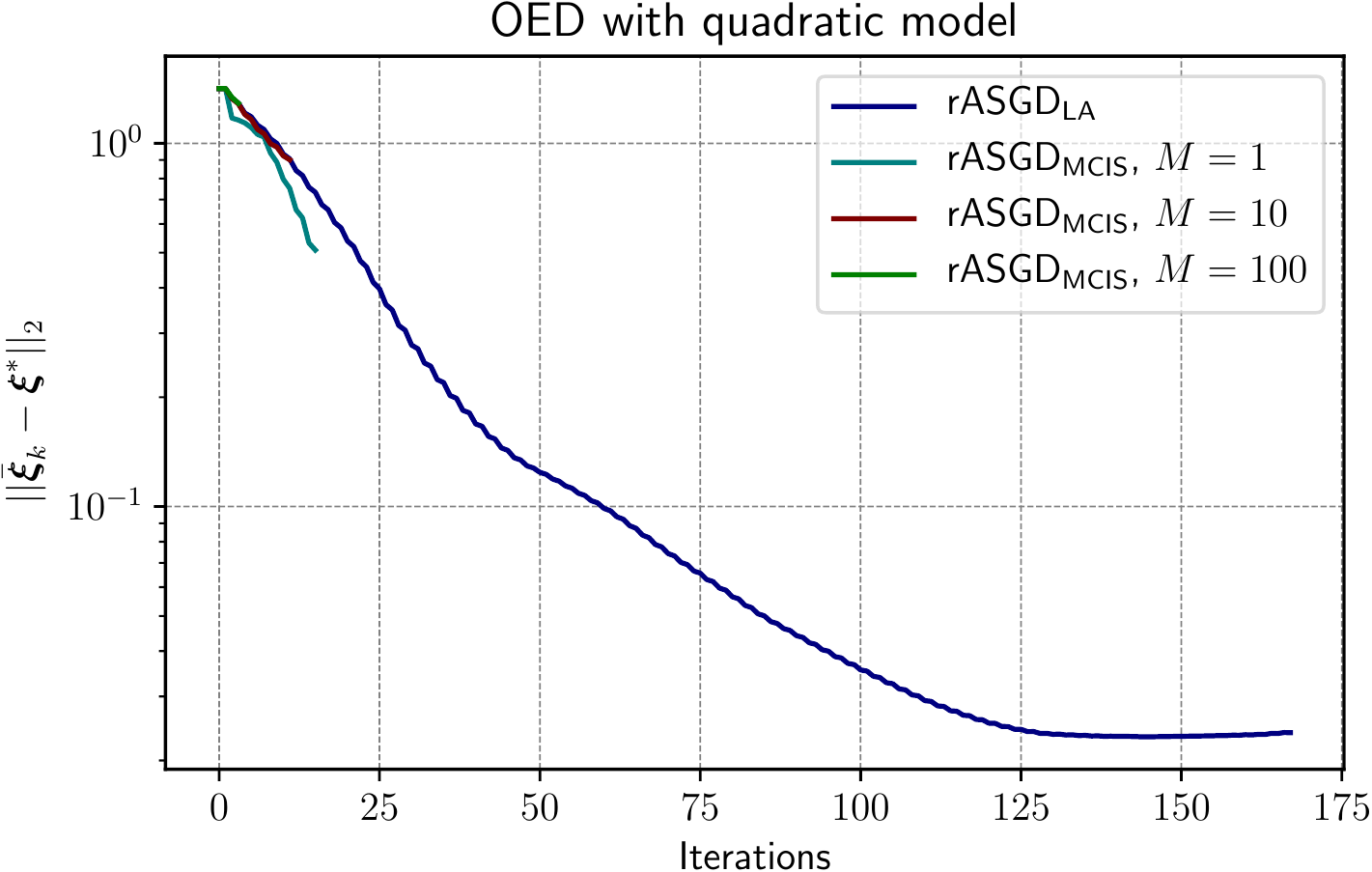}
  \caption{(Example 2) Convergence to the optimum in relation to iterations for \rASGD{LA} and \rASGD{MCIS}.}
  \label{fig:QuadOED_IS_ASGDConv}
\end{figure}
\begin{figure}[H]
  \centering
  \includegraphics[width=.6\linewidth]{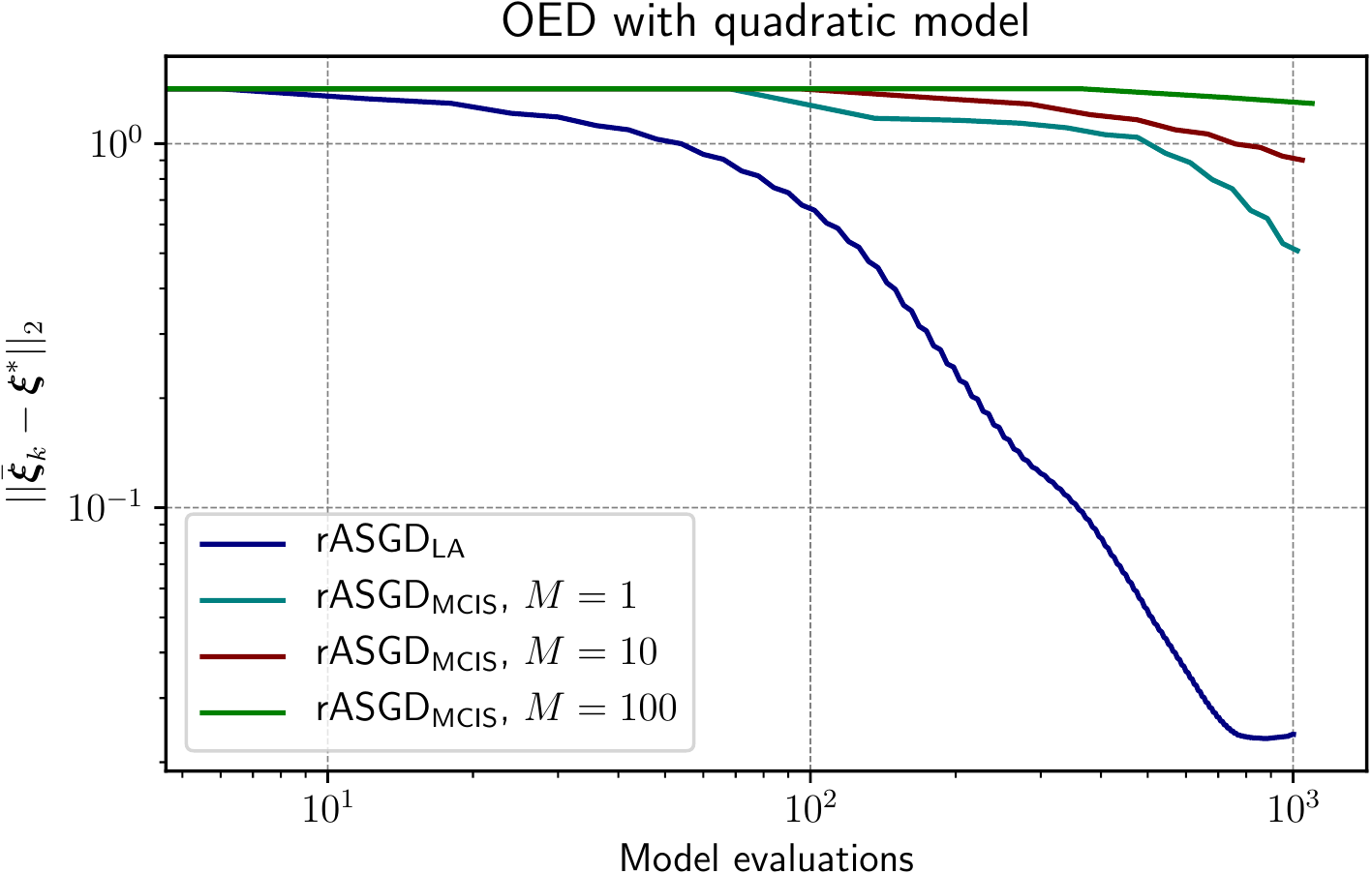}
  \caption{(Example 2) Convergence to the optimum in relation to model evaluations for \rASGD{LA} and \rASGD{MCIS}.}
  \label{fig:QuadOED_IS_ASGDByCost}
\end{figure}
As a sanity check to estimate the intrinsic bias of the Laplace approximation in the optimization carried out with the estimator \SG{LA}, we compute the expected value of the gradient using DLMCIS at the optimum found.
Using $N=10^5$ and $M=10^3$ in DLMCIS, we obtain a gradient with a norm of $10^{-6}$, which means that the bias introduced by the Laplace approximation is negligible in this case.

\subsection{Example 3: Strain gauge positioning on Timoshenko beam}\label{sec:ex3}

In this example, we look at a beam with the dimensions $10$ m length, $2$ m height, and $0.1$ m base width.
A uniform load of $1.00$ kN/mm is imposed on the beam's vertical axis and distributed along its main axis.
We characterize the beam's mechanical properties, namely the Young modulus $E$ and the shear modulus $G$, given measurements obtained from a strain gauge.
The geometry of the beam, the load, and the position of the strain gauge are illustrated in Figure 6.
\begin{figure}[H]
	\centering
	\includegraphics[width=0.4\textwidth]{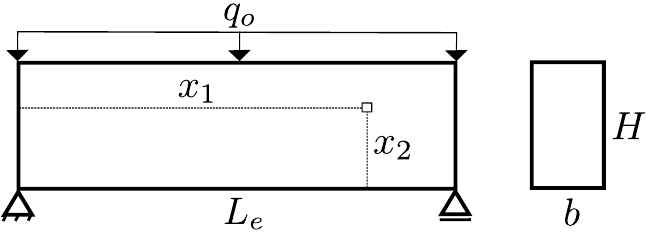}
	\caption{(Example 3) Geometry of the Timoshenko beam.}
	\label{fig:TimoGeometry}
\end{figure}
We aim to locate a strain gauge on the beam that maximizes the information on $E$ and $G$.
We model the beam following Timoshenko's theory \cite{timoshenko1921lxvi}, a mechanical model that captures the strains resulting from both normal and shear stresses.
The Timoshenko beam model is
\begin{align}\label{eq:timo}
\begin{cases}
K_s G A_r \varepsilon_{12} = \frac{q_o L_e}{2} - q_o x_1, \\
EI_n \varepsilon_{11} = \frac{q_o x_1(L_e-x_1)}{2} x_2,
\end{cases}
\end{align}
where $\varepsilon_{11}$ is the normal strain, $\varepsilon_{12}$ is the shear strain, $x_1$ and $x_2$ are the positions of the strain gauge on the horizontal and vertical axes, respectively, $q_o$ is the uniform load, $L_e$ is the length of the beam, $I_n$ is the inertia moment of the cross section, $K_s$ is the Timoshenko constant ($K_s=5/6$ in all test cases), and $A_r$ is the cross-section area.

\subsubsection{Bayesian formulation}

The optimal position for the strain gauge that provides the maximum information about E and G is denoted by $\bs{\xi}^*=(x_1^*,x_2^*)$.
The longitudinal strain on the main axis of the beam, denoted by $\varepsilon_{11}$, together with the transverse strain $\varepsilon_{12}$, composes the output of the forward model.
Therefore, based on \eqref{eq:timo}, we find that
\begin{align}
\bs{g}(\bs{\xi}, \bs{\theta})
&= \left(\varepsilon_{11}(\bs{\xi}, \bs{\theta}), \varepsilon_{12}(\bs{\xi}, \bs{\theta}) \right)\nonumber\\
&= \left(\frac{\xi_2\left(q_oL_e \xi_1 - q_o \xi_1^2 \right)}{2 \theta_1 I_n},
\frac{\frac{L_e}{2} q_o - q_o \xi_1}{K_s \theta_2 A_r}\right),
\label{eq:timo2}
\end{align}
where $(x_1,x_2)$ and $(E,G)$ are replaced by $(\xi_1,\xi_2)$ and $(\theta_1,\theta_2)$, respectively.
The additive error of the measurement is Gaussian $\bs{\epsilon} \sim \mathcal{N}(0,\bs{\Sigma}_{\epsilon})$, where the noise covariance matrix is $\bs{\Sigma}_{\epsilon} = \hbox{diag} \left\{ \sigma_{\epsilon_{1}}^2, \sigma_{\epsilon_{2}}^2 \right\}$.

\subsubsection{Test cases}

We assess the robustness of the proposed methods in four test cases, in which we
attempt to locate the optimal strain-gauge placement on a beam.
We test all the different cases, changing the variance of the prior pdf of $\bs{\theta}$, the dispersion of the measurement noise, and the number of experiments.
All four cases are tested with the \SG{LA} estimator, and the prior pdf of $\bs{\theta}$ is Gaussian with the distribution $\pi(\bs{\theta}) \sim \mathcal{N}\left((\mu_{pr}^E,\mu_{pr}^G)^T,\hbox{diag}\left\{(\sigma_{pr}^G)^2, (\sigma_{pr}^E)^2 \right\}\right)$, where $\mu_{pr}^E=30.00\text{ GPa}$ and $\mu_{pr}^G=11.54\text{ GPa}$.
Table \ref{tab:partimoshenko} presents the parameters used in each of the four cases.
\begin{table}[H]
\centering
\caption{Parameters for the Timoshenko beam problem (Example 3).}
\begin{tabular}{lccccc}
	\toprule
	Parameter &  $N_e$   &  $\sigma^E_{pr}$(GPa) &  $\sigma^G_{pr}$(GPa) & $\sigma_{\epsilon_{1}}(\times 10^{-4})$   &   $\sigma_{\epsilon_{2}}(\times 10^{-4})$   \\
	\midrule
	Case 1    &    3     &       9.00            &         3.46   &    6.25   &    1.30 \\
	Case 2    &    1     &       6.00            &         2.31   &    3.75   &    0.78 \\
	Case 3    &    1     &       6.00            &         0.46   &    3.75   &    0.78 \\
	Case 4    &    1     &       1.20            &         2.31   &    3.75   &    0.78 \\
 \bottomrule
\end{tabular}
\label{tab:partimoshenko}
\end{table}

In this section, we focus on \rASGD{LA}, and assess the bias using the expected value of the gradient of DLMCIS at the optimum.
The optimization paths for the placement of the strain gauges on the beam are drawn against contour plots of the expected information gain across the optimization domain in Figure \ref{fig:Timoshenko}.
\begin{figure}[hp]
	\centering
	\includegraphics[width=\linewidth, clip, trim={0 0 0 0.0cm}]{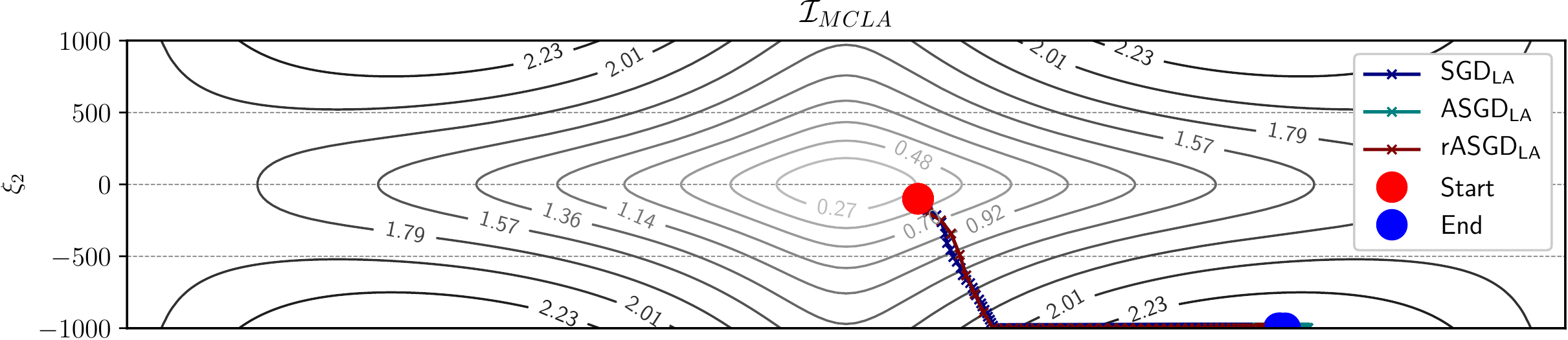}\\ \vspace{1em}
	\includegraphics[width=\linewidth]{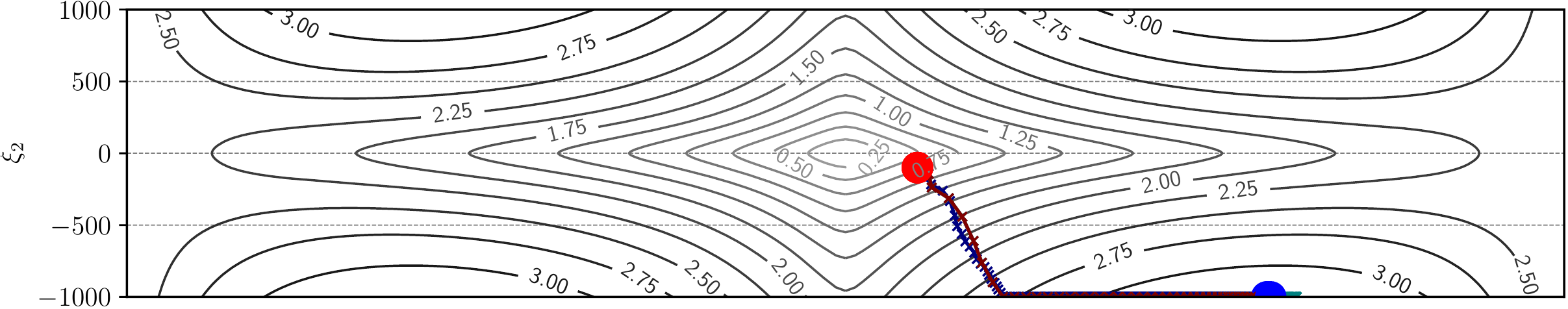}\\ \vspace{1em}
	\includegraphics[width=\linewidth]{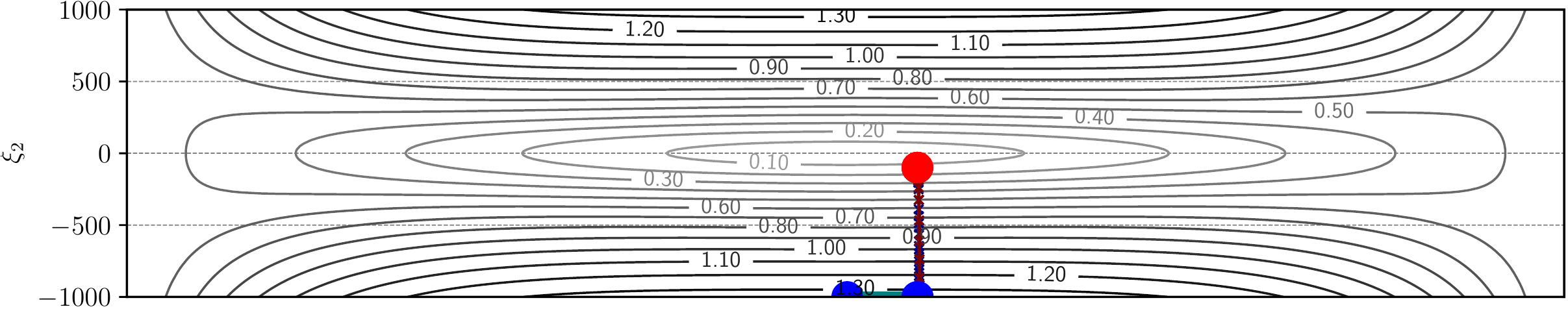}\\ \vspace{1em}
	\includegraphics[width=\linewidth]{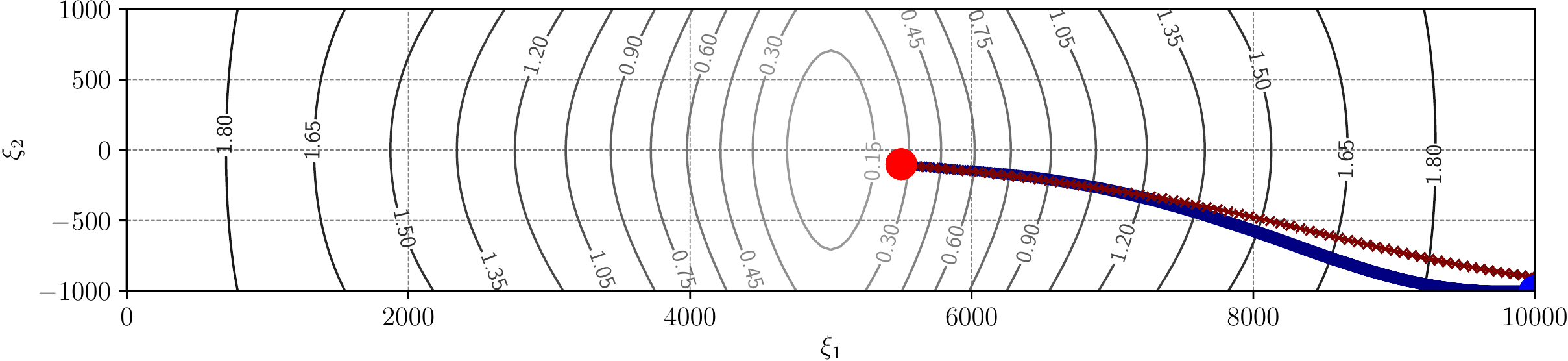}
	\caption{(Example 3) From top to bottom, cases 1 to 4  summarized in Table \ref{tab:resultstimoshenko}.
Expected information gain contours computed with MCLA and optimization ascent paths using \SGD{LA}, \ASGD{LA}, and \rASGD{LA}.}
	\label{fig:Timoshenko}
\end{figure}

In cases 1 and 2, the optima are similarly located near the bottom of the beam, between the middle and the end.
In case 3, the optimum is located in the bottom-middle of the beam; in case 4, the optimum is located on the supports.
These placements are expected, as the Young modulus depends on the bending moment (for which the maximum is at the middle of the beam ($x_1 = L/2$)), and the shear modulus depends on the shear stress (for which the maximum is at the beam supports ($x_1 = 0$ and $x_1 = L$)).
In case 3, the prior information about $G$ is more accurate; consequently, the algorithm converges to the middle of the beam where more information about $E$ can be collected.
Similarly, in case 4, the algorithm converges to the beam supports where data is more informative about $G$.

In Table \ref{tab:resultstimoshenko}, we present the initial guesses, the optimized setups, the respective expected information gains in relation to the prior, and the standard deviations of the posterior pdfs of the parameters $E$ and $G$ for the four cases.
The posteriors are evaluated at $\bs{\hat{\theta}} = (\mu_{pr}^E,\mu_{pr}^G)$ for the four cases presented in Figure \ref{fig:Timoshenkoposterior}.
We observe a reduced variance in the optimized experiment compared to the original, reflecting the importance of an informative experiment.
In cases 3 and 4, no information is acquired about $G$ and $E$, respectively, since the variances in the axes are not reduced compared to the prior.
\begin{table}[H]
 \centering
 \caption{Results from the Timoshenko beam problem (Example 3).}
 \begin{tabular}{lcccccc}
  \toprule
                          &          & $x_1^*$(mm) & $x_2^*$(mm) & $\I{MCLA}$ & $\sigma^E_{post}$ (GPa) & $\sigma^G_{post}$ (GPa) \\
  \cmidrule{3-7}
  \multirow{2}{*}{Case 1} & Non-Opt. & 5500.00     & -100        & 0.14               & 8.00                    & 2.40                    \\
                          & Opt.     & 8022.59     & -1000.00    & 2.43               & 2.48                    & 0.54                    \\ \midrule
  \multirow{2}{*}{Case 2} & Non-Opt. & 5500.00     & -100        & 0.23               & 2.38                    & 1.38                    \\
                          & Opt.     & 7962.77     & -1000.00    & 3.35               & 1.60                    & 0.74                    \\ \bottomrule
  \multirow{2}{*}{Case 3} & Non-Opt. & 5500.00     & -100        & 0.06               & 5.70                    & 0.46                    \\
                          & Opt.     & 5004.47     & -1000.00    & 1.28               & 1.72                    & 0.46                    \\ \bottomrule
  \multirow{2}{*}{Case 4} & Non-Opt. & 5500.00     & -100        & 0.22               & 1.20                    & 1.93                    \\
                          & Opt.     & 10000.00    & -1000.00    & 1.94               & 1.20                    & 0.33                    \\ \bottomrule
 \end{tabular}
 \label{tab:resultstimoshenko}
\end{table}
Because we use the biased and inconsistent \SG{LA} estimator of the gradient, as a sanity check, we evaluate the gradient at the optima we found (the first two cases), using the full gradient of the DLMCIS estimator with $N=10^3$ and $M=10^2$.
In both cases, the gradient norm is below $10^{-3}$, meaning that the bias of the Laplace approximation is considerably small at the optima.
We conclude that the biased optima are not significantly distant to the real optima.
To plot the convergence, we estimate the real optima using DLMCIS.
The convergences from the first two cases are presented in Figure \ref{fig:timoshenko1conv}.

\begin{figure}[h]
	\centering
	\includegraphics[width=0.48\linewidth]{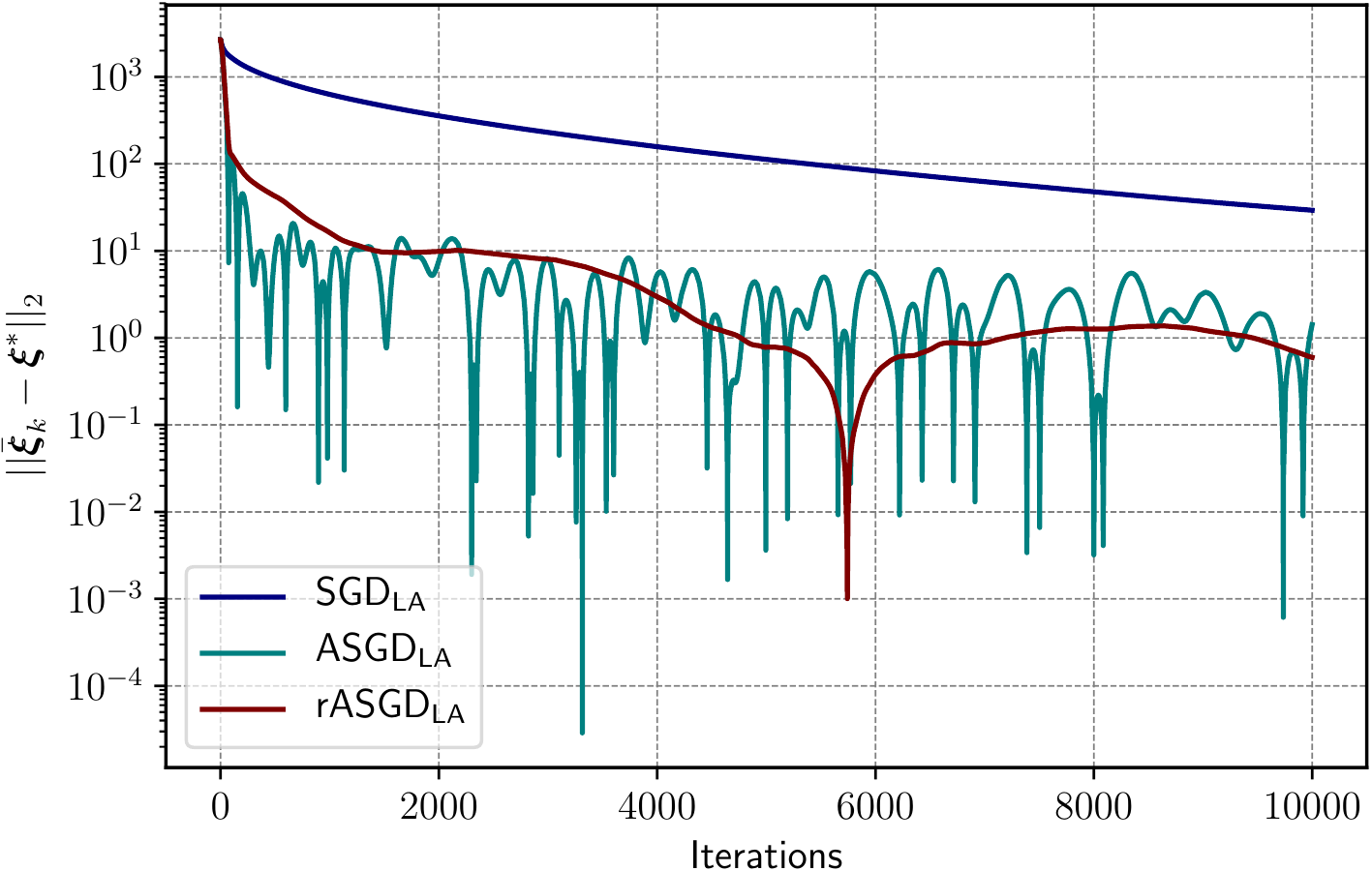}
	\includegraphics[width=0.48\linewidth]{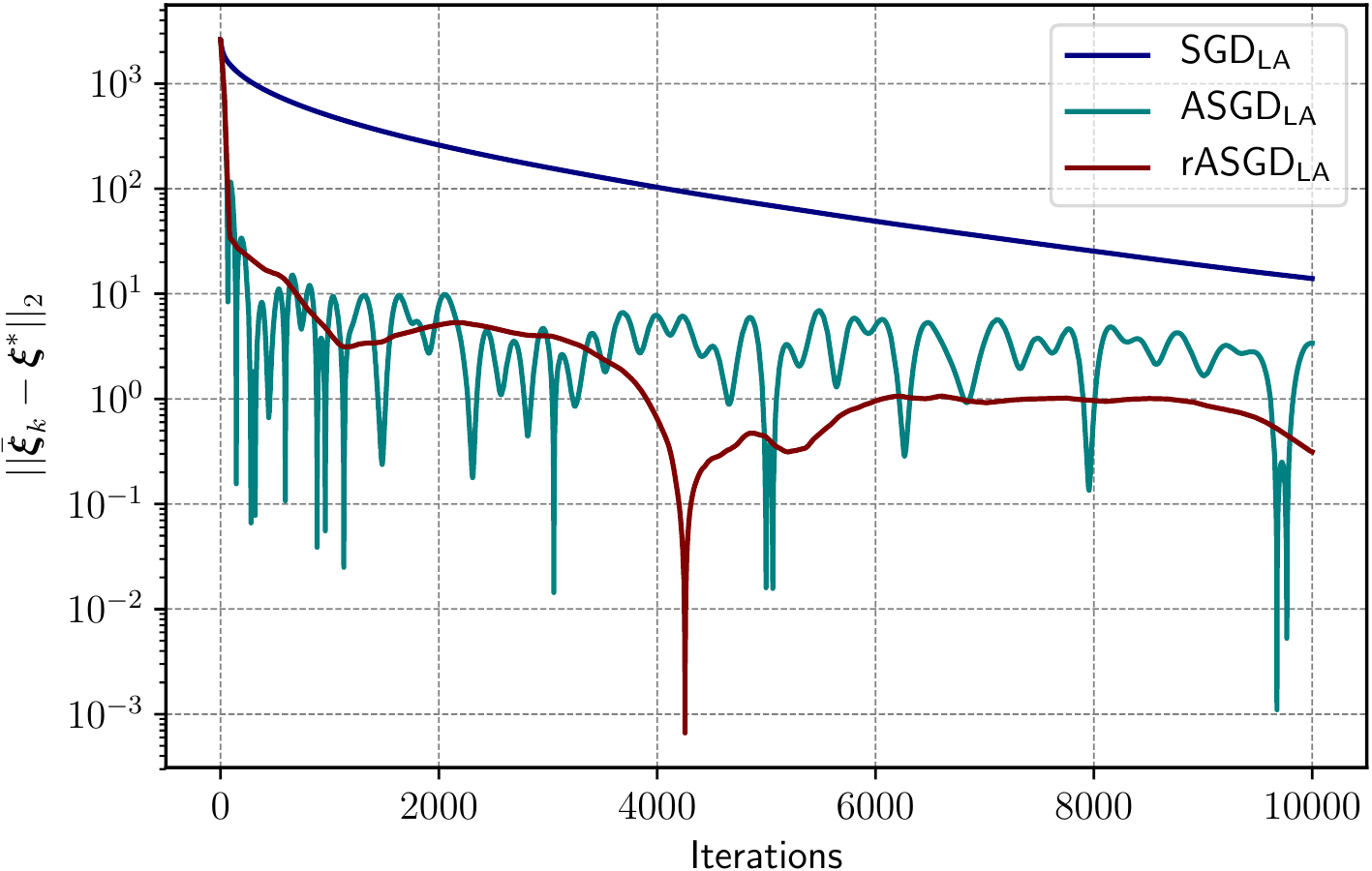}
	\caption{(Example 3) Convergences from cases 1 (left) and 2 (right) (\rASGD{LA}: error 1 mm, or relative error of $10^{-4}$).}
	\label{fig:timoshenko1conv}
\end{figure}

\begin{figure}[hp]
	\centering
	\includegraphics[width=0.48\textwidth]{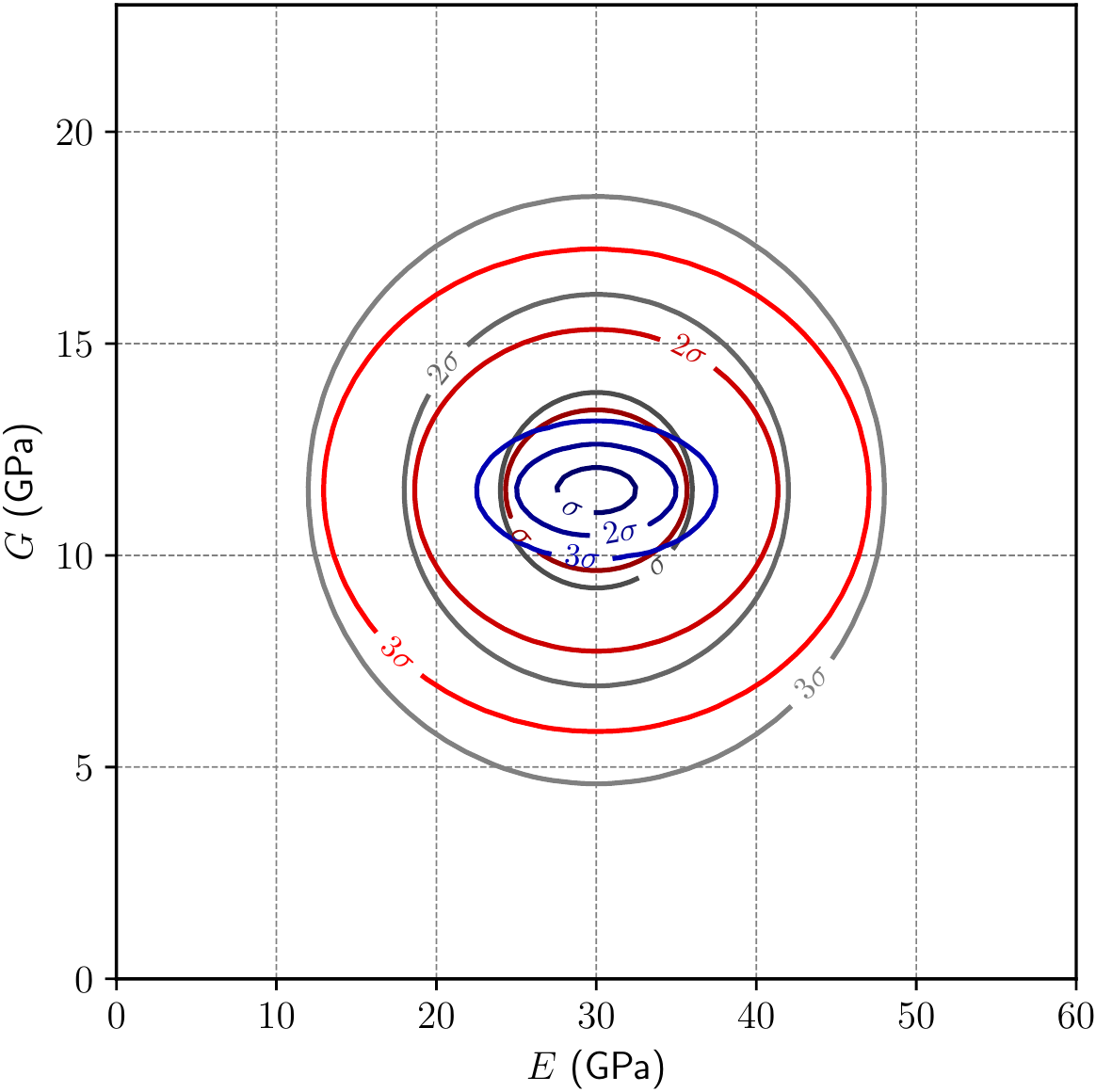}
	\includegraphics[width=0.48\textwidth]{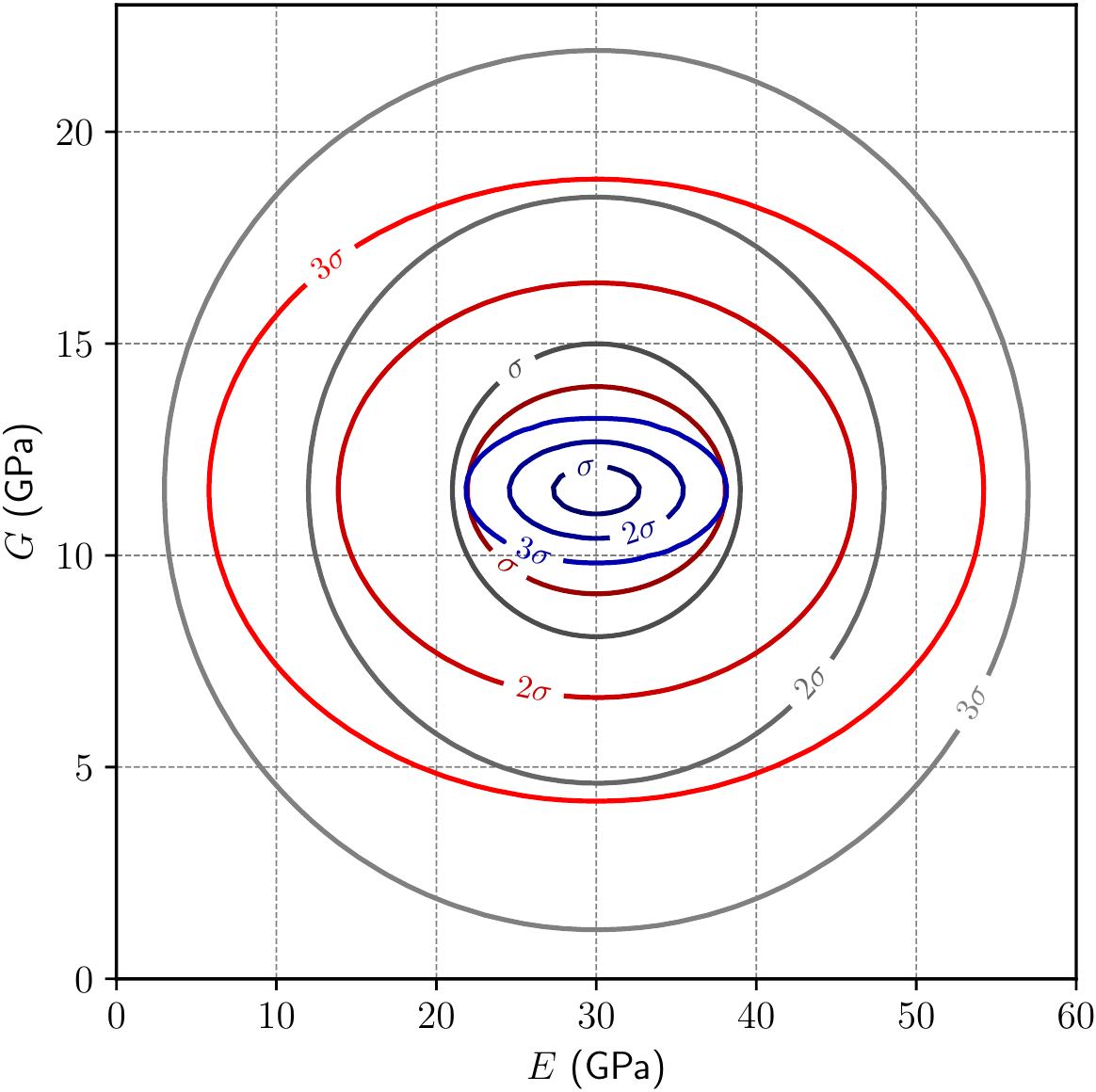}\\ \vspace{1em}
	\includegraphics[width=0.48\textwidth]{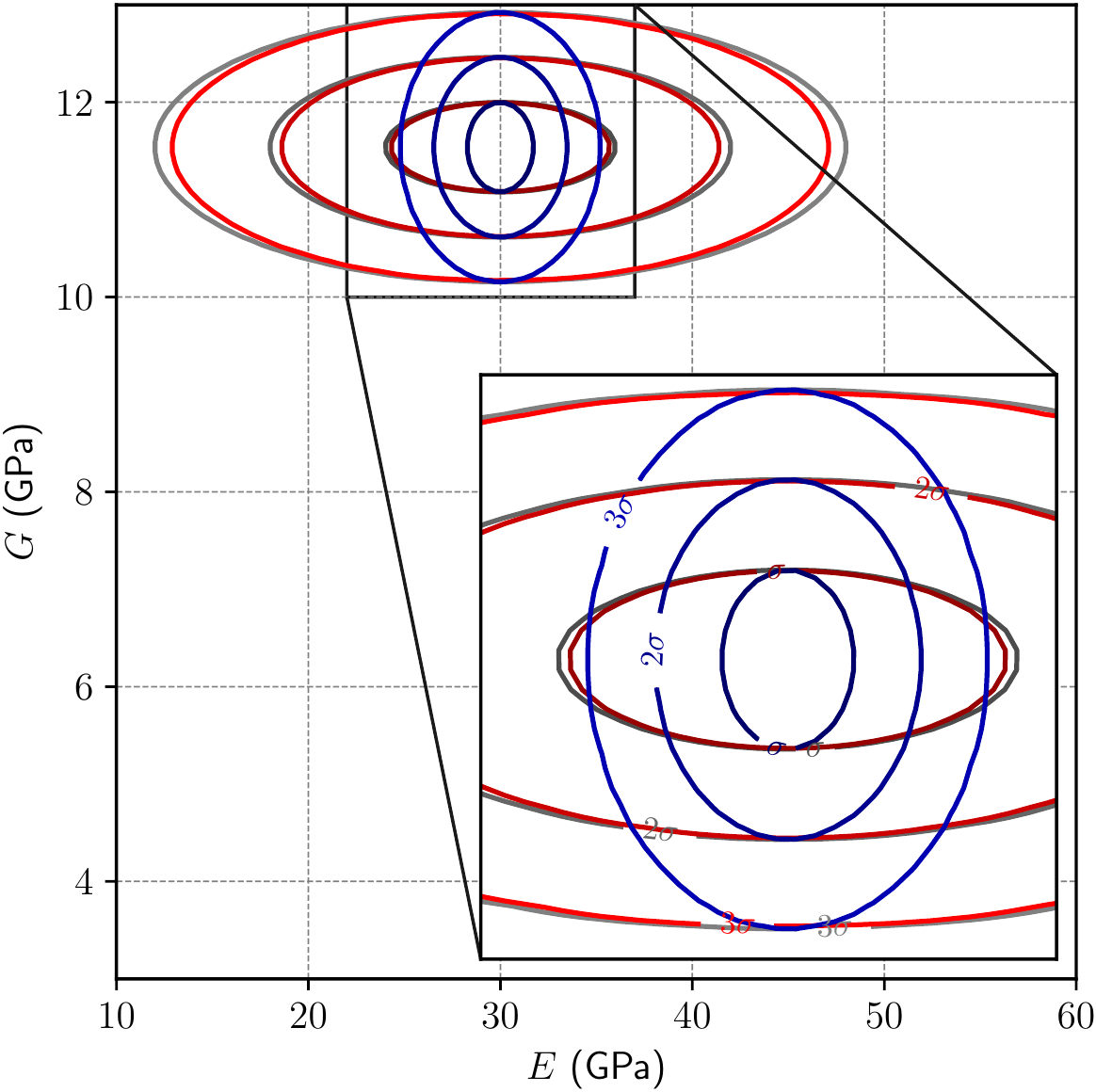}
	\includegraphics[width=0.48\textwidth]{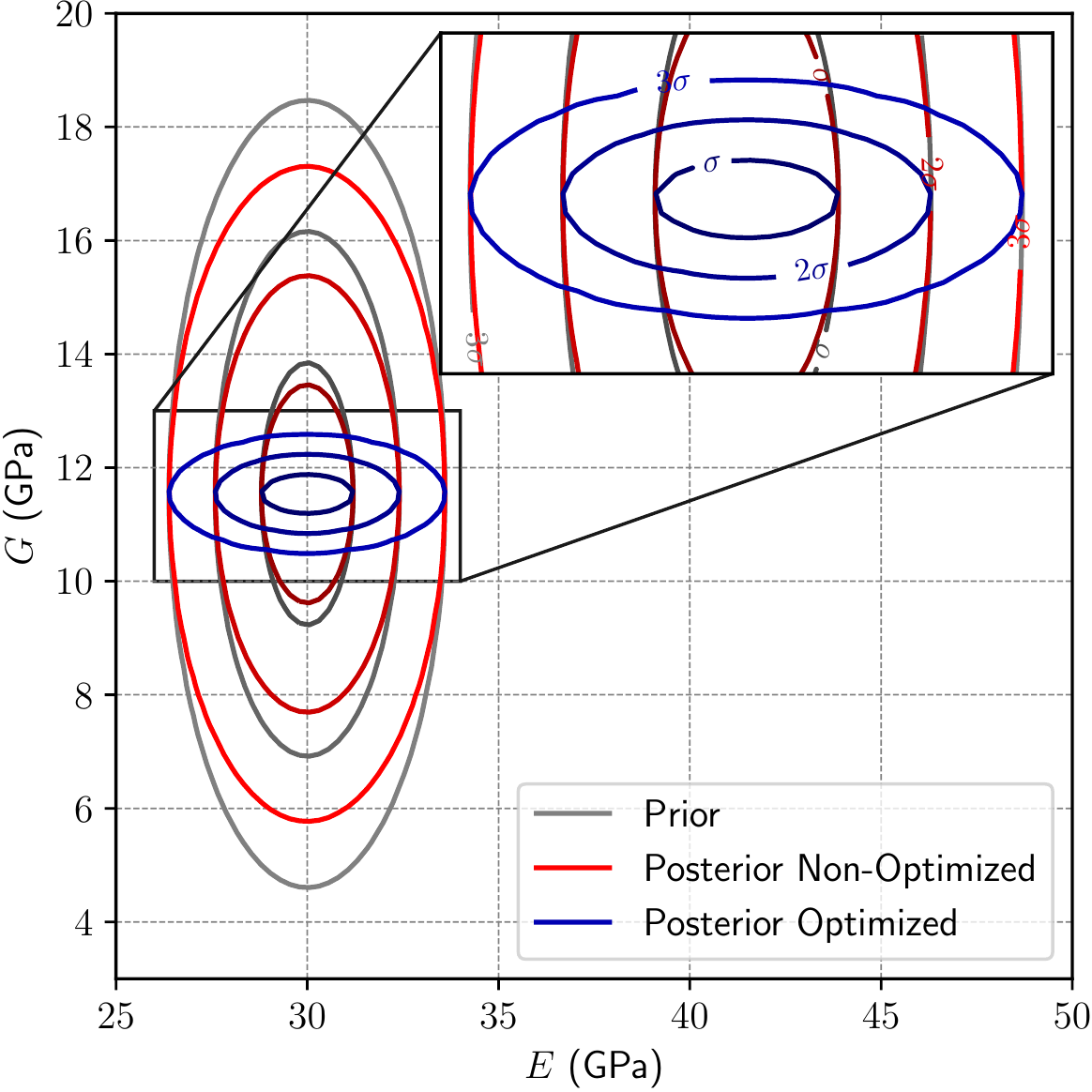}
	\caption{(Example 3) Prior, posterior, and optimized posterior pdfs for the Young modulus $E$ and the shear modulus $G$ for cases 1 (top-left), 2 (top-right), 3 (bottom-left), and 4 (bottom-right).}
	\label{fig:Timoshenkoposterior}
\end{figure}

\subsection{Example 4: Electrical impedance tomography}\label{sec:ex4}

EIT is an imaging technique that infers the conductivity of a closed body from potential measurements obtained from electrodes placed on the boundary surface of the body.
Here, we consider the optimal design of an EIT experiment conducted on two orthotropic plies, in which the potential field is assumed to be quasi-static.
The physical phenomenon is governed by a second-order partial differential equation combined with the complete electrode boundary model \cite{somersalo}.
Beck et al. \cite{Beck2017fast} prove that the bias of the Laplace approximation for this problem is negligible.
Therefore, in this example, we use \rASGD{LA}.

\subsubsection{Bayesian setting}

We consider a body $D$ that is $20$ cm long and composed of two plies that are each 1 cm thick, resulting in a total thickness of  2 cm.
Both plies are made of the same material, but are oriented at different angles.
The conductivity of each ply is $\bs{\bar{\sigma}}(\bs{\theta},\bs{x}) = \bs{Q}(\theta_k) \cdot \bs{\sigma} \cdot \bs{Q}(\theta_k)$, where $\bs{\sigma} = \hbox{diag}\left\{10^{-2}, 10^{-3}, 10^{-3} \right\}$, and $\bs{Q}(\theta_k)$ is
an orthogonal matrix depending on the unknown orientation angle $\theta_k$ that governs the rotation of ply $k$, counting from bottom to top.
The objective is to infer $\theta_1$ and $\theta_2$, about which we assume the prior information to be $\pi(\theta_1) \sim \mathcal{U}(\frac{\pi}{4.5}, \frac{\pi}{3.5})$ and $\pi(\theta_2) \sim \mathcal{U}(-\frac{\pi}{3.5}, -\frac{\pi}{4.5})$.
During the EIT experiment, low-frequency electrical currents are injected through the electrodes $E_l$ (with $l=1, \cdots, N_{el}$) attached to the boundary of the body, with $N_{el}$ being the number of electrodes.
The potentials at the electrodes are calculated as
\begin{equation}
\bs{y}_i(\bs{\xi}) = \bs{g}_h(\bs{\xi},\bs{\theta}_t) + \bs{\epsilon}_i \myeq \bs{U}_h(\bs{\xi},\bs{\theta}_t) +\bs{\epsilon}_i, \quad for \quad i = 1, \cdots, N_e \,,
\end{equation}
where $\bs{y}_i \in \mathbb{R}^{N_{el}-1}$, $\bs{\theta}_t=(\theta_{t,1},\theta_{t,2})$ are the true orientation angles that we intend to infer.
In the Bubnov--Galerkin sense, $\bs{U}_h = (U_1, \cdots, U_{N_{el}-1})$ is the finite elements approximation (i.e., the potential at the electrodes) of $\bs{U}$ from the following variational problem: find $\left(u,\bs{U}\right) \in  L^2_{\mathbb{P}} \left(\Theta;\mathcal{H} \right)$ such that
\begin{equation}
\mathbb{E} \left[B\left((u,\bs{U}),(v,\bs{V})\right) \right] = \bs{I}_e \cdot \mathbb{E} \left[ \bs{U} \right], \;\;\;\;\hbox{for all}\;\;\;(v,\bs{V}) \in L^2_{\mathbb{P}} \left(\Theta;\mathcal{H} \right),
\end{equation}
where $\bs{I}_e$ represents the values of injected current at $N_{el}-1$ electrodes, $\bs{I}_e = \left(I_{e_1}, \cdots, I_{e_{N_{el}-1}}\right)^T$.
Let the constitutive relation for the current flux be $\bs{\jmath}(\bs{\theta}, \bs{x}) = \bs{\bar{\sigma}}(\bs{\theta},\bs{x}) \cdot \nabla u(\bs{\theta},\bs{x})$.
Then, the bilinear form $B:\mathcal{H}\times\mathcal{H} \rightarrow \mathbb{R}$ is
\begin{equation}
B\left((u,\bs{U}),(v,\bs{V})\right) = \int_D \bs{\jmath} \cdot \nabla v d D + \sum_{l=1}^{N_{el}} \frac{1}{z_l} \int_{E_l} \left(U_l -u \right) \left(V_l - v \right) \text{d} E_l,
\end{equation}
where $z_l$ is the surface contact impedance between the electrode $l$ and the surface of the body.
The space of the solution for the potential field $(u(\theta),\bs{U}(\theta))$ is  $\mathcal{H} \myeq H^1(D) \times \mathbb{R}^{N_{el}}_{\text{free}}$ for a given random event $\theta \in \Theta$, where $H^1$ is the Sobolev space of functions that belong to $L^2$, and whose first-order partial derivatives also belong to $L^2$.
Then, $L^2_{\mathbb{P}} \left(\Theta;\mathcal{H} \right)$ is  the Bochner space given by
\begin{equation}
L^2_{\mathbb{P}} \left(\Theta;\mathcal{H} \right) \myeq \left\{ (u,\bs{U}): \Theta \rightarrow \mathcal{H} \;\;\;\hbox{s.t.}\;\; \int_{\Theta} \left\|(u(\bs{\theta}),\bs{U}(\bs{\theta}))\right\|^2_{\mathcal{H}} d \mathbb{P}(\bs{\theta}) < \infty \right\}.
\end{equation}
The measurement-error distribution is $\bs{\epsilon} \sim \mathcal{N}(0,100.0)$, i.e., the standard deviation of the noise is around $5\%$ of the magnitude of measured potential.
We note that, by imposing the Kirchhoff law on $\bs{I}_e$ and the zero-potential law on $\bs{U}_h$, the model output $\bs{g}$ is projected to a suitable space for the optimization.

The optimization parameters are defined as the current intensity to be injected through the electrodes, i.e., $\bs{\xi} = \left(\{I_{e}\}_{i=1}^{N_{el}}\right)$, where each $I_{e}$ is the normalized current intensity applied to the $i^{th}$ electrode such that $I_{e} \in [-1,1]$.
A schematic of the experimental setup showing the laminated material with four electrodes is depicted in Figure \ref{fig:EITsetup}., which shows the composite material with four electrodes.
\begin{figure}[H]
\centering
\includegraphics[width=1\linewidth]{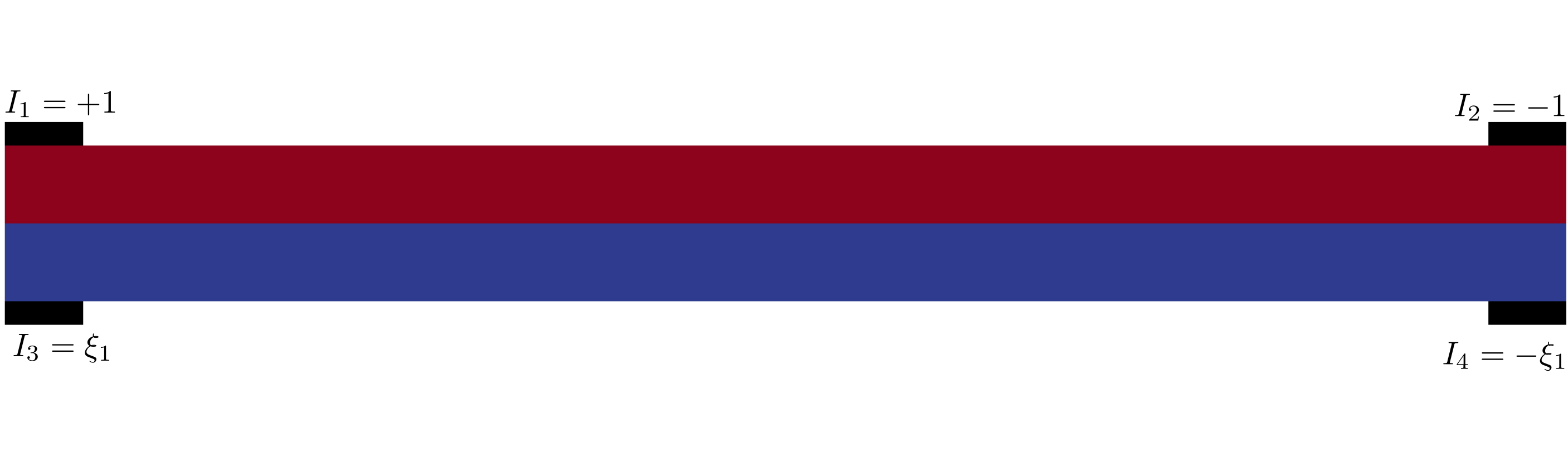}
\caption{(Example 4) Experimental configuration for EIT with two plies and four electrodes.}
\label{fig:EITsetup}
\end{figure}

\subsubsection{Numerical tests for EIT}

To evaluate the efficiency of \rASGD{LA} in solving the EIT problem, we solve four different cases using different numbers of electrodes of different lengths and positions.
In all cases, the number of experiments is $N_e=1$.
To generate the plots with the posteriors pdfs, the MAP value is approximated by the mean of the prior, i.e., $\bs{\hat{\theta}} = (\frac{\pi}{3.9375},-\frac{\pi}{3.9375})$.

\paragraph{Test case 1 (Configuration with four electrodes and one variable)} We aim to find the most informative current intensity to inject through three out of the four electrodes attached to the two-ply composite material described above and shown in Figure \ref{fig:EITsetup}.
The current at the fourth electrode is defined by Kirchhoff's law.
The electrodes are 1 cm long and have fixed positions.

We approximate the covariance of the posterior pdf for each $\bs{\xi}$ by $\bs{\Sigma}_{post}(\bs{\xi})$, as presented in \eqref{eq:sigmahat}, using the mean of the prior to approximate the MAP value.
Thus, the approximated covariances at the initial guess and the optimum solution are
\begin{equation}
\bs{\Sigma}_{post}(\bs{\xi}_0)=
\begin{bmatrix}
7.21 \times 10^{-3} & 9.73 \times 10^{-4} \\
9.73 \times 10^{-4} & 1.35 \times 10^{-4}
\end{bmatrix}, \quad
\bs{\Sigma}_{post}(\bs{\xi}^*)=
\begin{bmatrix}
5.39 \times 10^{-6} & 3.21 \times 10^{-6} \\
3.21 \times 10^{-6} & 3.39 \times 10^{-6}
\end{bmatrix}.
\label{eq:coveit}
\end{equation}
The optimization reduces the terms in the covariance matrices by two orders of magnitude, meaning that the optimized experiment provides preciser estimates of the quantities of interest.
Due to the symmetry of the problem, there are two local maxima, one with $\xi_1=-1$ and one with $\xi_1=1$.
However, the local maximum where $\xi_1=1$ is also the global maximum, with a larger expected information gain.
Therefore, we conclude that we can obtain more information about the angles of the plies from the optimized configuration than from the non-optimized configuration.

In Figure \ref{fig:EIT4ePlot}, we present the electric potential and the current streamlines both before and after the optimization.
We also present the expected information gain when using the MCLA estimator with the optimization path and the pdfs of the prior and the posteriors.
The initial guess provides less information about $\theta_1$ than about $\theta_2$.
However, the optimized position significantly reduces the variance of the $\theta_1$ estimation and provides insightful information on both parameters $\theta_1$ and $\theta_2$ with almost the same uncertainty.
\begin{figure}[H]
\centering\includegraphics[width=1\linewidth]{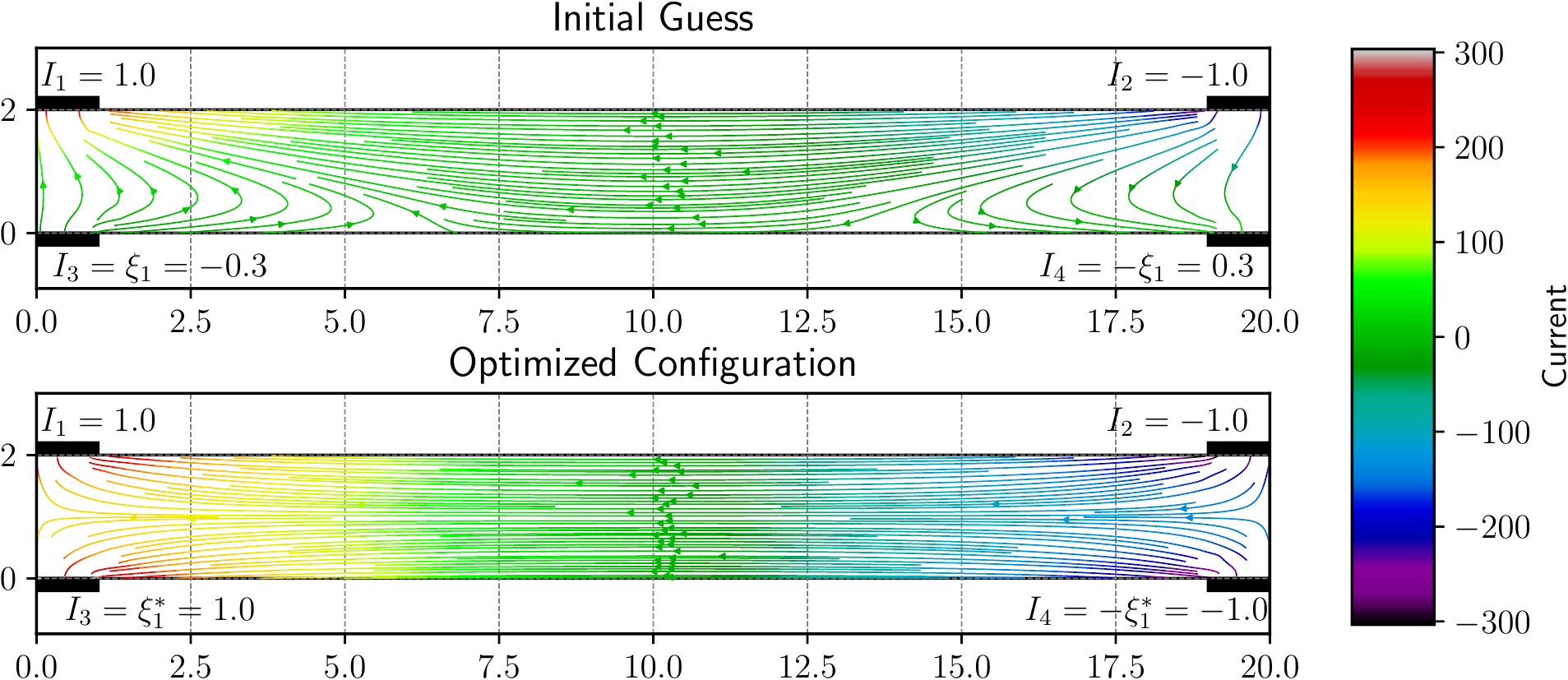}\\ \vspace{1em}
	\includegraphics[width=0.45\linewidth]{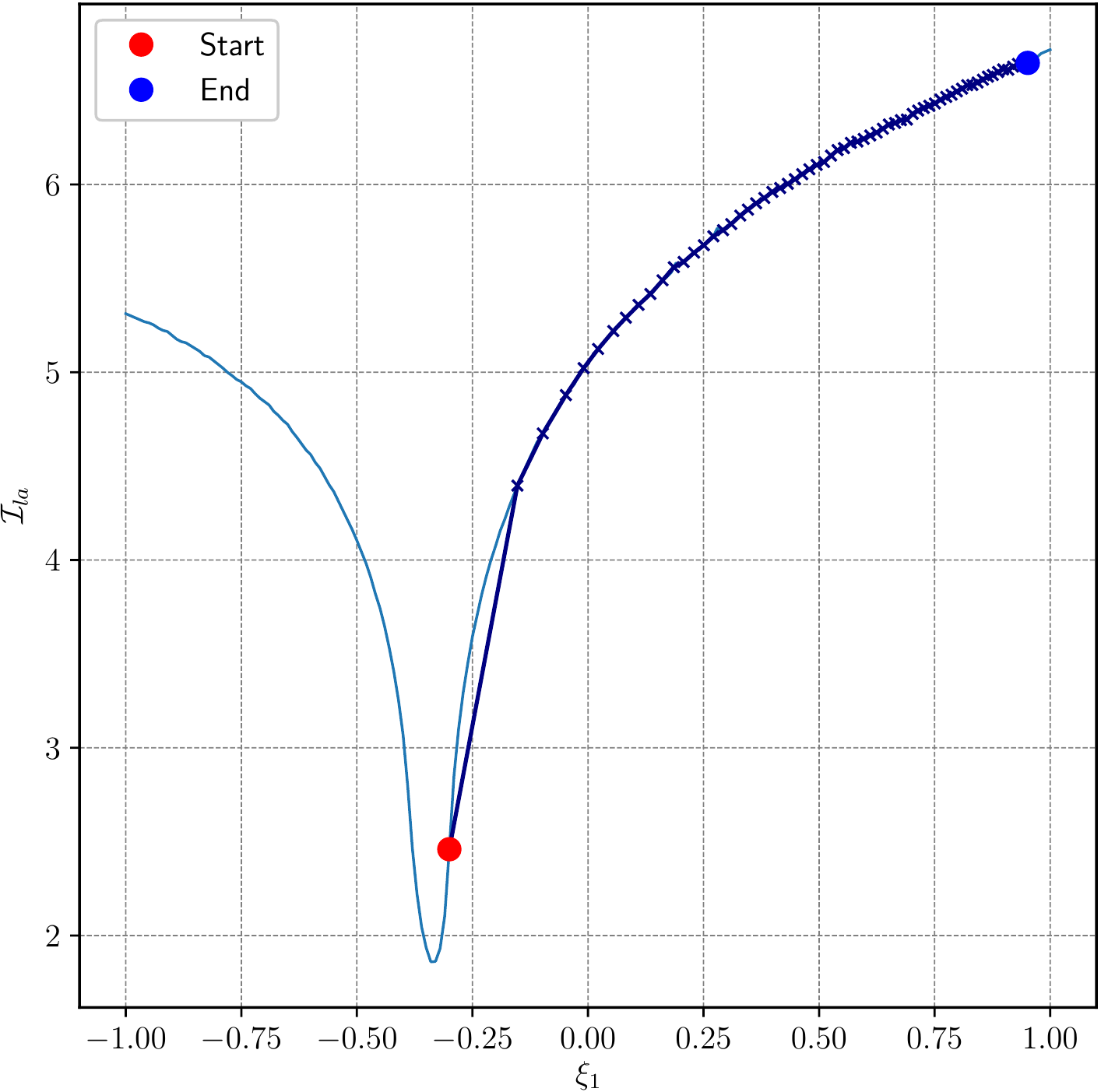}
	\includegraphics[width=0.48\linewidth]{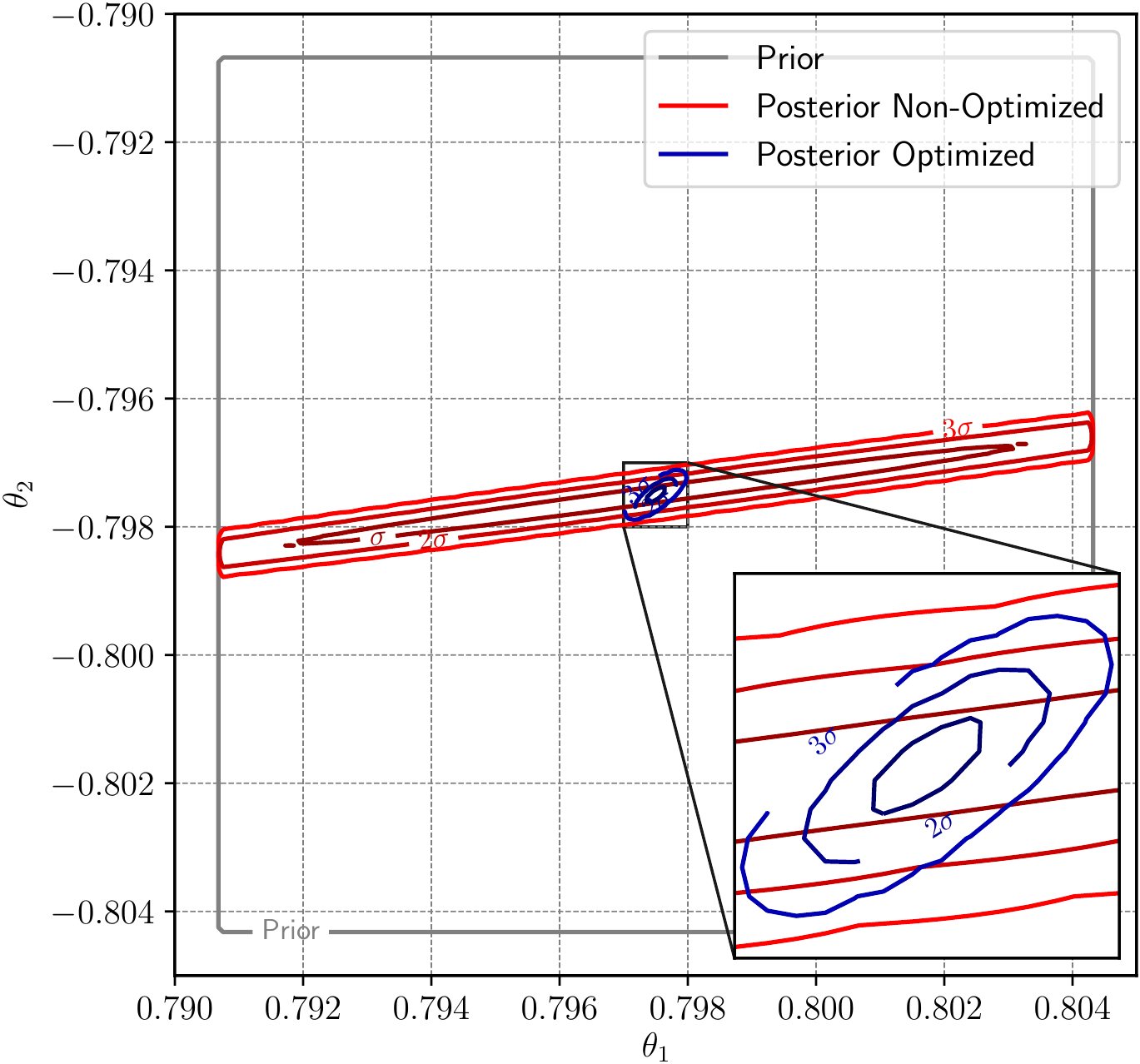}
	\caption{(Example 4, test case 1) Current streamlines, optimization path, and pdfs of both the initial and optimized configurations.}
	\label{fig:EIT4ePlot}
\end{figure}

\paragraph{Test case 2 (Configuration with three electrodes and two variables)} Here, we consider a configuration for the EIT experiment with two electrodes on the top of the two-ply composite body and one at the bottom, each $4$ cm long.
We allow the current applied to the two top electrodes to vary from $-1$ to $1$, i.e., the optimization variables are $\bs{\xi}=(I_1, I_2)$.
To impose Kirchhoff's law, the current on the third electrode (on the bottom) is set as the negative sum of the two electrodes at the top.
A constraint is imposed on $\bs{\xi}$ to guarantee that $I_3$ is between $-1$ and $1$.
To test the global convergence properties of the optimization methods, we perform optimization from two different initial guesses.
Figure \ref{fig:EIT3e1} presents the current streamlines for one of the initial guesses, $\bs{\xi} = (0.8, -0.4)$ and the posteriors from both guesses.
\begin{figure}[H]
\centering
	\includegraphics[width=0.95\linewidth]{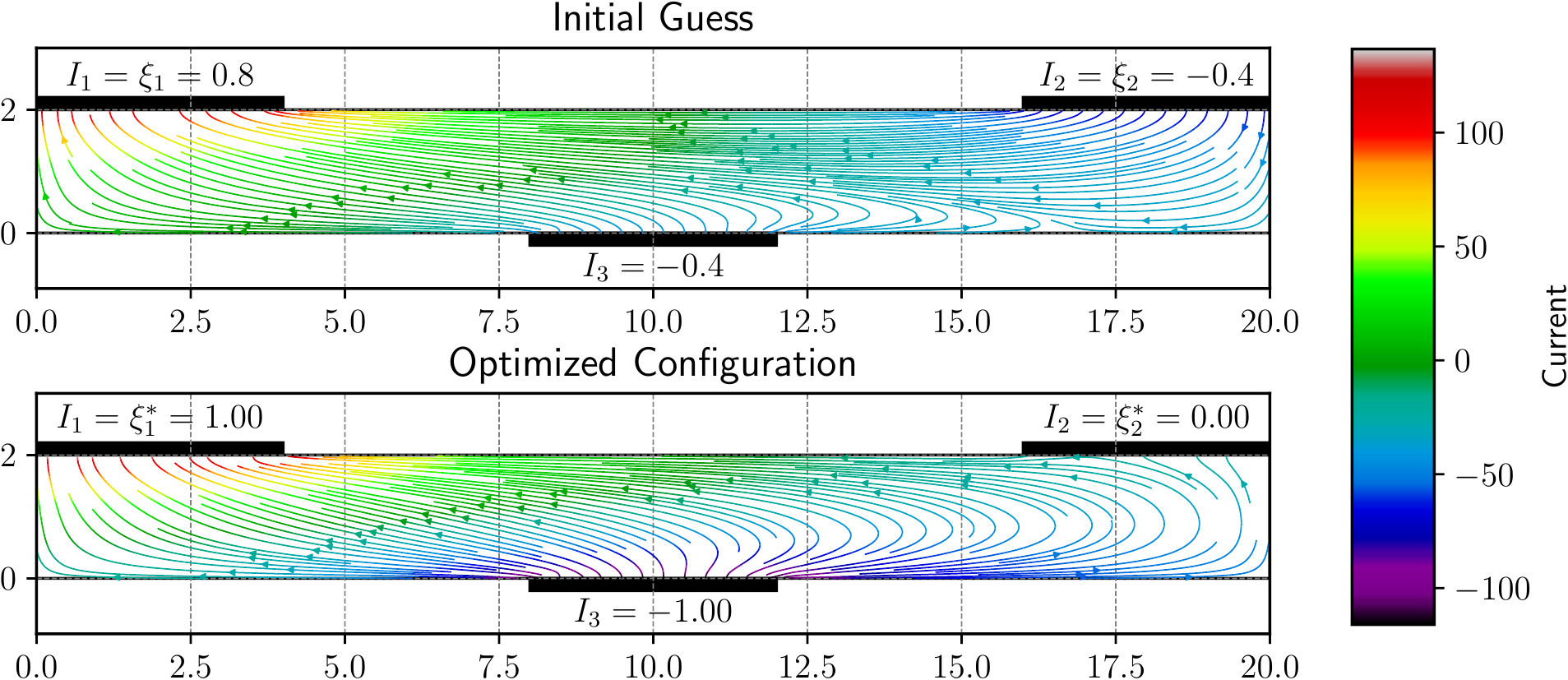}\\ \vspace{1em}
	\includegraphics[width=0.45\linewidth]{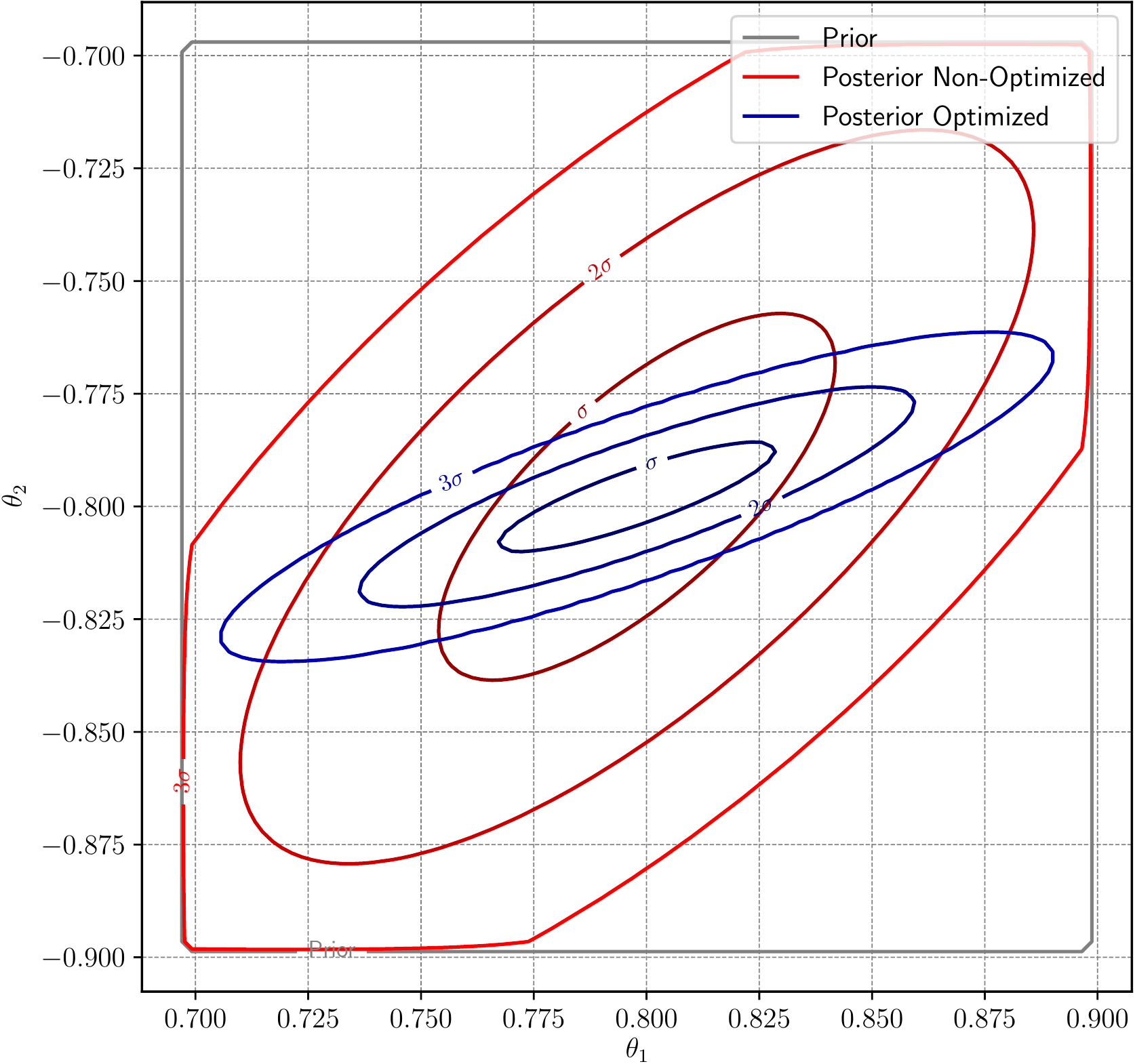}
	\includegraphics[width=0.45\linewidth]{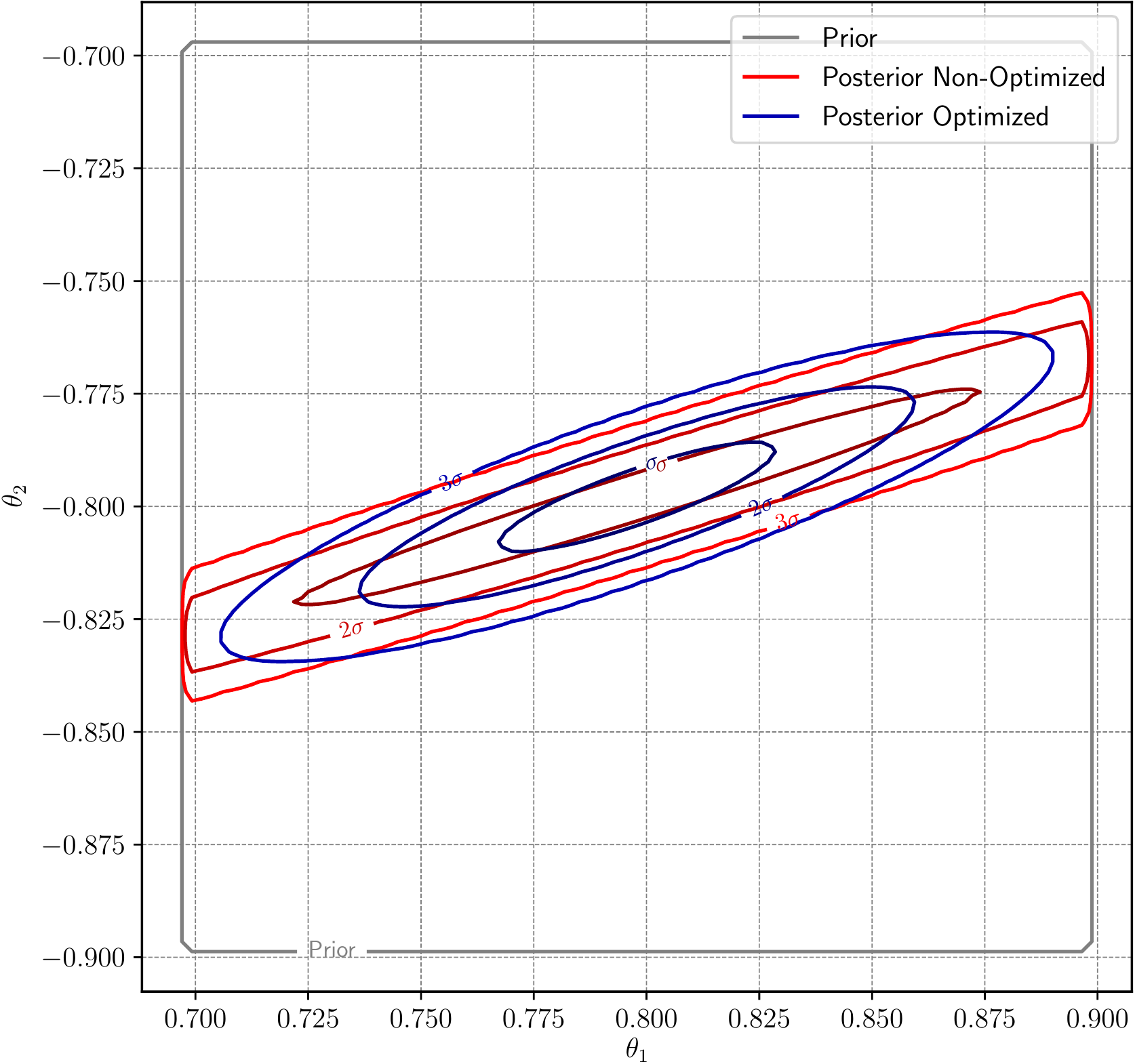}
	\caption{(Example 4, test case 2) Current streamlines for guess 2 and pdfs for both guess 1 (bottom-left) and guess 2 (bottom-right).}
	\label{fig:EIT3e1}
\end{figure}

The contour plot of the expected information gain and the ascent paths of two different initial guesses are presented in Figure \ref{fig:EIT3e}, where the infeasible regions are illustrated in blue.
The optimization is presented for the two initial guesses over the contour lines of the expected information gain.
The region shaded in gray indicates where the experiment does not provide any information gain, i.e., where $\mathcal{I}=0$.
\begin{figure}[H]
\centering
\includegraphics[width=.65\linewidth, clip, trim={0 0 0 0.5cm}]{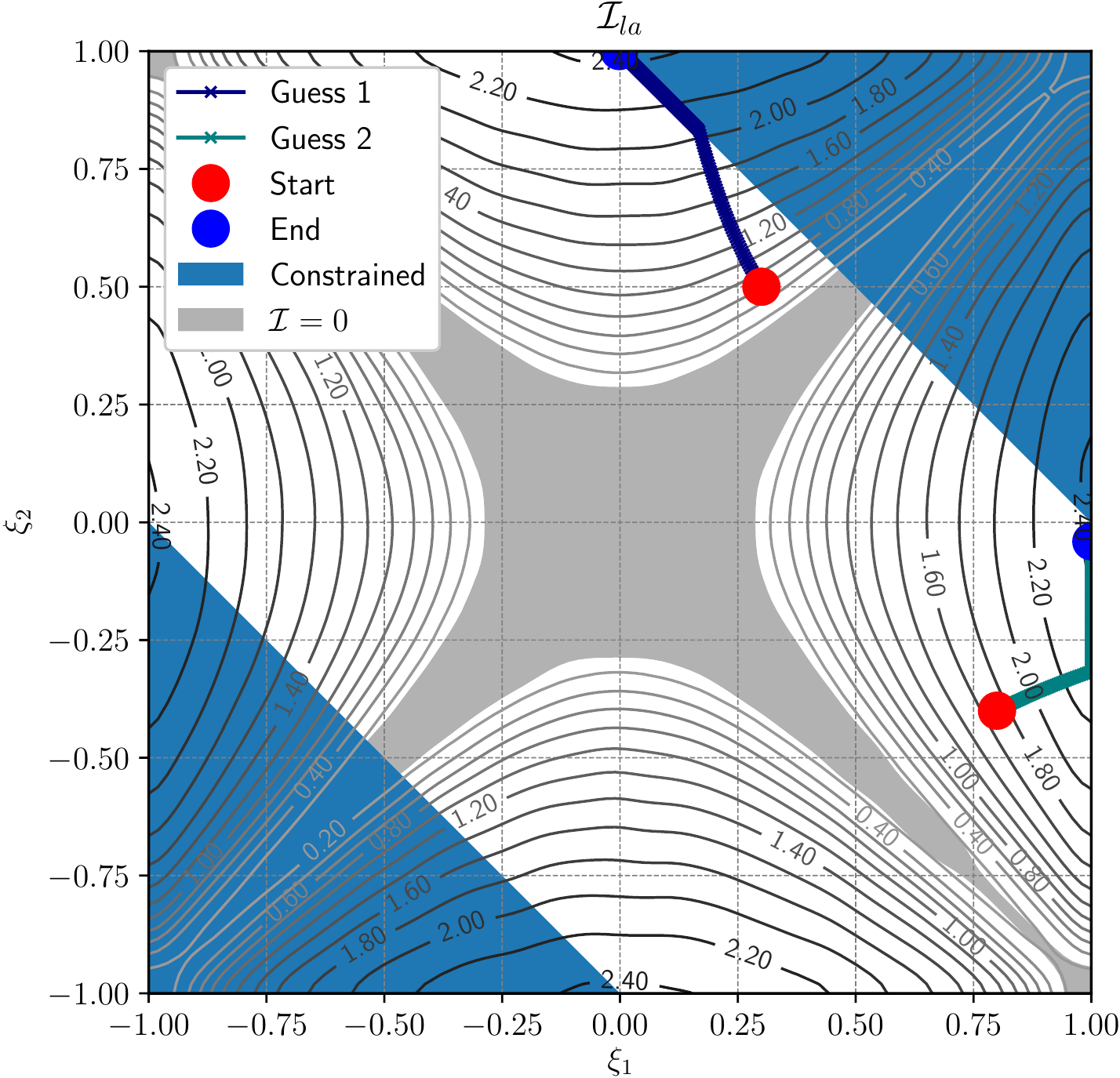}
\caption{(Example 4, test case 2) Contour of $\I{MCLA}$ with optimization paths for EIT.}
\label{fig:EIT3e}
\end{figure}

As shown in Figure \ref{fig:EIT3e}, this problem has four optima: $(0,1)$, $(1,0)$, $(0,-1)$, and $(-1,0)$.
These optima have in common the fact that one of the two top electrodes has null-current while the other two electrodes have current 1 or -1.
Figure \ref{fig:EIT3e} shows that the optimization converges to local optima for the two initial guesses, arriving at solutions where the expected information gain is around $2.4$.

This problem is symmetric in the vertical axis, as can be seen in Figure \ref{fig:EIT3e1}.
Because of this symmetry, the two optima found, $(1,0)$, and $(0,1)$ are reflections of one another over the symmetry axis, the reason why the two optimized posteriors look alike.
Moreover, this symmetry results in the diagonal symmetry of the expected information gain that can be observed in Figure \ref{fig:EIT3e}.

\paragraph{Test case 3 (Configuration with ten electrodes and ten variables)} We now consider a more complex EIT experiment with ten $2$ cm long electrodes.
The intensity of the initial current applied is $0.5$ at the inlet electrodes (on top of the two-ply composite body) and $-0.5$ for the outlet electrodes (on the bottom).

The current streamlines, before and after the optimization, are depicted at the top of Figure \ref{fig:EIT10e}.
The optimization converges to a setup with both positive and negative currents applied on both the top and the bottom electrodes.
This optimal setup provides an expected information gain of 7.18.
For the sake of comparison, the expected information gain from the setup with currents of 1.0 and -1.0 applied to the top and bottom electrodes, respectively, is only 2.95.
On the bottom left of Figure \ref{fig:EIT10e}, the posteriors show that the variance of the quantities of interest for the optimized configuration is remarkably smaller than for the initial guess.
On the bottom right of the figure, we present the self-convergence test where we see that using Nesterov's acceleration resulted in an accelerated convergence of the optimizer.
\begin{figure}[H]
\centering
	\includegraphics[width=\linewidth]{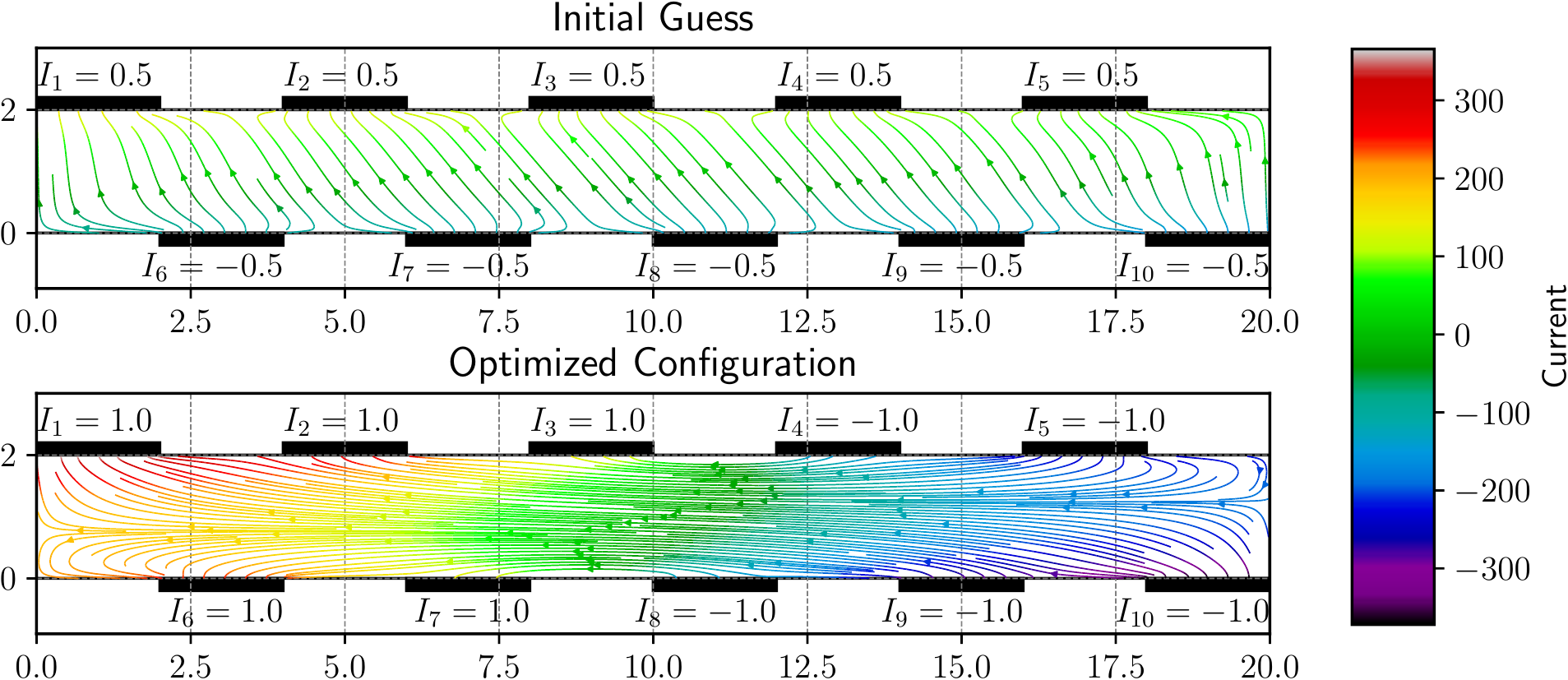} \\ \vspace{1em}
	\includegraphics[width=.48\linewidth]{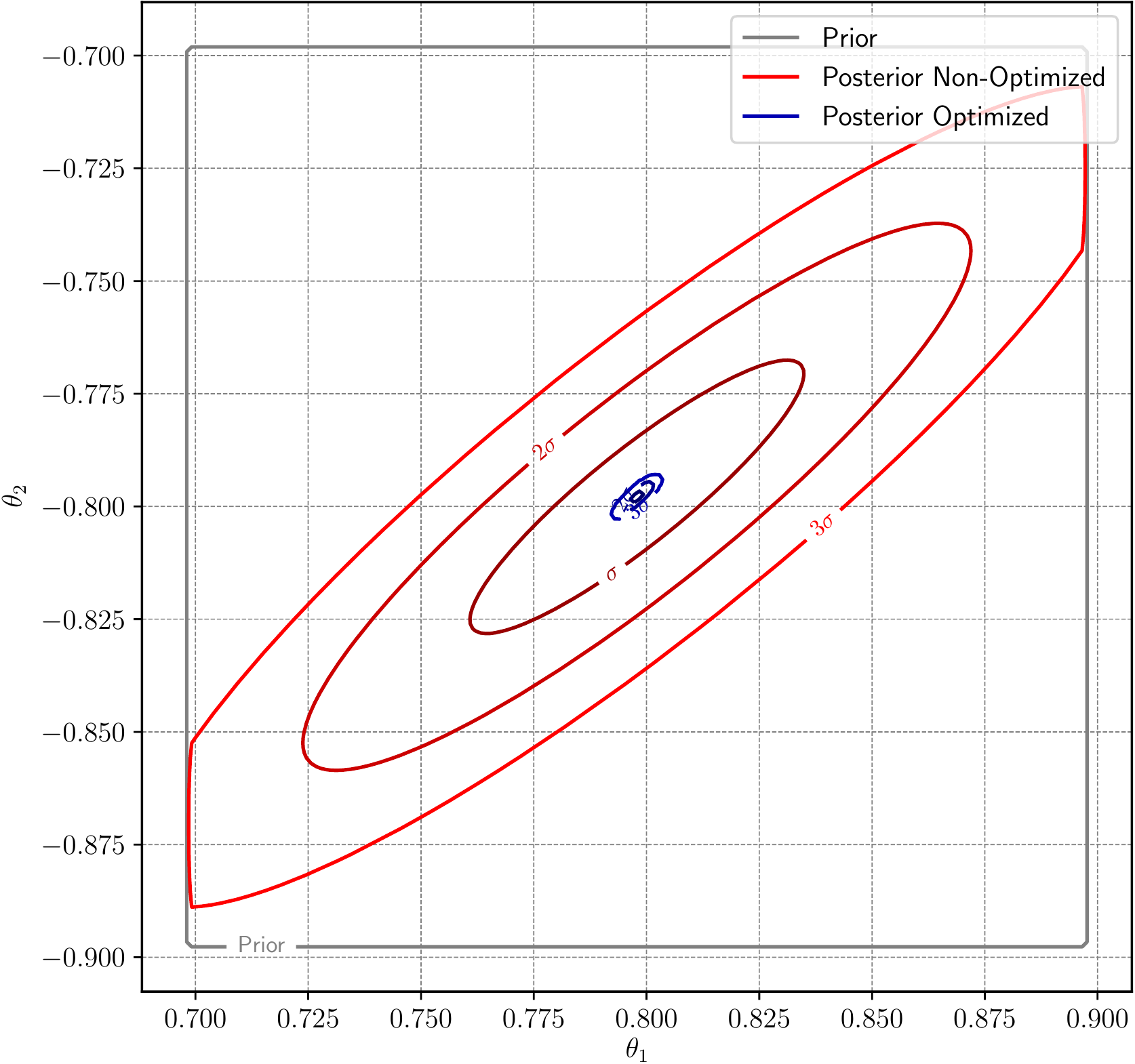} \hspace{0.5em}
	\includegraphics[width=.48\linewidth]{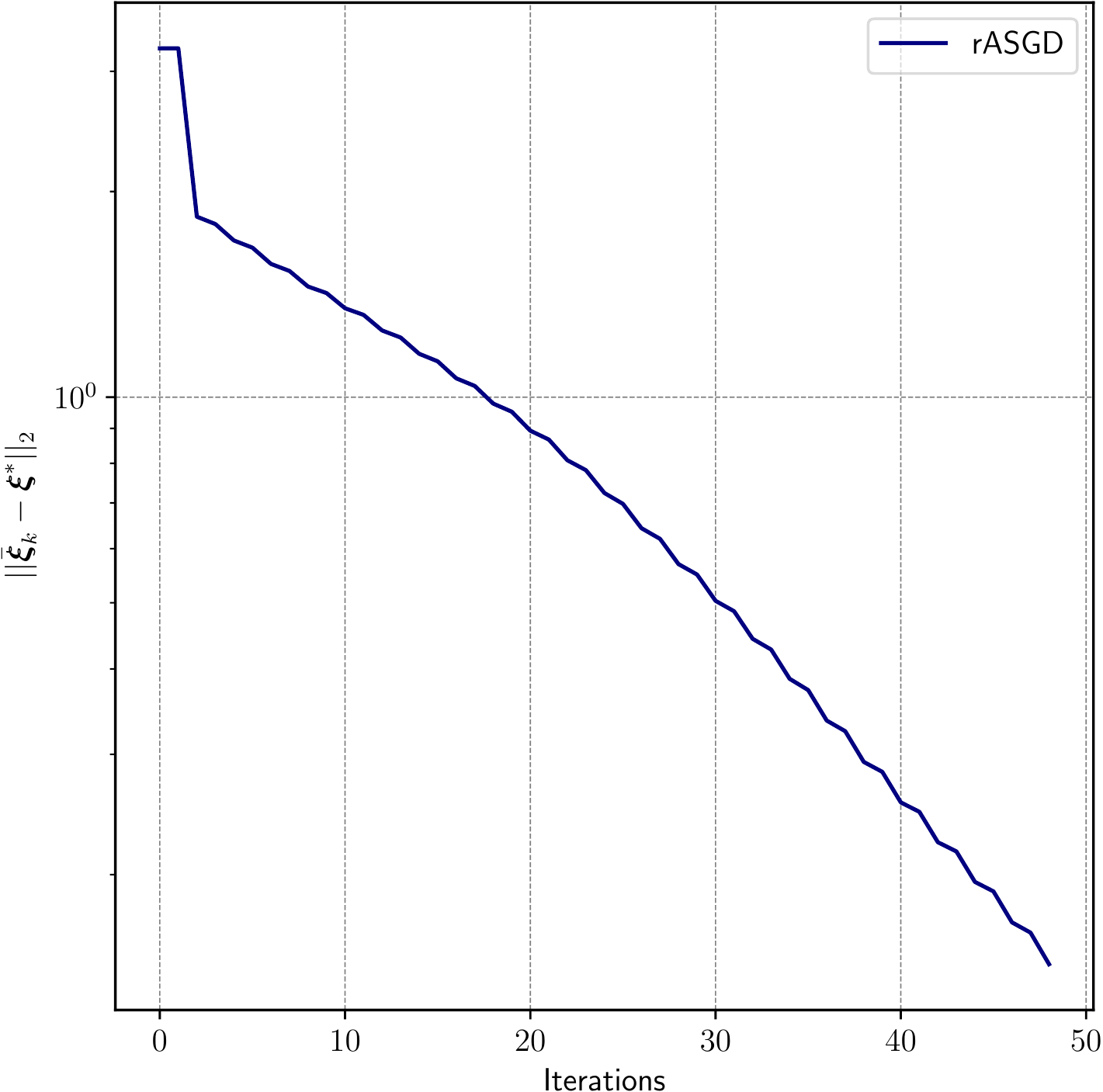}
	\caption{(Example 4, test case 3) Current streamlines, pdfs of initial and optimized configurations, and self-convergence to the optimum.}
	\label{fig:EIT10e}
\end{figure}

The expected information gains for all of the four cases presented in Example 4 are listed in Table \ref{tab:ExpInfGainEx4}.
\begin{table}[H]
\centering
\caption{Expected information gain using MCLA with $N=1000$ in Example 4.}
\begin{tabular}{ccc}
	\toprule
	                & Initial Guess & Optimized \\ \midrule
	    Case 1      &     2.26      &   6.72    \\
	Case 2, Guess 1 &     0.64      &   2.46    \\
	Case 2, Guess 2 &     1.74      &   2.47    \\
	    Case 3      &     1.57      &   7.18    \\ \bottomrule
\end{tabular}
\label{tab:ExpInfGainEx4}
\end{table}

\section*{Conclusion}

In this work, we couple the Nesterov-based accelerated stochastic gradient with momentum-restart and Laplace-based methods in order to solve Bayesian optimal experimental design problems.
For the gradient estimator, we use two strategies, a Laplace approximation and a Monte Carlo method with Laplace-based importance sampling, to approximate the solution of the inner integral that appears in the expectation of the Shannon information gain.
Moreover, we derive the explicit formula for the gradient of the EIG when using the Laplace approximation and for the double-loop Monte Carlo estimator with Laplace-based importance sampling.
The stochastic gradient with the Laplace approximation estimator (\SG{LA}) leads to convergence to the optimum in the examples where it is tested; it is also significantly cheaper than the other gradient estimators.
We observe that the bias introduced by the Laplace approximation is not relevant for the problems solved here.
The stochastic gradient with Monte Carlo importance sampling (\SG{MCIS}) estimator is more expensive than \SG{LA}, but less costly than stochastic gradient with Monte Carlo (\SG{MC}).
However, \SG{MCIS} has the advantage of being a consistent estimator, whereas \SG{LA} is not.
Moreover, the \SG{LA} and \SG{MCIS} estimators do not experience \emph{numerical underflow}, unlike the double-loop Monte Carlo estimator (DLMC) or its stochastic gradient, \SG{MC}.

Nesterov's acceleration and the restart technique improve the convergence, in comparison with simple steepest descent using the stochastic gradient estimators.
Our accelerated stochastic gradient descent (ASGD) with the restart technique (rASGD) efficiently solves stochastic optimization problems, even without the use of variance reduction techniques.

We analyze two benchmark problems based on benchmark analytical functions, one of them based on OED; and two common problems found in engineering.
The two benchmark problems are used to assess the efficiency of the optimization methods, as well as the \SG{LA} and \SG{MCIS} estimators.
The rASGD algorithm combined with \SG{LA}, \rASGD{LA}, performs better than the other methods in the second example; thus we opt to use it on the two engineering problems.
The first engineering problem is to determine the optimal positioning of strain gauges on a beam in order to accurately measure the beam's mechanical properties.
The second engineering problem is finding the optimal currents to be applied to electrodes during an electrical impedance tomography experiment, in order to measure the orientation of the plies in a composite laminate material, using the complete electrode model.
In the engineering examples, \rASGD{LA} performs well in terms of its ability to solve OED problems.
Since we use the \SG{LA} estimator for the two engineering examples, we use the DLMCIS estimator to determine whether the biases of the gradients in the optima found are relevant.
Our numerical tests show that the biased optima are sufficiently close of the real optima for the desired precision.
In situations where the biased optimum is not sufficiently close to the real optimum, we suggest a two-phase optimization, first with \SG{LA}, and second with \SG{MCIS}, to correct the bias.

In future work, we plan on using mini-batches and other variance reduction techniques to address problems where the variance of the stochastic gradient estimators is large or the admissible error is considerably small.

\section*{Acknowledgments}

The research reported in this publication was supported by funding from King Abdullah University of Science and Technology (KAUST), KAUST CRG3 Award Ref:2281, and KAUST CRG4 Award Ref:2584.
The authors also gratefully acknowledge the financial support of CNPq (National Counsel of Technological and Scientific Development) and CAPES (Coordination of Superior Level Staff Improvement).

% \section*{References}
\bibliography{OptExpDesign}

\begin{thebibliography}{10}
\expandafter\ifx\csname url\endcsname\relax
  \def\url#1{\texttt{#1}}\fi
\expandafter\ifx\csname urlprefix\endcsname\relax\def\urlprefix{URL }\fi
\expandafter\ifx\csname href\endcsname\relax
  \def\href#1#2{#2} \def\path#1{#1}\fi

\bibitem{chaloner1995bayesian}
K.~Chaloner, I.~Verdinelli, {B}ayesian experimental design: A review,
  Statistical Science (1995) 273--304.

\bibitem{ryan2003estimating}
K.~J. Ryan, Estimating expected information gains for experimental designs with
  application to the random fatigue-limit model, Journal of Computational and
  Graphical Statistics 12~(3) (2003) 585--603.

\bibitem{huan2010accelerated}
X.~Huan, Accelerated bayesian experimental design for chemical kinetic models,
  Ph.D. thesis, Massachusetts Institute of Technology (2010).

\bibitem{huan2013simulation}
X.~Huan, Y.~M. Marzouk, Simulation-based optimal {B}ayesian experimental design
  for nonlinear systems, Journal of Computational Physics 232~(1) (2013)
  288--317.

\bibitem{spall1988stochastic}
J.~C. Spall, A stochastic approximation algorithm for large-dimensional systems
  in the {K}iefer--{W}olfowitz setting, in: Decision and Control, 1988.,
  Proceedings of the 27th IEEE Conference on, IEEE, 1988, pp. 1544--1548.

\bibitem{long2013fast}
Q.~Long, M.~Scavino, R.~Tempone, S.~Wang, Fast estimation of expected
  information gains for {B}ayesian experimental designs based on {L}aplace
  approximations, Computer Methods in Applied Mechanics and Engineering 259
  (2013) 24--39.

\bibitem{huan2014gradient}
X.~Huan, Y.~Marzouk, Gradient-based stochastic optimization methods in
  {B}ayesian experimental design, International Journal for Uncertainty
  Quantification 4~(6) (2014) 1--41.

\bibitem{Beck2017fast}
J.~{Beck}, B.~M. {Dia}, L.~F.~R. {Espath}, Q.~{Long}, R.~{Tempone}, {Fast
  {B}ayesian experimental design: {L}aplace-based importance sampling for the
  expected information gain}, Computer Methods in Applied Mechanics and
  Engineering 334 (2018) 523--553.

\bibitem{nemirovski2005efficient}
A.~Nemirovski, Efficient methods in convex programming, Technion, 2005.

\bibitem{robbins1951stochastic}
H.~Robbins, S.~Monro, A stochastic approximation method, The Annals of
  Mathematical Statistics (1951) 400--407.

\bibitem{nesterov1983method}
Y.~Nesterov, A method of solving a convex programming problem with convergence
  rate o (1/k2), in: Soviet Mathematics Doklady, Vol.~27, 1983, pp. 372--376.

\bibitem{o2015adaptive}
B.~O'Donoghue, E.~Cand\`es, Adaptive restart for accelerated gradient schemes,
  Foundations of Computational Mathematics 15~(3) (2015) 715--732.

\bibitem{nitanda2016accelerated}
A.~Nitanda, Accelerated stochastic gradient descent for minimizing finite sums,
  in: Artificial Intelligence and Statistics, 2016, pp. 195--203.

\bibitem{shannon1948mathematical}
C.~E. Shannon, A mathematical theory of communication, Bell Syst. Tech. J. 27
  (1948) 623--656.

\bibitem{nelder1965simplex}
J.~A. Nelder, R.~Mead, A simplex method for function minimization, The Computer
  Journal 7~(4) (1965) 308--313.

\bibitem{kiefer1959optimum}
J.~Kiefer, J.~Wolfowitz, Optimum designs in regression problems, The Annals of
  Mathematical Statistics (1959) 271--294.

\bibitem{lai1979adaptive}
T.~L. Lai, H.~Robbins, Adaptive design and stochastic approximation, The Annals
  of Statistics (1979) 1196--1221.

\bibitem{polyak1992acceleration}
B.~T. Polyak, A.~B. Juditsky, Acceleration of stochastic approximation by
  averaging, SIAM Journal on Control and Optimization 30~(4) (1992) 838--855.

\bibitem{cotter2011better}
A.~Cotter, O.~Shamir, N.~Srebro, K.~Sridharan, Better mini-batch algorithms via
  accelerated gradient methods, in: Advances in Neural Information Processing
  Systems, 2011, pp. 1647--1655.

\bibitem{nesterov2013introductory}
Y.~Nesterov, Introductory lectures on convex optimization: A basic course,
  Vol.~87, Springer Science \& Business Media, 2013.

\bibitem{rumelhart1986learning}
D.~E. Rumelhart, G.~E. Hinton, R.~J. Williams, Learning representations by
  back-propagating errors, {N}ature 323~(6088) (1986) 533.

\bibitem{johnson2013accelerating}
R.~Johnson, T.~Zhang, Accelerating stochastic gradient descent using predictive
  variance reduction, in: Advances in Neural Information Processing Systems,
  2013, pp. 315--323.

\bibitem{allen2017katyusha}
Z.~Allen-Zhu, Katyusha: The first direct acceleration of stochastic gradient
  methods, in: Proceedings of the 49th Annual ACM SIGACT Symposium on Theory of
  Computing, ACM, 2017, pp. 1200--1205.

\bibitem{su2016differential}
W.~Su, S.~Boyd, E.~J. Cand\`es, A differential equation for modeling
  {N}esterov's accelerated gradient method: theory and insights, Journal of
  Machine Learning Research 17~(153) (2016) 1--43.

\bibitem{timoshenko1921lxvi}
S.~P. Timoshenko, {LXVI}. {O}n the correction for shear of the differential
  equation for transverse vibrations of prismatic bars, The London, Edinburgh,
  and Dublin Philosophical Magazine and Journal of Science 41~(245) (1921)
  744--746.

\bibitem{somersalo}
E.~Somersalo, M.~Cheney, D.~Isaacson., Existence and uniqueness for electrode
  models for electric current computed tomography, SIAM J. Appl. Math, 52
  (1992) 1023--1040.

\end{thebibliography}

\end{document}